\documentclass[a4paper,reqno,11pt]{amsart}

\usepackage[left=1 in, right=1 in,top=1 in, bottom=1 in]{geometry}

\usepackage{amssymb,amsmath,amsfonts,amsthm}
\usepackage{mathrsfs}
\usepackage{color}
\allowdisplaybreaks

\newcommand{\beq}{\begin{equation}}
\newcommand{\eeq}{\end{equation}}
\renewcommand{\(}{\left(}
 \renewcommand{\)}{\right)}
\renewcommand{\[}{\left[}
 \renewcommand{\]}{\right]}

\def\eps{\varepsilon}
\def\sig{\sigma}
\newcommand{\n}{{\bf n}}
\newcommand{\D}{\partial^{\varphi}}

\newcommand{\V}{\mathcal{V}_{t}}

\newcommand{\X}{{\mathbb H}}

\newcommand{\Y}{{\mathbb W}}
\newcommand{\es}{{\eps,\sig}}
\newcommand{\N}{{\bf N}}

\newcommand{\pa}{\partial}

\providecommand{\abs}[1]{\left\vert#1\right\vert}
\providecommand{\norm}[1]{\left\Vert#1\right\Vert}

\providecommand{\Rn}[1]{\mathbb{R}^{#1}}

\def\dt{\partial_t}
\def\dtt{ \frac{d}{dt}}
\def\hal{\frac{1}{2}}
\def\thal{\frac{3}{2}}
\def\fhal{\frac{5}{2}}
\def\ls{\lesssim}

\def\linf{\Lambda_\infty}

\newtheorem{prop}{Proposition}
\newtheorem{theoreme}[prop]{Theorem}
\newtheorem{lem}[prop]{Lemma}

\newtheorem{rem}[prop]{Remark}

\numberwithin{equation}{section}
\numberwithin{prop}{section}

\title[Incompressible viscous surface waves]{Vanishing viscosity and surface tension limits of incompressible viscous surface waves}

\author{Yanjin Wang}

\address{School of Mathematical Sciences\\
Xiamen University\\
Xiamen, Fujian 361005, China
\newline \indent and
\newline \indent The Institute of Mathematical Sciences\\
The Chinese University of Hong Kong\\
Shatin, NT, Hong Kong}
\email[Y. J. Wang]{yanjin$\_$wang@xmu.edu.cn}
\thanks{Y. J. Wang was supported by the National Natural Science Foundation of China (11771360, 11531010) and the Natural Science Foundation of Fujian Province of China (2019J02003).}

\author{Zhouping Xin}

\address{The Institute of Mathematical Sciences\\
The Chinese University of Hong Kong\\
Shatin, NT, Hong Kong}
\email[Z. P. Xin]{zpxin@ims.cuhk.edu.hk}

\thanks{Z. P. Xin was partially supported by the Zheng Ge Ru Foundation, and Hong Kong RGC Earmarked Research Grants
CUHK-4048/13P, CUHK-14305315, CUHK-14302819, CUHK-14300917,  CUHK-14302917,
a Focus Area Grant from CUHK, a grant from Croucher Foundation, and a NSFC/RGC Joint Research Scheme.}



\begin{document}

\begin{abstract}
Consider the dynamics of a layer of viscous incompressible fluid under the influence of gravity. The upper boundary is a free boundary with the effect of surface tension taken into account, and the lower boundary is a fixed boundary on which the Navier-slip condition is imposed. It is proved that there is a uniform time interval on which the estimates independent of both viscosity and surface tension coefficients of the solution can be established.  This then allows one to justify the vanishing viscosity and surface tension limits by the strong compactness argument. In the presence of surface tension, the main difficulty lies in the less regularity of the highest temporal derivative of the mean curvature of the free surface and the pressure. It seems hard to overcome this difficulty by using the vorticity in viscous boundary layers.  One of the key observations here is to find that there is a crucial cancelation between the mean curvature and the pressure by using the dynamic boundary condition.
\end{abstract}

\maketitle


\section{Introduction}

\subsection{Formulation}

We consider the motion of an incompressible viscous fluid under the influence of a uniform gravitational force in a moving domain
\beq\label{omegat}
\Omega(t)= \left\{ x \in \mathbb{R}^3 \mid  -b< x_{3} < h(t, x_{1}, x_{2}) \right\}.
\eeq
The lower boundary of $\Omega(t)$ is assumed to be rigid and given with the constant $b>0$, but the upper boundary is a free surface that is the
graph of the unknown function $ h:\mathbb{R}_+\times\Rn{2}\to \Rn{}$. The fluid is described by its velocity and pressure, which are given for
each
$t\ge0$ by $ u(t,\cdot):\Omega(t) \to \Rn{3} $ and $ p(t,\cdot):\Omega(t) \to \Rn{}$, respectively.  For each $t>0$, $(u, p, h)$ satisfy the free-surface incompressible Navier-Stokes equations
\beq\label{NS}
\begin{cases}
\partial_t u + u \cdot \nabla u + \nabla p-\eps  \Delta u  = 0& \text{in }
\Omega(t) \\
\nabla\cdot{u}=0 & \text{in }\Omega(t) \\
p {\bf n}-  2  \eps  S u {\bf n} =  {g} h {\bf n}-\sigma H {\bf n} & \text{on } \{x_3=h(t,x_1,x_2)\} \\
\partial_t h = u\cdot\N &
\text{on }\{x_3=h(t,x_1,x_2)\} \\
u_3 = 0,\quad {\varepsilon}(Su (- e_3))_i=-\kappa {\varepsilon} u_i,\ i=1,2 & \text{on } \{x_3=-b\}%
\end{cases}%
\eeq
for $Su = {1 \over 2} \( \nabla u + \nabla u^t\)$ the symmetric part of the gradient of $u$ and ${\bf n}=\N/|\N|$ the outward unit normal of the
free surface with $\N=(-\pa_1h,-\pa_2h,1)^t$. $\eps>0$ is the viscosity, $g>0$ is the strength of gravity, $\sigma>0$ is the surface tension
coefficient and $\kappa$ is the friction coefficient. Finally, $H$ is twice the mean curvature of the free surface given by the formula
\begin{equation}
H =\nabla\cdot\left(\frac{\nabla h }
{\sqrt{1+|\nabla h|^2}}\right).
\end{equation}
The kinematic boundary condition, the fourth equation in \eqref{NS}, implies that the free surface is adverted with the
fluid, and the dynamic boundary condition, the third equation, states the balance of normal stress  on the free surface. Note that in \eqref{NS} we have shifted the gravitational forcing to the boundary and
eliminated the constant atmospheric pressure, $p_{atm}$, in the usual way by adjusting the actual pressure $\bar{p}$ according to $p = \bar{p} +
g x_3 - p_{atm}$. We have imposed the
Navier
slip boundary condition on the fixed lower boundary.  {Note that setting formally $\varepsilon=0$ in \eqref{NS} it corresponds to  the free-surface incompressible Euler equations.}

The initial surface is given by the graph of the function $h(0)=h_0: \Rn{2}\rightarrow \mathbb{R}$, which yields the initial domain $\Omega(0)$
on which the initial velocity $u(0)=u_0: \Omega(0) \rightarrow \mathbb{R}^3$ is specified. It  will  be assumed that $h_0 > -b$ and that $(u_0,
h_0)$ satisfy certain compatibility conditions.

The movement of the free boundary and the subsequent change of the domain create numerous mathematical difficulties. To circumvent these, as
usual, we will transform the free boundary problem under consideration to a problem with a fixed domain and fixed boundary. Consider a  family of
diffeomorphism $\Phi(t, \cdot)$ of the form
\beq\label{diff}
\begin{split}
\Phi(t, \cdot) : \,   \Omega= \mathbb{R}^2 \times (-b, 0) & \rightarrow   \Omega({t}) \\
(y,z) &\rightarrow    (y, \varphi(t, y, z)).
\end{split}
\eeq
$\varphi$ is chosen so that $\partial_{z} \varphi>0$ which ensures that $\Phi(t, \cdot)$ is a diffeomorphism.
For the fluid domain under consideration, $\varphi$ can be chosen as
\beq\label{eqphi}
\varphi(t,y,z)=    z + \eta(t, y, z),
\eeq
where $\eta$ is a chosen extension of $h$ onto $\{z\le 0\}$ defined by
\beq\label{eqeta}
\hat{\eta}(t, \xi, z)= \( 1+\frac{z}{b}\)\exp{(A|\xi|z)} \hat{h}(t, \xi).
\eeq
Here $\hat{\cdot}$ stands for the horizontal Fourier transform with respect  to the $y$ variable. It is verified in Proposition \ref{diffeoprop}
that for given $h_0$ if the  number $A>0$ is chosen sufficiently small, then
\beq\label{Adeb}
\partial_{z} \varphi(0,y,z)\ge c_0 >0\text{ in }\Omega.
\eeq

Then one can reduce the problem into the fixed domain $ \Omega  $ by setting
\beq\label{vdef}
v(t,y,z)= u(t, \Phi(t,y,z)), \ q(t,y,z)= p(t, \Phi(t,y,z) )\text{ in }\Omega.
\eeq
Set
$$ \partial^\varphi_{i} =  \partial_{i}  - { \partial_{i} \varphi \over \partial_{z} \varphi}
\partial_{z}, \quad i=t, \, 1,\, 2, \quad \D_{3}= \D_{z}=  {1 \over \partial_{z} \varphi} \partial_{z} $$
such that
$$   \partial_{i}u \circ \Phi(t, \cdot)= \D_{i}v, \quad i= t, \, 1, \, 2, \, 3.
     $$
Then by the change of coordinates \eqref{diff}, the problem \eqref{NS} becomes
\beq\label{NSv}
\begin{cases}   \D_{t}v +  v \cdot \nabla^\varphi  v + \nabla^\varphi q - \eps   \Delta^\varphi v =0 &\text{in }\Omega
\\ \nabla^\varphi  \cdot  v = 0  &\text{in }\Omega
\\q \n - 2 \eps S^\varphi v   \n = gh  \n-\sigma H\n &\text{on }\{z=0\}
\\\partial_{t}h = v \cdot \N &\text{on }\{z=0\}
\\v_3 = 0,\quad (S^\varphi v   e_3 )_i=  \kappa v_i,\ i=1,2 &\text{on }\{z=-b\}.
\end{cases}
\eeq
Here we have naturally written $ \(\nabla^\varphi  \)_i =  \D_{i}  , \ \Delta^\varphi=  \D_{i}\D_{i} , \  \nabla^\varphi\cdot  v=  \D_{i} v_{i} $
and $S^\varphi v= {1 \over 2}\big( \nabla^\varphi v + (\nabla^\varphi v)^t)$.
Note that $\nabla^\varphi\cdot  S^\varphi v= {1 \over 2} \Delta^\varphi v$ for vector fields satisfying $\nabla^\varphi\cdot  v=0$.

\subsection{Previous works}

Free boundary problems in fluid mechanics have been studied intensively in the mathematical community. There are a huge amount of mathematical
works, and we only mention briefly some of them below. We may refer to the references cited in these works for more proper survey of the
literature. For the incompressible Navier-Stokes equations, we refer to, for instance, Beale \cite{beale_1}, Hataya \cite{H}, Guo and Tice
\cite{GT13,GT_per,GT_inf} for the well-posedness without surface tension, and Beale \cite{beale_2}, Tani \cite{Ta}, Tanaka and Tani \cite{TT} for
the well-posedness with surface tension. Those well-posedness results are strongly based on the regularizing effect of the viscosity, and the
solutions are shown to be global for the small initial data \cite{beale_2,TT,H,GT_per,GT_inf}. Note that the surface tension has a
regularizing effect on the free surface, and it enhances the decay rate, see \cite{GT_per,GT_inf} for more discussions. For the
incompressible Euler equations, the problem becomes much more difficult. The early works were focused on the irrotational fluids, which began
with
the pioneering work of Nalimov \cite{N} of the local well-posedness without surface tension for the small initial data and was generalized to the
general initial data by the breakthrough of Wu \cite{Wu1,Wu2} for the case without surface tension and by Beyer and G\"unther \cite{BG} for the
case with surface tension. For the irrotational inviscid fluids, certain dispersive effects can be used to establish the global well-posedness
for the small initial data; we refer to  Wu \cite{Wu3,Wu4}, Germain, Masmoudi and Shatah \cite{GMS1}, Ionescu and Pusateri \cite{IP} and Alazard
and Delort \cite{AD} for the case with gravity but without surface tension, Germain, Masmoudi and Shatah \cite{GMS2} and Ionescu and Pusateri \cite{IP2} for
the case with surface tension but without gravity, and Deng,  Ionescu,    Pausader and  Pusateri \cite{DIPP} for
the case with both gravity and surface tension. For the general incompressible Euler equations without the irrotational assumption, only local well-posedness
results could be found. The first local well-posedness in 3D was obtained by Lindblad \cite{Lindblad05} for the case without surface tension and
by Coutand and Shkoller
\cite{CS07} for the case with (and without) surface tension, and we also refer to Shatah and Zeng \cite{SZ} and Zhang and Zhang
\cite{ZZ}.

Various approaches are used to prove those well-posedness results mentioned above, depending on whether viscosity or surface tension is presented
or not. It is then very natural and interesting to study the asymptotic behavior of vanishing these two parameters in the equations. The
vanishing viscosity limit for the Navier-Stokes equations is a classical issue. When there is no boundary, the problem has been well studied; we
refer to Swann \cite{Swann71}, Kato \cite{Kato72}, DiPerna and Majda \cite{DM1,DM2}, Constantin \cite{Con86} and Masmoudi \cite{M} for example.
However, in the presence of boundaries, the situation is more complicated and the problem becomes challenging due to the possible formation of
boundary layers. In a fixed domain with the no-slip boundary condition, there is formation of boundary layers in the vicinity of the boundary and
the solution $u^\eps$ of the Navier-Stokes equations is expected to behavior like $u^\eps\sim u^0+U(t,y,z/\sqrt{\eps})$ (we assume the boundary
is locally given by $z=0$), where $u^0$ is the solution of the Euler equations satisfying only the impermeable boundary condition and $U$ is some
profile. In view of this small scale behavior, it is impossible in general to get uniform strong estimates in any Sobolev spaces containing
normal derivatives. Consequently, the vanishing viscosity problem with the no-slip boundary condition is widely open except Asano
\cite{Asano} and Sammartino and Caflisch \cite{SC} in the framework of analytic initial data, Maekawa \cite{Mae14} for the
initial vorticity located away from the boundary and Guo and Nguyen \cite{GN} for a steady flow over a moving plane. However,
when the no-slip boundary condition is replaced by the Navier slip boundary condition,
the situation becomes  better. Indeed, now the solution is expected to behavior like $u^\eps\sim u^0+\sqrt{\eps}U(t,y,z/\sqrt{\eps})$; the
amplitude of the boundary layer is  weaker, and one can hope to get an uniform estimates involving one normal derivative. In this case, the
vanishing viscosity limit has been justified rigorously in Iftimie and Planas \cite{Iftimie-Planas}, Iftimie and Sueur
\cite{Iftimie-Sueur}, Masmoudi and Rousset \cite{MasRou12} and Xiao and Xin \cite{XiaoXin13}. Furthermore,  for some special types of Navier
boundary conditions or boundaries,  uniform estimates in higher order Sobolev spaces can be obtained, see  Xiao and Xin \cite{Xin} and Beir\~ao
da Veiga and Crispo \cite{BC11}.

Going back to the free-surface incompressible Navier-Stokes equations, since the dynamic boundary condition can be viewed as the same type of
slip boundary conditions, one has the hope to establish the vanishing viscosity limit. For the case without surface tension and there is no
boundary below the fluid, Masmoudi and Rousset \cite{MasRou} justified the inviscid limit by using the framework of their earlier work
\cite{MasRou12} and some additional techniques. Elgindi and Lee \cite{EL14} discussed the same problem for the case with fixed surface tension,
however, some key points in their arguments are not clear to us (especially  before we posted the earlier version of our paper onto arXiv on 21 April, 2015). On the other hand, for the free boundary problems, it is also interesting to
show the vanishing surface tension limit. This is supposed to be somewhat simpler than the inviscid limit problem since the equation on the free
boundary is defined without boundary. Yet, one needs to develop the well-posedness which is uniform with respect to surface tension, and
generally this is nontrivial. These have been done for the irrotational Euler equations, see Ambrose and Masmoudi \cite{AM1,AM2} and references therein; for the
general Euler equations, a priori uniform estimates have been derived in \cite{SZ}. Note that these results are local in time. {For the Navier-Stokes equations, Tan and Wang \cite{TW14} proved the global-in-time vanishing surface tension limit for the small initial data. }The purpose of this paper is to derive the uniform estimates of solutions of the system \eqref{NS}
(equivalently, \eqref{NSv}) on a time interval independent of both viscosity and surface tension coefficients. These allow one to justify the
vanishing viscosity and surface tension limits by the strong compactness argument. As a byproduct, one can get  a unified local well-posedness of the free-surface incompressible Euler equations with or without surface tension by the inviscid limit.

\section{Main results}\label{sec main}

\subsection{Statement of the results}

We shall use Sobolev conormal spaces on $\Omega$ as \cite{MasRou12,MasRou}. Set
$$  Z_{i}= \partial_{i}, \, i=1, \, 2, \ Z_{3}= z(z+b) \partial_{z},$$
which are tangent to $\pa\Omega$. The Sobolev conormal space $H^k_{co}$  is defined as
$$ H^k_{co}(\Omega)= \left\{ f \in L^2(\Omega), \   Z^\alpha  f  \in L^2 (\Omega), \  \alpha\in\mathbb{N}^3,\ | \alpha| \leq k \right\},$$
where $ Z^\alpha = Z_{1}^{\alpha_{1}} Z_{2}^{\alpha_{2}} Z_{3}^{\alpha_{3}}  $ for $\alpha\in\mathbb{N}^3$, with norm defined as
$$ \norm{  f  }_{k} =\sum_{\alpha\in\mathbb{N}^3\atop|\alpha| \leq k}\norm{ Z^\alpha f }_{L^2}, \ \norm{  f  } = \norm{ f  }_0 = \norm{  f
}_{L^2}.
  $$
Similarly, $W^{k, \infty}_{co}$  is defined as
$$ W^{k, \infty}_{co}(\Omega)= \left\{ f \in L^\infty(\Omega), \   Z^\alpha  f  \in L^\infty (\Omega), \  \alpha\in\mathbb{N}^3, \  | \alpha|
\leq k \right\}$$
with norm
$$ \norm{  f  }_{k, \infty} =\sum_{\alpha\in\mathbb{N}^3\atop|\alpha| \leq k} \norm{ Z^\alpha f }_{L^\infty}.$$
$H^k$ and $W^{k,\infty}$ will denote for the usual Sobolev spaces on $\Omega$, and $|\cdot|_{s}$ and $|\cdot|_{s,\infty}$ stand for the standard
Sobolev norms on $\mathbb{R}^2$.
We introduce also the {\it spatial-time} Sobolev conormal norms on $\Omega$ as:
\beq\label{xym}
\norm{  f  }_{\X^{m,k}}^2 =\sum_{\ell=0}^m \norm{\dt^\ell f }_{m+k-\ell}^2, \  \norm{  f   }_{\X^{m}}= \norm{  f   }_{\X^{m,0}}, \text{ and
}\norm{  f  }_{\Y^{m}} =\sum_{\ell=0}^m \norm{\dt^\ell f }_{m-\ell,\infty},
\eeq
 and the {\it spatial-time} Sobolev norms on $\mathbb{R}^2$ :
\beq\label{xymb}
\abs{  f  }_{\X^{m,s}} =\sum_{\ell=0}^m \abs{\dt^\ell f }_{m+s-\ell}^2, \   \abs{  f  }_{\X^{m}}= \abs{  f  }_{\X^{m,0}},\text{ and }\abs{  f
}_{\Y^{m}} =\sum_{\ell=0}^m \abs{\dt^\ell f }_{m-\ell,\infty}.
\eeq
Note that in these definitions $k$ and $m$ are assumed to be non-negative integers, but $s$ is allowed  to be any real number, typically, halfs. {For $s\in\mathbb{R}$,   $\[s\]$ denotes   the largest integer no more than $s$. For the convenience, we shall use $\mathbb{N}^{1+d} = \{ \alpha = (\alpha_0,\alpha_1,\dotsc,\alpha_d) \}$ to emphasize that the $0-$index term is related to temporal
derivatives; for $\alpha \in \mathbb{N}^{1+d}$,  we denote $Z^\alpha = \dt^{\alpha_0} Z_1^{\alpha_1}\cdots
Z_d^{\alpha_d}.$}

{Since we will work in a high-regularity context with regularity up to $m$ temporal derivatives, then one needs to use the initial data $(v_0,h_0)$ and the equations \eqref{NSv} to construct the initial data $\dt^j v(0)$ and $\dt^j h(0)$  for $j=1,\dotsc,m$, and $\dt^j q(0)$ for $j = 0,\dotsc, m-1$. Such a construction is standard and classical, so will be omitted here, see \cite{GT13,TW14} for more details. For our analysis, these initial data must then satisfy various conditions which will be specified when necessary.}

The aim of this paper is to get a local well-posedness result for strong  solutions of  \eqref{NSv} in Sobolev conormal spaces
which is valid on an interval of time independent of $\es \in (0,1]$.  Note that such a result will also  imply  the  local
well-posedness for the Euler equation with or without surface tension. As it is well-known that
when there is no surface tension a Taylor sign condition on the free boundary is needed
to get local well-posedness for the Euler equation; in the presence of surface tension, no such condition is needed.
By the change of coordinates \eqref{diff}, the Taylor sign condition reads as
\beq\label{Taylor2}
 -    \partial_{z}^{\varphi_0} q_0 +g\geq c_{0}>0  \text{ on }\{z=0\}.
\eeq

In the below, $\N$ and $\n$ are extended to $\bar{\Omega}$ by
 \beq\N(t,y,z)=(-\pa_1\varphi(t,y,z),-\pa_2\varphi(t,y,z),1)^t\text{ and }\n=\N/|\N|.
 \eeq
 Note that $\N(t,y,0)=(-\pa_1h,-\pa_2h,1)^t$ and $\N(t,y,-b)=e_3$.
Define $\Pi=  \mbox{I} - {\bf n}\otimes {\bf n}$, and let $\chi(z)$ be a smooth function in $\mathbb{R}$ which takes the value zero in the vicinity of
$\{z=0\}$ and one in the vicinity of $\{z=-b\}$. Let $\norm{\cdot}_{L^p_TX}$ be the norm of the space $L^p([0,T];X)$. The following is the statement of the a priori estimates on a time interval, independent of $\eps$ and $\sigma$, for a sufficiently smooth solution $(v^\es, h^\es)$ of \eqref{NSv}.
  \begin{theoreme}
    \label{main}
Let $m \geq {13}$. Assume that {the initial data $(v ^\es_0, h^\es_0)$ is given such that}
\begin{align}\label{hypinit}
  &\abs{h^\es(0)}_{\X^m}^2+ \sigma \abs{h ^\es(0)}_{\X^{m,1}}^2+ \eps  \abs{h ^\es(0)}_{\X^{m,\hal}}^2
   \\\nonumber&\quad+ \norm{v ^\es(0)}_{\X^m}^2 + \norm{\partial_{z} v ^\es(0)}_{\X^{m-1}}^2
+ \norm{ \partial_{z} v^\es(0)}_{\Y^{{\[\!\frac{m}{2}\!\]}+2}}^2 +\eps  \norm{\partial_{zz} v^{\eps,\sigma}(0) }_{L^\infty}^2  \leq R_{0},   \end{align}
  and that  {\eqref{Adeb} and \eqref{Taylor2} hold uniformly in $\varepsilon$ and $\sigma$. Furthermore, the following compatibility condition
  \beq\label{comc}
  \Pi_0^\es \(S^{\varphi_0^\es} v_0^\es {\bf n}_0^\es-\kappa \chi v_0^\es\)= 0\text{ on }\pa\Omega
   \eeq
is valid. There exist $T>0$ and $C>0$  such that for every  $\eps,\sigma\in (0, 1]$, the  solution} $(v^\es, h^\es)$  of  \eqref{NSv} {with the initial data $(v ^\es_0, h^\es_0)$} satisfies the
estimate
\begin{align}\label{mainborne1}
&\abs{h^\es }_{L^\infty_T \X^{m-1,1}}^2+ \sigma \abs{h ^\es }_{L^\infty_T\X^{m-1,2}}^2+ \eps  \abs{h ^\es }_{L^\infty_T\X^{m-1 , {3\over 2}}}^2
\\&\quad+ \norm{v ^\es }_{L^\infty_T\X^{m-1,1}}^2 + \norm{\partial_{z} v ^\es }_{L^\infty_T\X^{m-2}}^2  + \norm{ \partial_{z} v^\es
}_{L^\infty_T\Y^{{\[\!\frac{m}{2}\!\]}+2}}^2+\eps  \norm{\partial_{zz} v^{\es} }_{L^\infty_T L^\infty}^2\nonumber\\&\quad+  \abs{\dt^m h^\es }_{L^4_T L^2
}^2+ \sigma \abs{\dt^m h^\es }_{L^4_T H^1 }^2+ \eps  \abs{\dt^m  h ^\es }_{L^4_TH^{ \hal}}^2+ \norm{ \dt^m v^\es }_{L^4_T L^2 }^2+ \norm{
\partial_{ z} v^\es }_{L^4_T \X^{m-1} }^2\nonumber
 \\\nonumber&\quad+\sigma^2 \abs{h ^\es }_{L^2_T \X^{m-1,\frac{5}{2}} }^2 +\eps  \norm{\nabla v ^\es }_{L^2_T \X^{m-1,1} }^2+ \eps \norm{\nabla
 \partial_z v ^\es }_{L^2_T \X^{m-2} }^2   \leq   C.\end{align}
\end{theoreme}

\begin{rem}
It should be remarked that the reason for the $L^4$-in-time estimate of $\norm{\partial_{z} v^\es(t)}_{\X^{m-1}} $ rather than $L^\infty$
stated in
Theorem \ref{main} is related  to the boundary  control of the vorticity in viscous boundary layers{, which is same as the case without surface tension \cite{MasRou}}. However, in the presence of surface tension, as will be shown in our proof later,
the pressure is of less regularity, thus one can only prove
the $L^4$-in-time estimate {rather than $L^\infty$} of the highest time
derivative term $\abs{\dt^m h^\es(t) }_{L^2 }+ \sqrt{\sigma} \abs{\dt^m h^\es(t) }_{H^1 }+ \sqrt{\eps } \abs{\dt^m h^\es(t) }_{H^\hal }+ \norm{
\dt^m v^\es(t) }_{L^2 }${, which reveals the significant difference from the case without surface tension \cite{MasRou}}.
\end{rem}

\begin{rem}
We emphasize that in the derivation of the estimate \eqref{mainborne1} in  Theorem \ref{main} only the compatibility condition \eqref{comc} is required; indeed, \eqref{comc} is used only in deriving the bound of $\eps  \norm{\partial_{zz} v^{\es} }_{L^\infty_T L^\infty}^2$. Since we manage to avoid the control of $\eps  \norm{ \partial_{zz}\dt^j v^{\es} }_{L^\infty_T L^\infty}^2$ for $j\ge 1$, so no higher order compatibility conditions, i.e., time derivatives version in initial time of \eqref{comc}, are required. However, for the local well-posedness of \eqref{NSv} one needs the following $m$-th compatibility conditions:
  \beq
  \begin{cases}\label{compp}
  \nabla^{\varphi_0}  \cdot  v ^\es_0 = 0  \text{ in }\Omega,\quad v ^\es_{0,3} = 0 \text{ on }\{z=-b\}, \text{ and}\\
  \dt^{j}\(\Pi^\es \(S^{\varphi^\es} v^\es {\bf n}^\es-\kappa \chi v^\es\)\)\mid_{t=0}= 0\text{ on }\pa\Omega,\text{ for }j=0,\dots,m-1.
  \end{cases}
   \eeq
See \cite{GT13,TW14} for more details.
\end{rem}
\begin{rem}
Due to the construction of $\dt^j v^\es(0)$ and $\dt^j h^\es(0)$ for $j\ge 1$, the conditions \eqref{hypinit} and \eqref{compp} are in fact imposed on the initial data $(v ^\es_0, h^\es_0)$. In particular, they are well-defined provided that $v ^\es_0\in H^{2m}(\Omega)$ and $h^\es_0\in H^{2m+1}(\Sigma)$. Given such initial data satisfying \eqref{Adeb} and \eqref{compp}, by the local well-posedness of \eqref{NSv} in \cite{TW14}, for fixed $\varepsilon>0$ and $\sigma>0$, one can get a positive time $T^\es$ for which a unique solution $(v^\es,h^\es)$ of \eqref{NSv} achieving this initial data exists on $[0, T^\es]$ so that
$
v^\es\in C ([0,T^\es]; H^{2m}(\Omega))\cap L^2(0,T^\es; H^{2m+1}(\Omega))$ and $ h^\es \in C ([0,T^\es]; H^{2m+1}(\Sigma))\cap L^2(0,T^\es; H^{2m+3/2}(\Sigma))
$
as well as the regularity of the time derivatives of $(v^\es,h^\es)$ followed by using the equations \eqref{NSv}.
In particular, these guarantee all of the computations involved in the derivation of the a priori estimates of Theorem \ref{main}.
\end{rem}

As immediate consequences of the uniform estimates of Theorem \ref{main}, one can establish easily, by
standard compactness arguments,  the justification of the inviscid limit, vanishing surface tension limit,  and any {of their combinations.  To illustrate this, we will present only the argument for the inviscid limit as $\eps\rightarrow0$, which is stated as follows:
   \begin{theoreme}
    \label{main2}
Let $m \geq {13}$. Assume that the initial data $v ^\es_0\in H^{2m}(\Omega)$ and $h^\es_0\in H^{2m+1}(\Sigma)$  satisfying the same assumptions as in Theorem \ref{main}. Furthermore, the  $m$-th  compatibility conditions \eqref{compp} hold.
There exist $T>0$ and $C>0$  such that for every  $\eps,\sigma\in (0, 1]$, there exists a  unique solution $(v^\es, h^\es)$  of  \eqref{NSv} with the initial data $(v ^\es_0, h^\es_0)$ on $[0,T]$  satisfying the uniform estimate \eqref{mainborne1}. Moreover, if assume further that
\beq
v_0^\es\rightarrow v_0^\sigma\text{ in }L^2(\Omega)\text{ and }h_0^\es\rightarrow h_0^\sigma\text{ in }L^2(\Sigma), \text{ as }\varepsilon\rightarrow 0,
\eeq
then as $\varepsilon\rightarrow 0$, $(v^\es, h^\es)$ converges to the limit  $(v^\sigma, h^\sigma)$, which is the unique solution to the free-surface Euler equations
\beq\label{Eulerv}
\begin{cases}   \D_{t}v +  v \cdot \nabla^\varphi  v + \nabla^\varphi q   =0 &\text{in }\Omega
\\ \nabla^\varphi  \cdot  v = 0  &\text{in }\Omega
\\q   = gh   -\sigma H &\text{on }\{z=0\}
\\\partial_{t}h = v \cdot \N &\text{on }\{z=0\}
\\v_3 = 0  &\text{on }\{z=-b\}
\end{cases}
\eeq
with the initial data $(v ^\sigma_0,h ^\sigma_0)$ on $[0,T]$
that satisfies the estimate
\begin{align}\label{mainborne12}
&\abs{h^\sigma }_{L^\infty_T \X^{m-1,1}}^2+ \sigma \abs{h ^\sigma}_{L^\infty_T\X^{m-1,2}}^2+ \norm{v ^\sigma }_{L^\infty_T\X^{m-1,1}}^2+ \norm{\partial_{z} v ^\sigma }_{L^\infty_T\X^{m-2}}^2+ \norm{ \partial_{z} v^\sigma}_{L^\infty_T\Y^{{\[\!\frac{m}{2}\!\]}+2}}^2   \\&\quad\nonumber+  \abs{\dt^m h^\sigma }_{L^4_T L^2
}^2+ \sigma \abs{\dt^m h^\sigma }_{L^4_T H^1 }^2 + \norm{ \dt^m v^\sigma }_{L^4_T L^2 }^2+ \norm{
\partial_{ z} v^\sigma }_{L^4_T \X^{m-1} }^2\nonumber
+\sigma^2 \abs{h ^\sigma }_{L^2_T \X^{m-1,\frac{5}{2}} }^2  \leq   C.\end{align}
 \end{theoreme}}
\begin{rem}
{The convergence of $(v^\es, h^\es)$   to   $(v^\sigma, h^\sigma)$ as $\varepsilon\rightarrow 0$ in Theorem \ref{main2}  occurs in any spaces which contain the set of functions obeying \eqref{mainborne12} as a compact subset.}
For the inviscid limit, the Euler equations, by using the equation for the vorticity one can
improve  {those integral-in-time estimates in \eqref{mainborne12} to be in $L^\infty$, and one can also recover the normal regularity with the initial data in Sobolev spaces.  }
\end{rem}

\subsection{Strategy of the proof}

{The main part of the paper will be devoted to prove Theorem  \ref{main}, where the}
key step is to derive the a priori uniform estimates on a time interval small but independent of
$\es\in(0,1]$ for a sufficiently smooth solution of the equations \eqref{NSv}. Our approach is
strongly motivated by the strategy of Masmoudi and Rousset \cite{MasRou} where the vanishing viscosity limit was justified for the problem
without surface tension. However, there are several new difficulties arising in the presence of surface tension. Indeed, it has been already
known for the free-surface incompressible Euler equations that the problem with surface tension is more difficult than the one without
surface tension in certain sense, see \cite{CS07} for some discussions. It is even more so in the study of the inviscid limit for the free-surface incompressible Navier-Stokes equations. Indeed, in the vanishing viscosity limit for each fixed
$\sigma>0$, the surface tension serves only to provide the improved regularity of the free surface, and makes the problem hard in the sense that first, the required nonlinear estimates are more difficult to close due to the mean curvature
term and second, the less regularity of the pressure makes the arguments much more involved. Note that the less regularity of the free surface for
the problem without surface tension was overcome in \cite{MasRou} by using Alinhac good unknowns \cite{Alinhac}. In this paper we will also
use Alinhac good unknowns so as to not taking advantage of the improved regularity of the free surface provided by any fixed surface tension;
but we need to develop some ideas to overcome the difficulties caused by the presence of  surface tension as illustrated below. As a
consequence, we will be able to deal with both the vanishing viscosity and surface tension limits.

For notational convenience we shall suppress the subscripts $\eps$ and $\sigma$ below unless states otherwise. Let $\mathcal{N}(T)$  be the quantity appearing in the left hand side of \eqref{mainborne1}, while $\mathcal{Q}(T)$  be the first two lines in \eqref{mainborne1}. The key point is thus to get a
closed a priori estimates of $\mathcal{N}(T)$ on a small time interval independent of $\eps$ and $\sigma$ in terms of the initial data. We start
with the basic physical energy identity:
\begin{align}\label{intro1}
& \hal\dtt \(  \int_{\Omega} |v|^2 \, d\V +   \int_{z=0} \(g|h|^2+2\sigma\(\sqrt{1+|\nabla_y h|^2}-1\) \)dy  \)
 \\\nonumber&\qquad\qquad\qquad\qquad\qquad\qquad+ 2 \eps  \int_{\Omega} |S^\varphi v|^2\, d\V+2 \kappa \eps\int_{z=-b}|v|^2\, dy=0.
 \end{align}
Here $d\V$ stands for the volume element induced by the change of variable \eqref{eqphi}: $d\V=\partial_{z}\varphi \, dydz$.

To get the estimates for higher order conormal energy estimates, one starts with applying the {\it spatial} conormal derivatives $Z^\alpha$ for $\alpha\in \mathbb{N}^3$ with $1\le |\alpha|\le m$ to the equations \eqref{NSv}.
Since the operators $\pa_i^\varphi$ involve $\nabla\varphi$, the estimate of the commutator between $Z^\alpha$ and $\pa_i^\varphi$ needs a
control of $\norm{Z^\alpha \nabla\varphi}\ls \abs{Z^\alpha h}_{1/2}$. In the absence of the surface tension this yields a loss of $1/2$
derivative, and this difficulty was overcome in Masmoudi and Rousset \cite{MasRou} by using a
crucial cancellation observed by Alinhac \cite{Alinhac}. The idea in \cite{MasRou} is to use the good unknowns $V^\alpha= Z^\alpha v - {\partial_{z}^\varphi v}
Z^\alpha \eta$ and $Q^\alpha=  Z^\alpha q - {\partial_{z}^\varphi q} Z^\alpha \eta$, and some cancellation occurs when considering the equations
for $V^\alpha$ and $Q^\alpha$ which allows one to derive an $L^2$ type energy estimates similar to  \eqref{intro1}:
\begin{align}\label{intro2}
&\hal\dtt\( \int_{\Omega} |V^\alpha|^2 \, d\V +  \int_{z=0} \((g - \partial_{z}^\varphi q ) |Z^\alpha h|^2+\sigma\abs{\nabla_y Z^\alpha h}^2\) dy\)+
2 \eps  \int_{\Omega} |S^\varphi V^\alpha|^2\, d\V
\\\nonumber&\quad =-\int_{z= 0}  \sigma Z^\alpha H \sum_{|\alpha'|=1}C_{\alpha}^{\alpha'}Z^{ \alpha'}\N\cdot Z^{\alpha-\alpha'} v
\,dy-\int_{\Omega} \mathcal{C}^\alpha (d) Z^\alpha q \, d\V +  {\sum}_0 +  {\sum}_\sigma,
\end{align}
where, using the symmetric commutator notation $\[ \cdot ,  \cdot , \cdot \]$ defined by \eqref{symcom},
\begin{align}\label{intro11}
  \pa_z\varphi \mathcal{C}^\alpha (d)   =  \[  Z^\alpha ,  \N , \cdot\partial_{z}v  \]  +   \[  Z^\alpha ,  \partial_{z}
  \eta,\pa_1v_1+\pa_2v_2\].
     \end{align}
Here ${\sum}_0$ denotes the terms that can be controlled in a similar way as the case without surface tension \cite{MasRou}, and ${\sum}_\sigma$
stands for  the terms related to surface tension that can be controlled well with the energy estimates
$\sigma\abs{  Z^\alpha h}_1^2$. The first two terms in the right hand side of \eqref{intro2} are singled out in order to indicate the main
difficulties for the case with surface tension. Note that the regularity $\sigma\abs{ Z^\alpha h}_1^2$ in the energy is not enough to control the
first term. Indeed, one needs $\sigma^2\abs{Z^\alpha h}_{3/2}^2$, $i.e.,$ there is a loss of $1/2$ derivative for $\sigma>0$. To improve the regularity of $h$, one
then resorts to using the normal component of the dynamic boundary condition in \eqref{NSv}, $i.e.$,
\beq\label{qeq}
-\sigma H  =q-gh-  2 \eps S^\varphi v  \,{\bf n} \cdot
\,{\bf n}\quad \text{on }\{z=0\},
\eeq
which requires a control of $ \abs{Z^\alpha q}_{-1/2}^2 $.  A natural way to control the
pressure $q$ is through the elliptic problem
\beq\label{intro3}
\begin{cases}
\Delta^\varphi q = - \nabla^\varphi \cdot (v\cdot \nabla^\varphi v)& \text{in }\Omega
\\\nabla^\varphi q\cdot\N = -\dt v\cdot \N-(v_y\cdot\nabla_y) v\cdot \N+\eps   \Delta^\varphi v \cdot\N &\text{on }\{z=0\}
\\\nabla^\varphi q\cdot\N =  \eps   \Delta^\varphi v \cdot\N &\text{on }\{z=-b\}.
\end{cases}
\eeq
It should be remarked that when $\sigma=0$, to estimate $q$ one can use instead on the boundary $\{z=0\}$ the Dirichlet boundary condition $q=-\sigma H +gh+  2 \eps S^\varphi v
\,{\bf n} \cdot \,{\bf n}  $ by \eqref{qeq}, see \cite{MasRou}. However, for $\sigma>0$, this boundary condition can not be used to estimate $q$ as one has not controlled $-\sigma H$ yet.
Note that the appearance of $\dt v$ in the Neumann boundary condition on $\{z=0\}$ of
\eqref{intro3} forces one to include the time derivative $\dt$ in $Z^\alpha${. So, from now on, one takes the {\it spatial-time} conormal derivatives $Z^\alpha$ for $\alpha\in \mathbb{N}^{1+3}$}. The elliptic estimates for \eqref{intro3} provide the control of
$\norm{\nabla q}_{L^2_T\X^{m-1}}^2$, {where the good unknown $V^m$ ($i.e.$, $V^\alpha$ for $\alpha_0=m$) is used to derive that $\dt^m v\cdot\N \in H^{-1/2}(\Sigma)$ when estimating  $\norm{\dt^{m-1}\nabla q}_{0}^2$.} By using \eqref{qeq}, the pressure estimates then yield the control of $\sigma^2\abs{h}_{L^2_T\X^{m-1,\frac{5}{2}}}^2$.  Now one separates the estimates of \eqref{intro2} into two cases: when $\alpha_0\le m-1$ or when $\alpha_0=m$, and for the former case one can use these estimates of $q$ and $h$ above to conclude that
\begin{align}\label{intro4}
&\norm{v(t)}_{\X^{m-1,1}}^2+\abs{h(t)}_{\X^{m-1,1}}^2 +\sigma\abs{h(t)}_{\X^{m-1,2}}^2+\eps  \abs{h(t)}_{\X^{m-1, \frac{3}{2}}}^2
\\\nonumber&\quad \le C_0+\Lambda\(\mathcal{Q}(T)\)\(t+\int_0^t \( \abs{\dt^m h}_{-\hal}^2+\sigma \abs{\dt^m h}_{1}^2 +\norm{\dt^m v}_{0}^2+
\norm{\pa_z v}_{\X^{m-1}}^2 \)\).
\end{align}
For the case when $\alpha_0=m$, \eqref{intro2} can be rewritten as
\begin{align}\label{intro5}
&\hal\dtt\( \int_{\Omega} |V^m|^2 \, d\V +  \int_{z=0} \((g - \partial_{z}^\varphi q ) |\dt^m h|^2+\sigma\abs{\nabla_y \dt^m h}^2\) dy\)+ 2 \eps
\int_{\Omega} |S^\varphi V^m|^2\, d\V
\\\nonumber&\quad =-\int_{z= 0}  \sigma \dt^m H m\dt \N\cdot \dt^{m-1} v\,dy-\int_{\Omega} m\dt \N\cdot \dt^{m-1}\pa_z v \dt^m q \, dydz +
{\sum}_0 +  {\sum}_\sigma+  {\sum}_q.
\end{align}
Here ${\sum}_q$ stands for the terms related to $q$, which are relatively easier to estimate than the second term in the
right hand side of \eqref{intro5}. The main difficulty now is that there are no any estimates for $\dt^mq$ and the regularity $\sigma\abs{ \dt^m h}_1^2$ in the energy can no longer be improved by using \eqref{qeq}, hence it is
difficult to bound the first two terms in the
right hand side of \eqref{intro5}. {It should be emphasized that for the Euler equations with surface tension one can use the vorticity equation to prove that $ \dt^{m-1}v\in H^{3/2}(\Omega)$ (see \cite{CS07}) and $\Pi\dt^m v\in H^{-1/2}(\Sigma)$ (see \cite{CCS08}), each of which can be employed to control these two terms; unfortunately, both of these two estimates are difficult to be derived in the presence of viscosity.} Our approach  to
overcome the difficulty is to do the integration by parts over $t$ and $z$ in an appropriate order to obtain a crucial cancelation. More precisely,
we integrate by parts in $z$ first and then in $t$ the second term to obtain
\begin{align}\label{intro6}
&-\int_{\Omega} m\dt \N\cdot \dt^{m-1}\pa_z v \dt^m q \, dydz
\\&\nonumber\quad=-\int_{z=0} m\dt \N\cdot \dt^{m-1} v \dt^m q \, dy +\int_{\Omega} m \pa_z(\dt \N \dt^m q)\cdot \dt^{m-1} v \, dydz
\\&\nonumber\quad=-\int_{z=0} m\dt \N\cdot \dt^{m-1} v \dt^m q \, dy +\frac{d}{dt}\int_{\Omega} m\pa_z(\dt\N \pa_t^{m-1} q)\cdot \pa_t^{m-1} v
dydz +{\sum}_0  .
\end{align}
Note that one does not integrate by parts in $t$ for the first term in the last line of \eqref{intro6} since one can not control $\Pi\dt^{m} v$ on
$\{z=0\}$. The crucial observation is that there is a cancelation between this term and the first term in the right hand side of \eqref{intro5}.
Indeed, by the dynamic boundary condition \eqref{qeq}, it holds that
\begin{align}\label{intro7}
&-\int_{z= 0}  \sigma \dt^m H m\dt \N\cdot \dt^{m-1} v\,dy-\int_{z=0} m\dt \N\cdot \dt^{m-1} v \dt^m q \, dy
 \\\nonumber&\quad=-\int_{ {z=0}}   m\pa_t  \N  \cdot \pa_t^{m-1} v  \(g\dt^mh +2\eps \dt^m \( S^\varphi v  \,{\bf n} \cdot \,{\bf n}\)\)\,
 dy={\sum}_0.
\end{align}
It should be pointed out that after we posted the earlier version of our paper onto arXiv on 21 April, 2015, Elgindi and Lee \cite{EL14} had used our idea here to fix some arguments in their study of the inviscid limit for the free-surface Navier-Stokes equations for fixed $\sigma>0$ (see page 46 of \cite{EL14}).
The last difficulty in the estimates of \eqref{intro2} when $\alpha_0=m$ is due to the second term in the last line of \eqref{intro6} since the
control of $\norm{\nabla q(t)}_{\X^{m-1}}$ is $L^2$-in-time rather that $L^\infty$. One
is then forced to integrate in time twice \eqref{intro2} when $\alpha_0=m$, and the main conclusion is that
\begin{align}
&\(\int_0^t\(\norm{\dt^m v }^2+\abs{\dt^m h }_{0}^2 +\sigma \abs{\dt^m h }_{1}^2+ \eps   \abs{\dt^m h}_{\hal}^2 \)^2\)^\hal
\\\nonumber&\quad\le  C_0+   \Lambda(\mathcal{Q}(T))\(t + \int_0^t \( \abs{\dt^mh}_{0}^2 + \sigma     \abs{\dt^mh}_{1}^2
+ \norm{\dt^m v}_{0} ^2 + \norm{\pa_z v}_{\X^{m-1}}^2 + \eps  \abs{\dt^m h}_{\hal}^2 \) \)^\hal .
\end{align}

In order to close the argument, one needs to estimate $\pa_z v$. We start with the conormal energy estimates of $\pa_z v$. As in \cite{MasRou}, one first introduces the equivalent quantity
$S_{\n}= \Pi \(S^\varphi v   \n-\kappa\chi v\)$, which satisfies a convention-diffusion type equation with the homogeneous Dirichlet boundary
condition. When $\sigma=0$ one estimates only the spatial conormal derivatives of $\pa_z v$, hence to control the commutator between $Z^\alpha$
and $\varepsilon \Delta^\varphi v$ it needs only the estimate of $\sqrt{\eps} \norm{\partial_{zz}v}_{L^\infty}$, see \cite{MasRou}. However, for $\sigma>0$ as the time
derivatives involved, if one followed the arguments of \cite{MasRou}, one would need  to control $\sqrt{\eps}  \norm{\partial_{zz}v}_{\Y^{k}}$ for certain $k\ge 1$, and this would then require higher order compatibility conditions other than \eqref{comc}. Our way to avoid this is to control instead $\eps \partial_{zz}v $ (and even $\eps \partial_{zzz}v$!), and this can be done in a much more direct way, due to that we have included the time derivatives
in $Z^\alpha$, by using the first equation in
\eqref{NSv}. One can then perform the $L^2$ type energy estimates to conclude that
\begin{align}
\norm{\pa_z v(t)}_{\X^{m-2}}^2\le  C_0+   \Lambda(\mathcal{Q}(T))\(t + \int_0^t  \norm{\pa_z v}_{\X^{m-1}}^2   \) .
\end{align}
Note that  the $m-2$ order estimate above cannot be improved to be $m-1$ due to the appearance of
$(\nabla^\varphi)^2 q$ in the source term in the equation for $S_{\n}$, which is only in $\X^{m-2}$. To get a better estimate, following
\cite{MasRou,MasRou12,Xin}, one would proceed with the vorticity  $\omega=\nabla^\varphi\times v$ instead of $S_{\n}$ to get rid of the pressure term. $\omega$ again satisfies
a convention-diffusion type equation, but the main difficulty is that it does not vanish on the boundary and its boundary value is at a low
regularity.  {Note that away from the boundary
  the conormal Sobolev norm is equivalent to the standard Sobolev norm, so one needs only to estimate  $  \omega  $ near the boundary. We consider only the
estimates of  $\omega$ near $\{z=0\}$, and the estimates near $\{z=-b\}$ follow similarly. Let $\chi(z)$ be smooth compactly supported near $\{z=0\}$ and equal to $1$ in a vicinity of $\{z=0\}$, and one may regard the equation satisfied by $\chi\omega$ as to be defined in  the half space $\{z<0\}$.}  To split the difficulty, for $|\alpha|=m-1$ we set
$ Z^\alpha {(\chi\omega)}= \omega^\alpha_{nh}+ \omega^{\alpha}_{h}$, where $\omega^\alpha_{h}$ satisfies the nonhomogeneous equation with the
homogeneous boundary condition which can be handled by performing the $L^2$ type energy estimates and $\omega^\alpha_{nh}$ satisfies the
homogeneous equation with the nonhomogeneous boundary condition. Note that  {$\omega^\alpha_{nh}$ solves exactly the same problem in $\{z<0\}$ as \cite{MasRou}}
so that one can apply directly the estimates of $\omega^\alpha_{nh}$ in Theorem 10.6 of \cite{MasRou}. This together with the estimates of $\omega^\alpha_{h}$ yields
\begin{align}
\(\int_0^t \norm{\pa_z v(t)}_{\X^{m-1}}^4\)^\hal \le  C_0+   \Lambda(\mathcal{Q}(T))\(t + \int_0^t  \norm{\pa_z v}_{\X^{m-1}}^2   \) .
\end{align}

Now we derive the $L^\infty$ estimates of $\pa_zv$ and $\sqrt{\eps}\pa_{zz}v$. In this step, one needs to estimate only a low number of
derivatives of $v$, say ${\[\!\frac{m}{2}\!\]}+2$, while the boundary is $H^m$ with  $m$ being as large as needed, hence it is convenient to use a normal
geodesic coordinate system in the vicinity of the boundary so that the Laplacian has the simplest expression. Note that $\sqrt{\eps}\pa_{zz}v$
can be controlled in the same way as \cite{MasRou}, however, for $\pa_zv$ we will employ a different argument. Again, it is more convenient to
estimate the equivalent quantity $S_\n$. After some computations, one finds that $\rho$, an equivalent quantity of $S_\n$ {near $\{z=0\}$} in the new coordinates,
solves
\beq \label{intro10}
\partial_{t}  \rho + w \cdot \nabla   \rho - \eps  \partial_{zz}   \rho= \mathcal{H} \quad{\text{in }\{z<0\}}
\eeq
for some source term $\mathcal{H}$ and a vector field $w$ with $w_3=0$ on the boundary. The main difficulty in the analysis of \eqref{intro10} is the
commutator between $Z^\alpha$ and $\eps\pa_{zz}\rho$ which is hard to
control when applying the maximum principle. The main idea in \cite{MasRou} to overcome this difficulty is to rewrite \eqref{intro10} into a one-dimensional Fokker-Planck
type equation and then use the explicit representation of the solution. Here we will use a different argument from \cite{MasRou}, which is much simpler; since we have included the time derivatives in $Z^\alpha$, it is easy to estimate the commutator by using again the first equation of \eqref{NSv}. The main
conclusion is that
\begin{align}
\norm{\pa_zv(t)}_{\Y^{{\[\!\frac{m}{2}\!\]}+2}}^2+\eps\norm{\pa_{zz}v(t)}_{L^\infty}^2\le  C_0+   \Lambda(\mathcal{Q}(T))t .
\end{align}

Combining the estimates in all these steps, we then derive the desired estimates $\mathcal{N}(T)\le C_0$ for some $T$ sufficiently small but
independent of $\eps$ and $\sigma$. Note that the Taylor sign condition and the condition that $\Phi(t,\cdot)$ is a diffeomorphism can be easily
justified due to our estimates of time derivatives. Finally we remark that for each fixed $\sigma>0$, the Taylor sign condition is no longer
needed for the inviscid limit problem. This can be seen from \eqref{intro1} and \eqref{intro2}: even $g-\pa_z^\varphi q\le 0$, one can use the
Sobolev interpolation to get the estimates of $h$ for each $\sigma>0$.

We will set the conventions for notation to be used later.  The Einstein convention of summing over repeated indices will be used.
Throughout the paper $C>0$ will denote a generic constant that does
not depend on the data, the surface tension coefficient $\sigma$ and
the  viscosity coefficient $\eps$, but can depend on the
other parameters of the problem, $g,\kappa$, $m\ge {13}$ and $\Omega$.  We refer to
such constants as ``universal''.  Such constants are allowed to change from line to line.  We will
employ the notation $a \lesssim b$ to mean that $a \le C b$ for a
universal constant $C>0$. Throughout the paper, the notation $\Lambda(\cdot, \cdot)$ stands for a continuous increasing function in all its
arguments, independent of $\eps$ and $\sigma$ and  that may change from line to line, and $\Lambda_0=\Lambda( \frac{1}{c_0})$.

The rest of the paper is organized as follows. We collect some analytic tools related to Sobolev conormal spaces, the properties of Poisson
extension and some geometric estimates in Appendixes \ref{aa1}, \ref{aa2} and \ref{aa3}, respectively. In Section \ref{sectionconorm} we study
the equations satisfied by $(Z^\alpha v, Z^\alpha q, Z^\alpha h)$ and present the estimates of the commutators. Section \ref{sectionpressure} is
devoted to derive the pressure estimates using elliptic regularity in Sobolev conormal spaces, and Section \ref{sectionprelimeta} contains the
smoothing regularity estimates of $h$ due to viscosity and surface tension. In Section \ref{secconormal}, the conormal estimates of the solution
are derived, and the conormal estimates for normal derivatives are given in Section \ref{secnormal}. In Section \ref{secinfty}, we prove the
needed $L^\infty$ estimates for normal derivatives.
Finally, the proofs of Theorems \ref{main} and \ref{main2} are given in Sections \ref{finalsec} and \ref{finalsec2}, respectively.

\section{Equations satisfied by $(Z^\alpha v,Z^\alpha q,Z^\alpha h)$}\label{sectionconorm}

\subsection{A commutator estimate}

In order to perform higher order conormal estimates, one needs to compute the equations satisfied by $(Z^\alpha v,Z^\alpha q,Z^\alpha h)$ {for $\alpha\in \mathbb{N}^{1+3}$ with $1\le |\alpha|\le m$}, which
requires  to commute $Z^\alpha$ with each term in the equations \eqref{NSv}. It is thus useful to establish the following general expressions and
estimates for commutators to be used often later.

We will not commute $Z^{\alpha}$ with $\D_t $ directly. For $i=  1, \, 2,\, 3$, set
\beq\label{com1}
Z^\alpha \D_{i}f=\D_{i} Z^\alpha f-\D_{z}f \D_{i} Z^\alpha \eta+\mathcal{C}^\alpha_{i}(f),
\eeq
where the commutator $\mathcal{C}^\alpha_{i}(f)$ is given for $\alpha \neq 0$ and $i \neq 3$ by
\beq\label{Cialpha}
\mathcal{C}^\alpha_{i}(f)= \mathcal{C}^\alpha_{i, 1}(f)+ \mathcal{C}^\alpha_{i,2}(f)+\mathcal{C}^\alpha_{i, 3}(f)
\eeq
with
\begin{align}
\label{Cialpha1} \mathcal{C}^{\alpha}_{i,1}&=-\[ Z^\alpha, {\partial_{i} \varphi \over \partial_{z} \varphi }, \partial_{z} f  \], \\
\label{Cialpha2} \mathcal{C}^{\alpha}_{i,2}&=-\partial_{z} f \[ Z^\alpha, \partial_{i} \varphi, {1 \over \partial_{z} \varphi}\]     -
\partial_{i} \varphi\partial_zf\[Z^{\alpha-\alpha'}, { 1 \over (\partial_{z} \varphi)^2}\]Z^{\alpha'}\partial_{z} \eta, \\
\label{Cialpha3} \mathcal{C}^\alpha_{i, 3}&= - {\partial_{i} \varphi \over \partial_{z} \varphi} [Z^\alpha, \partial_{z}]f
 + {\partial_{i} \varphi \over (\partial_{z} \varphi)^2}\, \partial_{z} f\,  [Z^\alpha, \partial_{z}] \eta,
\end{align}
for any $\alpha'<\alpha$ with $|\alpha'|=1$. Note that for $i= 1, \, 2$ ,  $\partial_{i}\varphi=\partial_{i}\eta$ and that for $\alpha \neq 0$,
$Z^\alpha\partial_{z}\varphi= Z^\alpha \partial_{z}\eta$. For $i=3$, similar decomposition for the commutator holds (basically, it suffices to
replace $\partial_{i}\varphi$ by $1$ in the above expressions). Since $\D_{i}$ and $\D_{j}$ commute, it holds that
\beq\label{com122}
Z^\alpha \D_{i} f =\D_{i}(Z^\alpha f-\D_{z}f  Z^\alpha \eta)+\D_{z}\D_{i}f Z^\alpha \eta+\mathcal{C}^\alpha_{i}(f).
\eeq
It was first observed by Alinhac \cite{Alinhac} that the highest order term of $\eta$ (which is difficult to control for the case $\sigma=0$)
will be canceled when one uses the good unknown $Z^\alpha f-\D_{z}fZ^\alpha \eta$, which allows one to perform high order energy estimates.

Since the expressions as $f/\partial_{z} \varphi$ will appear often later, we shall first state a general estimate. It is assumed that
$\partial_{z} \varphi \geq \frac{{c_{0}}}{2}$ and $|h|_{2, \infty} \leq {1 \over c_{0}}$.
\begin{lem}
For every $k\in \mathbb{N}$, it holds that
\beq \label{quot}
\norm{ {f \over \partial_{z} \varphi}}_{\X^k} \leq \Lambda\( {1 \over c_{0}},  \abs{h}_{\Y^{{\[\!\frac{k}{2}\!\]}+1}}  + \norm{ f }_{\Y^{{\[\!\frac{k}{2}\!\]}}}
\)\(\abs{h}_{\X^{k,{1 \over 2}}} +\norm{f}_{\X^k} \).
\eeq
\end{lem}
\begin{proof}
Since $\partial_{z} \varphi= 1+ \partial_{z} \eta$, so
$$ {f \over \partial_{z} \varphi}=   f  -   f {  \partial_{z} \eta \over  1+  \partial_{z} \eta } = f - fF(\partial_{z} \eta),$$
where $F(x)= x/(1+ x)$ is smooth and bounded together with all  its derivatives on $1+x \geq c_{0}>0$ and $F(0)=0$. Consequently, the product
estimate \eqref{gues} implies that
$$\norm{{f \over \partial_{z} \varphi } }_{\X^k} \lesssim  \norm{f   }_{\X^k}+  \norm{ f }_{\X^k}\norm{ F(\partial_{z} \eta)
}_{\Y^{{\[\!\frac{k}{2}\!\]}}}+\norm{ f }_{\Y^{{\[\!\frac{k}{2}\!\]}}} \norm{ F(\partial_{z} \eta) }_{\X^k} .$$
Notice that
$$ \norm{ F(\partial_{z} \eta)}_{\Y^{{\[\!\frac{k}{2}\!\]}}} \le  \Lambda\( {1 \over c_{0}}, \norm{\nabla \eta}_{\Y^{{\[\!\frac{k}{2}\!\]}}}\) ,  $$
and \eqref{gues} implies again,
$$ \norm{F(\partial_{z} \eta) }_{\X^k} \lesssim \Lambda\( {1 \over c_{0}}, \norm{\nabla \eta }_{\Y^{{\[\!\frac{k}{2}\!\]}}}\) \norm{ \partial_{z}
\eta}_{\X^k}.$$
Hence the estimate \eqref{quot} follows from these,  \eqref{etainfty} and \eqref{etaharm}.
\end{proof}

Next lemma deals with the estimates of the commutators $\mathcal{C}^\alpha_{i}(f)$.
\begin{lem}\label{comi}
For $1 \leq | \alpha |\leq m$, $i=   1, \, 2, \, 3$,  it holds that
\beq\label{comiest}
\norm{\mathcal{C}_{i}^\alpha(f) }_{0}     \leq \Lambda\( {1 \over c_{0}}, |h |_{\Y^{{\[\!\frac{m}{2}\!\]}+1}}   + \norm{  \pa_z f }_{\Y^{{\[\!\frac{m}{2}\!\]}}} \)
\(\norm{\pa_z f }_{\X^{m-1}} + \abs{h}_{\X^{m-1,{1 \over 2 }}}\).
\eeq
\end{lem}
\begin{proof}
We only present the proof for $i=1, \, 2$, and the last case is similar and slightly easier.

First, for $\mathcal{C}_{i,1}^\alpha$, it follows from  the commutator estimate \eqref{comsym} that
$$\norm{  \mathcal{C}^\alpha_{i,1} }_{0}  \ls  \norm{  Z \({\partial_{i }  \varphi \over \partial_{z} \varphi } \)
}_{\Y^{\[\!\frac{m}{2}\!\]-1}}\norm{Z\partial_{z} f }_{\X^{m-2}} +   \norm{ Z \partial_{z} f }_{\Y^{{\[\!\frac{m}{2}\!\]-1}}} \norm{Z\({\partial_{i }  \varphi \over
\partial_{z} \varphi } \) }_{\X^{m-2}}.$$
Consequently,  \eqref{quot} yields that
$$\norm{  \mathcal{C}^\alpha_{i,1} }_{0}  \leq  \Lambda\({1\over c_{0}},\norm{\nabla \varphi}_{\Y^{\[\!\frac{m}{2}\!\]}} + \abs{h}_{\Y^{{\[\!\frac{m+1}{2}\!\]}}} +
\norm{  \partial_{z} f }_{\Y^{{\[\!\frac{m}{2}\!\]}}}\)\( \norm{ \partial_{z} f }_{\X^{m-1}}+\abs{h}_{\X^{m-1,{1 \over 2 }}} +\norm{ \partial_{i}\eta
}_{\X^{m-1}} \).$$
It then follows from \eqref{eqphi}, \eqref{etainfty} and \eqref{etaharm} that
\beq\label{Calpha1}
\norm{\mathcal{C}^\alpha_{i,1} }_{0}  \leq \Lambda\( {1 \over c_{0}}, \abs{h }_{\Y^{{\[\!\frac{m}{2}\!\]}+1}}+\norm{\pa_z f }_{\Y^{{\[\!\frac{m}{2}\!\]}}}\) \(\norm{\pa_z
f }_{\X^{m-1}} + \abs{h}_{\X^{m-1,{1 \over 2 }}}\).
\eeq

Next, for the first term in $\mathcal{C}_{i,2}^\alpha$, one can use similar arguments: \eqref{comsym}
   and \eqref{eqphi} yield  that
$$\norm{ \partial_{z} f \[ Z^\alpha, \partial_{i} \varphi, {1 \over \partial_{z} \varphi}\] }_{0}
\ls  \norm{ \partial_{z} f }_{L^\infty} \( \norm{Z\partial_{i }   \varphi }_{\X^{m-2}} \norm{  {Z\pa_z\eta \over (\partial_{z} \varphi)^2}
}_{\Y^{{\[\!\frac{m}{2}\!\]-1}}}+  \norm{  Z  \partial_{i }   \varphi  }_{\Y^{\[\!\frac{m}{2}\!\]-1}}\norm{ {Z\pa_z\eta \over (\partial_{z} \varphi)^2}}_{\X^{m-2}}
\)
$$
and hence by using \eqref{quot}, \eqref{etaharm} and \eqref{etainfty}, one can show that
$$ \norm{ \partial_{z} f \[ Z^\alpha, \partial_{i} \varphi, {1 \over \partial_{z} \varphi}\] }_{0}
  \le    \Lambda\( {1 \over c_{0}}, \abs{h }_{\Y^{{\[\!\frac{m}{2}\!\]}+1}}   + \norm{   \pa_z f }_{L^\infty} \)\abs{h}_{\X^{m-1,{1 \over 2 }}}.$$
By \eqref{com} instead of \eqref{comsym}, the same estimate holds for the second term in $\mathcal{C}_{i, 2}^\alpha$. Hence,
\beq\label{Calpha2}
\norm{\mathcal{C}^\alpha_{i, 2} }_{0}
\leq\Lambda\( {1 \over c_{0}}, \abs{h }_{\Y^{{\[\!\frac{m}{2}\!\]}+1}}   + \norm{   \pa_z f }_{L^\infty} \)\abs{h}_{\X^{m-1,{1 \over 2 }}}.
\eeq

It remains to estimate $\mathcal{C}_{i, 3}^\alpha$. Notice that
\beq\label{idcom}
\[Z^\alpha, \partial_{z}\] f= \sum_{| \beta | \leq m-1} c_{\beta } \partial_{z} ( Z^\beta f)
\eeq
for some harmless smooth bounded functions $c_{\beta}.$ This yields, by using again \eqref{etaharm},
\begin{align}\label{Calpha3}
\norm{ \mathcal{C}^\alpha_{i, 3}}_{0} &\leq \Lambda\( {1 \over  c_{0}}, \norm{\partial_{i} \varphi}_{L^\infty}+\norm{\partial_{z} f }_{L^\infty}\)\(
\norm{\partial_{z} f }_{\X^{m-1}} + \norm{\partial_{z} \eta }_{\X^{m-1}} \)
\\\nonumber&\leq \Lambda\( {1 \over  c_{0}}, \abs{h}_{1,\infty} + \norm{\pa_z f }_{L^\infty}\)\( \norm{\pa_z f }_{\X^{m-1}}+   \abs{h}_{\X^{m-1
,\hal}}\).
\end{align}

Consequently, the estimate \eqref{comiest} follows by collecting \eqref{Calpha1}, \eqref{Calpha2} and \eqref{Calpha3}.
\end{proof}

\subsection{Interior equations}

We shall now derive the equations  in the domain $\Omega$ satisfied by the good unknowns $V^\alpha = Z^\alpha v - \D_{z}v\, Z^\alpha \eta  $ and
$Q^\alpha = Z^\alpha q - \D_{z}q\, Z^\alpha \eta.$
  \begin{lem}
  \label{lemValpha}
  For $1 \leq | \alpha | \leq m$, it holds that
   \begin{eqnarray}
  \label{eqValpha} & &  \D_{t} V^\alpha + v \cdot \nabla^\varphi V^\alpha + \nabla^\varphi Q^\alpha- 2 \eps \nabla^\varphi \cdot S^\varphi
  V^\alpha
     \\
    \nonumber   & &  \quad  = (\D_{z}v  \cdot \nabla^\varphi v)
        Z^\alpha \eta-\mathcal{C}^\alpha(\mathcal{T}) - \mathcal{C}^\alpha(q)+\eps  \mathcal{D^\alpha}\big( S^\varphi v \big) + \eps
        \nabla^\varphi \cdot \big( \mathcal{E}^\alpha (v)\big)
       , \\
   \label{divValpha}
     & &    \nabla^\varphi \cdot V^\alpha =- \mathcal{C}^\alpha (d),
       \end{eqnarray}
       where the commutators $\mathcal{C}^\alpha(q)$, $\mathcal{C}^\alpha(d)$, $\mathcal{E}^\alpha(v)$ and $\mathcal{C}^\alpha (\mathcal{T})$
       satisfy the estimates:
       \begin{align}
       \label{Cq} &\norm{\mathcal{C}^\alpha (q)}_{0}  \leq \Lambda\( {1 \over c_{0}}, \abs{h }_{\Y^{{\[\!\frac{m}{2}\!\]}+1}}   + \norm{  \pa_z q
       }_{\Y^{{\[\!\frac{m}{2}\!\]}}} \) \( \norm{ \pa_z q }_{\X^{m-1}} + |h|_{\X^{m-1,{1 \over 2 }}}\), \\
        \label{Cd} &  \norm{\mathcal{C}^\alpha (d)}_{0}  \leq \Lambda\( {1 \over c_{0}}, |h |_{\Y^{{\[\!\frac{m}{2}\!\]}+1}}   +\norm{   \pa_z v
        }_{\Y^{{\[\!\frac{m}{2}\!\]}}} \) \( \norm{ \pa_z v }_{\X^{m-1}} + |h|_{\X^{m-1,{1 \over 2 }}}\),    \\
  \label{CE} &\norm{\mathcal{E}^\alpha(v)}_{0} \leq\Lambda\( {1 \over c_{0}}, |h |_{\Y^{{\[\!\frac{m}{2}\!\]}+1}}   + \norm{  \pa_z v }_{\Y^{{\[\!\frac{m}{2}\!\]}}} \)
   \( \norm{ \pa_z v}_{\X^{m-1}} + |h|_{\X^{m-1,{1 \over 2 }}}\),
  \\\label{CT}  &\norm{ \mathcal{C}^\alpha (\mathcal{T}) }_{0}   \leq \Lambda \( {1 \over c_{0}}, |h|_{\Y^{{\[\!\frac{m}{2}\!\]}+1}} +  \norm{   v
  }_{\Y^{{\[\!\frac{m}{2}\!\]}}}  +  \norm{  \pa_z v }_{\Y^{{\[\!\frac{m}{2}\!\]}}}   \)  \( \norm{v}_{\X^{m-1}}+\norm{ \pa_z v}_{\X^{m-1}}+
  \abs{h}_{\X^{m,-\hal}}\),
 \end{align}
and $\mathcal{D}^\alpha (S^\varphi v ) $ is given by
$$\mathcal{D}^\alpha\big( S^\varphi v \big)_{ij}= 2 \,  \mathcal{C}_{j}^\alpha \big( S^\varphi v)_{ij}. $$

  \end{lem}
  It should be noted  that the commutator $\mathcal{D}^\alpha(S^\varphi v)$ will be estimated later by using the integration by parts.
  \begin{proof}
First, the equations  \eqref{eqValpha}--\eqref{divValpha}  follows from applying $Z^\alpha$ to the equations \eqref{NSv}. Indeed,  \eqref{com122}
implies that
  \beq
  \label{comp1}
   Z^\alpha \nabla^\varphi q= \nabla^\varphi Q^\alpha  +\partial_{z}^\varphi \nabla^\varphi q   Z^\alpha \eta
   + \mathcal{C}^\alpha(q),
   \eeq
   where $  \mathcal{C}^\alpha(q)= \( \mathcal{C}_{1}^\alpha(q) ,\mathcal{C}_{2}^\alpha (q) ,
   \mathcal{C}_{3}^\alpha(q)\)^t$, and
\beq
\label{div1}
Z^\alpha  \nabla^\varphi \cdot v  = \nabla^\varphi\cdot V^\alpha + \partial_{z}^\varphi \nabla^\varphi\cdot v Z^\alpha \eta
 + \mathcal{C}^\alpha(d),
\eeq
where
$ \mathcal{C}^\alpha(d) =  \sum_{i=1}^3 \mathcal{C}^\alpha_{i}(v_i).$

Next, note that
   \beq
   \label{transportW} \D_{t} + v \cdot \nabla^\varphi
   =  \partial_{t}+ v_{y} \cdot \nabla_{y}v +  V_{z} \partial_{z},\eeq
 where $V_{z}$ is defined by
   \beq
   \label{Wdef}
    V_{z}= {1\over \partial_{z} \varphi} v_{z}\ \text{ with }v_z=    v\cdot \N- \partial_{t} \varphi =  v\cdot\N- \partial_{t} \eta .
   \eeq
By using \eqref{transportW}, one can thus get that
\begin{align}  \label{T1}    Z^\alpha \( \D_{t} + v \cdot \nabla^\varphi   \)v
&    =   \(\partial_{t}+ v_{y} \cdot \nabla_{y} +  V_{z} \partial_{z} \) Z^\alpha v + \big(v\cdot Z^\alpha \N-
\partial_{t} Z^\alpha \eta \big) \D_{z} v
\\\nonumber&\quad- \D_{z} Z^\alpha \eta \big( v \cdot \N - \partial_{t} \eta\big)
 \D_{z}v
+ \mathcal{C}^\alpha(\mathcal{T}) \\
\nonumber & =  \big( \D_{t}+ v \cdot \nabla^\varphi\big)Z^\alpha v - \D_{z}v \big( \D_{t}+ v \cdot \nabla^\varphi) Z^\alpha \eta
 +  \mathcal{C}^\alpha(\mathcal{T})
 \\
\nonumber & =  \big( \D_{t}+ v \cdot \nabla^\varphi\big)V^\alpha + \D_{z} \( \D_{t}+ v \cdot \nabla^\varphi\)v Z^\alpha \eta
-\D_{z} v \cdot \nabla^\varphi v Z^\alpha \eta
 +  \mathcal{C}^\alpha(\mathcal{T}),
 \end{align}
 where the commutator $\mathcal{C}^\alpha(\mathcal{T})$ is defined by
 \begin{align}\label{calphat}
 \mathcal{C}^\alpha(\mathcal{T})= &  \[Z^\alpha, v_{y}\]\partial_{y} v+ \[Z^\alpha, V_{z},\partial_{z}v\]+
   \[Z^\alpha, v_{z}, {1 \over \partial_{z} \varphi}\] \partial_{z} v+ {1 \over \partial_{z}\varphi} \[Z^\alpha, v
   \]\cdot\N\partial_{z}v\\\nonumber
       & +v_z\partial_zv\[Z^{\alpha-\alpha'}, { 1 \over (\partial_{z} \varphi)^2}\]Z^{\alpha'}\partial_{z} \eta+ V_{z} \[Z^\alpha,
       \partial_{z}\]v+{ v_z\partial_zv \over (\partial_{z} \varphi)^2}\[Z^{\alpha},\pa_z\]  \eta
 \end{align}
 for any $\alpha'<\alpha$ with $|\alpha'|=1$.

  It remains to  compute $ \eps Z^\alpha \Delta^\varphi v=  2\eps Z^\alpha \nabla^\varphi \cdot  (S^\varphi v).$ Note that
   $$ 2  Z^\alpha \nabla^\varphi \cdot  (S^\varphi v) =  2 \nabla^\varphi \cdot \big( Z^\alpha \, S^\varphi v \big) - 2 \big( \D_{z}\, S^\varphi
   v \big) \nabla^\varphi
     ( Z^\alpha \eta) + \mathcal{D^\alpha}\big( S^\varphi v \big)$$
with $ \mathcal{D}^\alpha\big( S^\varphi v \big)_{i}= 2 \,  \mathcal{C}_{j}^\alpha \big( S^\varphi v)_{ij},$ and
    \beq
    \label{comS} 2 Z^\alpha  \big( S^\varphi v \big) =  2 S^\varphi \big( Z^\alpha v \big) - \D_{z} v \otimes \nabla^\varphi Z^\alpha \eta
     - \nabla^\varphi Z^\alpha \eta
  \otimes \D_{z} v  + \mathcal{E}^\alpha(v)\eeq
     with $\big( \mathcal{E}^\alpha (v)\big)_{ij}= \mathcal{C}^\alpha_{i}(v_{j})+ \mathcal{C}^\alpha_{j}(v_{i}).$
Hence one deduces that
     \begin{align}
     \label{deltaexp}
       \eps Z^\alpha \Delta^\varphi v  & = 2 \, \eps\, \nabla^\varphi \cdot  S^\varphi(Z^\alpha v )  - 2 \eps \nabla^\varphi \cdot \Big(  \D_{z}
       v \otimes \nabla^\varphi Z^\alpha \eta
     - \nabla^\varphi Z^\alpha \eta  \otimes \D_{z} v \Big)\\
    \nonumber      &\quad -
      2  \eps \big( \D_{z}\, S^\varphi v \big) \nabla^\varphi
     ( Z^\alpha \eta)  +   \eps  \mathcal{D^\alpha}\big( S^\varphi v \big) + \eps  \nabla^\varphi \cdot \big( \mathcal{E}^\alpha( v)\big)
     \\
    \nonumber      & =2 \, \eps\, \nabla^\varphi \cdot  S^\varphi V^\alpha -
      2  \eps \big( \D_{z} \nabla^\varphi\cdot S^\varphi v \big)
      Z^\alpha \eta   +   \eps  \mathcal{D^\alpha}\big( S^\varphi v \big) + \eps  \nabla^\varphi \cdot \big( \mathcal{E}^\alpha(
      v)\big).\end{align}
Consequently, these and  \eqref{NSv}  imply \eqref{eqValpha}--\eqref{divValpha}.

Now we estimate these commutators. First, thanks to Lemma \ref{comi},  the estimates \eqref{Cq}--\eqref{CE} hold.
To estimate the  commutator $\mathcal{C}^\alpha(\mathcal{T})$ defined by \eqref{calphat}, one needs to bound $v_{z}$ and  $V_{z}$. It follows
from  \eqref{etainfty} that
 \begin{align}\label{vz1}
 \norm{v_{z}}_{\Y^{{\[\!\frac{m}{2}\!\]}}} +\norm{V_{z}}_{\Y^{{\[\!\frac{m}{2}\!\]}}}  &\leq \Lambda\({1 \over c_{0}},  \norm{v }_{ \Y^{{\[\!\frac{m}{2}\!\]}}}+  \norm{\nabla
 \eta }_{ \Y^{{\[\!\frac{m}{2}\!\]}}} + \norm{  \partial_{t} \eta }_{ \Y^{{\[\!\frac{m}{2}\!\]}}} \)\nonumber
 \\&\leq \Lambda\({1 \over c_{0}},  \norm{v }_{ \Y^{{\[\!\frac{m}{2}\!\]}}}+
 \abs{  h }_{ \Y^{{\[\!\frac{m}{2}\!\]}+1}} \).\end{align}
And  \eqref{gues}, \eqref{etaharm}, and \eqref{etainfty} yield that
     \begin{align}
 \label{Zpetitvz}
      \norm{Z v_{z}}_{\X^{m-2}} &  \lesssim   \norm{Z \partial_{t} \eta }_{\X^{m-2}}+ \(1+  \norm{\nabla \eta}_{\Y^{{\[\!\frac{m}{2}\!\]-1}}}
      \)\norm{v}_{\X^{m-1}} + \norm{v}_{\Y^{{\[\!\frac{m}{2}\!\]-1}}} \norm{\nabla \eta }_{\X^{m-1}}\\
   \nonumber & \le
       \Lambda\( \norm{v}_{\Y^{{\[\!\frac{m}{2}\!\]-1}}} +  \abs{h}_{\Y^{{\[\!\frac{m}{2}\!\]}}}   \) \(\norm{v}_{\X^{m-1}} +\abs{h }_{\X^{m ,-\hal}} \).\end{align}
With \eqref{Zpetitvz} and \eqref{vz1} in hand, by using \eqref{gues}, \eqref{quot}, \eqref{etaharm} and  \eqref{etainfty}, one can obtain that
\begin{align}\label{Zvzm-2prov}
\norm{ Z V_{z}}_{\X^{m-2}} &\lesssim \norm{Z\( {1\over \partial_{z} \varphi} \)  v_{z}}_{\X^{m-2}} +  \norm{{1\over \partial_{z} \varphi} Zv_{z}
}_{\X^{m-2}}
\\\nonumber&\leq \Lambda\({1 \over c_{0}},  \norm{v }_{ \Y^{{\[\!\frac{m}{2}\!\]}}}+  \abs{  h }_{ \Y^{{\[\!\frac{m}{2}\!\]}+1}} \)\(\norm{v}_{\X^{m-1}} +\abs{h
}_{\X^{m ,-\hal}} \).
\end{align}
Consequently, one may use \eqref{com}, \eqref{comsym}, \eqref{idcom} and \eqref{quot} combined with \eqref{vz1}--\eqref{Zvzm-2prov}
and also again \eqref{etaharm}, \eqref{etainfty}, similarly as in the proof of Lemma \ref{comi}, to conclude the estimate \eqref{CT}.
  \end{proof}

\subsection{Boundary conditions}

We shall now also compute the boundary conditions satisfied   by $(Z^\alpha v,Z^\alpha q,Z^\alpha h)$ when $\alpha_{3} = 0$ (for $\alpha_{3}\neq
0$,
$Z^\alpha v= 0$ on the boundary). As a preliminary, one has the following.

\begin{lem} \label{lembord}
For $k\in \mathbb{N}$:
\beq\label{dzvb}
\abs{\nabla v }_{\X^{k }}   \leq    \Lambda\( {1 \over c_{0}}, \abs{h }_{\Y^{{\[\!\frac{k}{2}\!\]}+1}}   + \norm{  \nabla v }_{\Y^{{\[\!\frac{k}{2}\!\]}}} \)  \(
\abs{  v }_{\X^{k, 1}}+\abs{h }_{\X^{k, 1}} \)
\eeq
and
\beq\label{dzvb2}
\abs{\nabla v }_{\X^{k,s}}\le\Lambda\( {1 \over c_{0}}, \abs{h }_{\Y^{{\[\!\frac{k}{2}\!\]}+2}}+\norm{\nabla v}_{\Y^{{\[\!\frac{k}{2}\!\]}+1}}\)\(\abs{v
}_{\X^{k,s+1}}+\abs{h }_{\X^{k, s+1}} \)\text{ for }s=-\hal,\hal.
\eeq
\end{lem}
\begin{proof}
Note that it suffices to prove the estimates for $\pa_z v$. Since $\nabla^\varphi \cdot  v=0$, thus
\beq\label{eqdzvn}
\partial_{z} \varphi \(\partial_{1} v_{1}+ \partial_{2} v_{2}\)+\partial_{z} v \cdot  \N=0.
\eeq
Then for $s=-\hal,\hal$,  \eqref{gues2} implies that
\begin{align}\label{dzvnb1}
\abs{\partial_{z}v \cdot \n}_{\X^{k,s}} &\le  \Lambda\( {1 \over c_{0}}, \abs{\nabla \eta}_{\Y^{{\[\!\frac{k}{2}\!\]}+1}}   + \norm{  \nabla_y v
}_{\Y^{{\[\!\frac{k}{2}\!\]}+1}} \)
 \(\abs{v}_{\X^{k,s+1}}+  \abs{\nabla \eta}_{\X^{k,s}}\)
 \\\nonumber&\le   \Lambda\( {1 \over c_{0}}, \abs{h}_{\Y^{{\[\!\frac{k}{2}\!\]}+2}}   + \norm{  \nabla_y v }_{\Y^{{\[\!\frac{k}{2}\!\]}+1}} \)
 \(\abs{v}_{\X^{k,s+1}}+  \abs{h}_{\X^{k,s+1}}\).
 \end{align}
where the second inequality follows from Lemma \ref{propeta} and the trace estimate \eqref{trace}.

To bound $\Pi \partial_{z}v $, we shall use the boundary conditions in \eqref{NSv} which yield
\beq \label{bordpi}
\Pi\(S^\varphi v \n-\kappa\chi v\)=0.
\eeq
To compute $\Pi \(S^\varphi v\n\)$, one can use the local basis $(\partial_{y^1}, \partial_{y^2}, \partial_{y^3})$ in $\Omega_{t}$ induced by
\eqref{diff}. The induced riemannian metric is given by $g_{ij}= \partial_{y^i}\cdot \partial_{y^j}$, whose inverse denoted by $g^{ij}$. It
follows
from the definition and \eqref{vdef} that $(\partial_{y^i}u)(t,  \Phi(t, \cdot))= \partial_{i}v.$ Hence,
\beq\label{Su}
2 Su \n=\n\cdot \nabla u+\nabla u_k\n_k= \partial_{\n}u + g^{ij} \partial_{y^j} u \cdot \n  \partial_{y^i}= \partial_{\n}u + g^{ij}  \partial_{j}
v \cdot \n  \partial_{y^i}.
\eeq
Note also that
\beq\label{eqdzvPi0}
\partial_{\n}u   =\frac{\N}{\abs{\N}}\cdot  \nabla^\varphi v
= {\abs{\N} \over \partial_{z} \varphi} \partial_{z}v- \partial_{1} \varphi \partial_{1} v - \partial_{2}\varphi \partial_{2} v,
\eeq
one then gets from \eqref{bordpi} that
\beq\label{eqdzvPi'}
\Pi \partial_{z}v={\partial_{z} \varphi \over \abs{\N}}\(\partial_{1}\varphi\Pi \partial_{1} v + \partial_{2}\varphi \Pi \partial_{2} v-g^{ij}
\partial_{j}v \cdot \n \Pi \partial_{y^i}-\kappa\chi\Pi v \).
\eeq
Hence, using \eqref{gues2} again shows that
\begin{align}\label{dzvnb2}
\abs{\Pi \partial_{z}v }_{\X^{k,s}} &\le  \Lambda\( {1 \over c_{0}}, \abs{\nabla \varphi}_{\Y^{{\[\!\frac{k}{2}\!\]}+1}}   + \norm{  \nabla  v
}_{\Y^{{\[\!\frac{k}{2}\!\]}+1}} \)
 \(\abs{v}_{\X^{k,s+1}}+  \abs{h}_{\X^{k,s+1}}+  \abs{\partial_{z}v \cdot \n}_{\X^{k,s}} \)
 \\\nonumber&\le   \Lambda\( {1 \over c_{0}}, \abs{h}_{\Y^{{\[\!\frac{k}{2}\!\]}+2}}   + \norm{  \nabla  v }_{\Y^{{\[\!\frac{k}{2}\!\]}+1}} \)
 \(\abs{v}_{\X^{k,s+1}}+  \abs{h}_{\X^{k,s+1}}\).
 \end{align}
where in the second inequality \eqref{dzvnb1} has been used.

Consequently, we conclude the estimate \eqref{dzvb2} by combining \eqref{dzvnb1} and \eqref{dzvnb2}. And \eqref{dzvb} follows similarly by using
\eqref{gues}
instead of \eqref{gues2}.
   \end{proof}

We now study the dynamic boundary condition on $\{z=0\}$ and the Navier slip boundary condition on $\{z=-b\}$.
  \begin{lem}
  \label{lembordV}
  For $1 \leq |\alpha | \leq m$ such that $\alpha_{3}= 0$, it holds that on $\{z=0\}$
    \begin{align}\label{bordV}
  &2 \eps S^\varphi V^\alpha\, \N  -\( Z^\alpha q  -  g  Z^\alpha h+\sigma Z^\alpha H \)\N
   \\ \nonumber&\quad= - 2 \eps S^\varphi v \Pi Z^\alpha \N  - 2 \eps Z^\alpha h \D_{z}\big( S^\varphi v \big)\N+\eps
   \mathcal{C}^\alpha(\mathcal{B}_\eps) ,
  \end{align}
  where the commutator $\mathcal{C}^\alpha(\mathcal{B}_\eps)$ satisfies the estimate:
  \beq
  \label{CalphaB}
 \abs{ \mathcal{C}^\alpha(\mathcal{B}_\eps) }_{0} \leq    \Lambda\( {1 \over c_{0}}, \abs{ h }_{\Y^{{\[\!\frac{m}{2}\!\]}+1}}   + \norm{  \nabla v
 }_{\Y^{{\[\!\frac{m}{2}\!\]}}} \) \( \abs{  v }_{\X^{m-1,1}} + \abs{ h}_{\X^{m-1,1}}\)
    .
  \eeq
Similarly, on $\{z=-b\}$ one has that
  \beq\label{bord222}
  V^\alpha_3 = 0,\quad (S^\varphi V^\alpha e_3)_i=  \kappa V^\alpha_i-\mathcal{E}^\alpha(v)_{i3},\ i=1,2,  \eeq
  and
\beq\label{eeest}   \abs{\mathcal{E}^\alpha(v) }_{0} \leq  \Lambda\( {1 \over c_{0}},  \abs{h}_{\Y^{{\[\!\frac{m}{2}\!\]}+1}}   +\norm{  \nabla v
}_{\Y^{{\[\!\frac{m}{2}\!\]}}} \) \( \abs{ v }_{\X^{m-1,1}} +  \abs{h}_{\X^{m-1,1}}\).\eeq
  \end{lem}

 \begin{proof}
 Applying $Z^\alpha$ to the dynamic boundary condition and using \eqref{comS}, one gets
 \begin{align}
 \label{bordV1}
  \nonumber&\eps \( 2  S^\varphi \( Z^\alpha v \)  -  \D_{z} v \otimes \nabla^\varphi Z^\alpha \varphi
     -  \nabla^\varphi Z^\alpha v \otimes \D_{z} v   + \eps \mathcal{E}^\alpha(v) \) \N- \( Z^\alpha q - g Z^\alpha h+\sigma Z^\alpha H\) \N
     \\&\quad =-
     \( 2 \eps S^\varphi v - (q- gh+\sigma H )I \) Z^\alpha \N   - \[Z^\alpha,2\eps S^\varphi v-(q-gh+\sigma H)I, \N\].\nonumber
     \\&\quad =- 2 \eps\(  S^\varphi v - S^\varphi v \n \cdot \n I \) Z^\alpha \N   - 2\eps \[Z^\alpha, S^\varphi v-S^\varphi v \n \cdot \n I,
     \N\]
     \\&\quad\equiv- 2 \eps  S^\varphi v \Pi  Z^\alpha \N   - 2\eps \[Z^\alpha,   S^\varphi v \Pi  , \N\].\nonumber\end{align}
This yields \eqref{bordV} with the commutator $\mathcal{C}^\alpha(\mathcal{B}_\eps)$ defined by
   \beq
 \mathcal{C}^\alpha(\mathcal{B}_\eps) =- \mathcal{E}^\alpha(v)  \N- 2  \[Z^\alpha,   S^\varphi v \Pi  , \N\].
  \eeq
Similarly, applying $Z^\alpha$ to the Navier slip boundary condition and noting that $V^\alpha=Z^\alpha v$, one shows \eqref{bord222}.

Now, it follows from \eqref{comsym}, \eqref{etainfty} and Lemma \ref{lembord} that
   \begin{align}\label{CBe1}
  \abs{\[Z^\alpha,   S^\varphi v \Pi  , \N\]}_{0}
  &  \lesssim  \abs{ Z(   S^\varphi v \Pi  )}_{\X^{m-2}} \abs{Z \N}_{\Y^{{\[\!\frac{m}{2}\!\]-1}}}+ \abs{Z(   S^\varphi v \Pi  )}_{\Y^{{\[\!\frac{m}{2}\!\]-1}}}
  \abs{Z \N}_{\X^{m-2}}
    \\\nonumber& \le \Lambda\( {1 \over c_{0}}, \norm{  \nabla \eta}_{\Y^{{\[\!\frac{m}{2}\!\]}}}   + \norm{  \nabla v }_{\Y^{{\[\!\frac{m}{2}\!\]}}} \) \(
    \abs{\nabla v }_{\X^{m-1}} + \abs{h}_{\X^{m-1,1}}\)
    \\\nonumber&   \le \Lambda\( {1 \over c_{0}}, \abs{h }_{\Y^{{\[\!\frac{m}{2}\!\]}+1}}   + \norm{  \nabla v }_{\Y^{{\[\!\frac{m}{2}\!\]}}} \) \(  \abs{  v
    }_{\X^{m-1,1}}  + |h|_{\X^{m-1,1}}\).
      \end{align}
On the other hand, following the proof of Lemma \ref{comi} and using again Lemma \ref{lembord}, one has
\begin{align}\label{CBe2}\abs{\mathcal{E}^\alpha(v)}_{0}& \leq  \Lambda\( {1 \over c_{0}}, \norm{\nabla\eta}_{\Y^{{\[\!\frac{m}{2}\!\]}}}   +\norm{  \nabla
v }_{\Y^{{\[\!\frac{m}{2}\!\]}}} \) \(  \abs{\nabla v}_{\X^{m-1}} +  \abs{\nabla \eta}_{\X^{m-1}}\)
 \\\nonumber&\leq  \Lambda\( {1 \over c_{0}},  \abs{h}_{\Y^{{\[\!\frac{m}{2}\!\]}+1}}   +\norm{  \nabla v }_{\Y^{{\[\!\frac{m}{2}\!\]}}} \) \( \abs{ v }_{\X^{m-1,1}}
 +  \abs{h}_{\X^{m-1,1}}\).
 \end{align}
Consequently, the estimates \eqref{eeest} and \eqref{CalphaB} follows.
 \end{proof}

Finally, we study the kinematic boundary condition on $\{z=0\}$.
\begin{lem}
\label{lembordh}
  For $1 \leq |\alpha | \leq m$ such that $\alpha_{3}= 0$, it holds that on $\{z=0\}$
  \beq
  \label{bordh}
  \partial_{t} Z^\alpha h + v_y \cdot\nabla_y Z^\alpha h - V^\alpha \cdot \N =-\D_zv\cdot\N Z^\alpha h+\mathcal{C}^\alpha(h)
  \eeq
  where the commutator $\mathcal{C}^\alpha(h)$ satisfies
  \beq
  \label{bordhC}
 \abs{ \mathcal{C}^\alpha(h)}_0\leq \Lambda\( {1 \over c_{0}},      \abs{h}_{\Y^{{\[\!\frac{m}{2}\!\]}+1}} + \norm{v}_{\Y^{{\[\!\frac{m}{2}\!\]}}}\) \(
 \abs{h}_{\X^{m-1,1}} +\abs{ v}_{\X^{m-1}}
\) .
\eeq
Moreover,
\beq\label{chtilde}
\mathcal{C}^\alpha(h)= \sum_{|\alpha'|=1}C_{\alpha}^{\alpha'}Z^{\alpha-\alpha'} v_y\cdot \nabla_{y} Z^{ \alpha'}h +
\tilde{\mathcal{C}}^\alpha(h),
\eeq
 and $\tilde{\mathcal{C}}^\alpha(h)$  satisfies the estimate:
  \beq
  \label{bordhC1}
 \abs{ \tilde{\mathcal{C}}^\alpha(h)}_1\leq \Lambda\( {1 \over c_{0}},      \abs{h}_{\Y^{{\[\!\frac{m}{2}\!\]}+2}} +  \abs{v}_{\Y^{{\[\!\frac{m}{2}\!\]}+1}}\) \(
 \abs{h}_{\X^{m-1,2}} +\abs{ v}_{\X^{m-2,1}}
\) .
\eeq
\end{lem}
\begin{proof}
Applying $Z^\alpha$ to the kinematic boundary condition yields
\beq
  \partial_{t} Z^\alpha h + v_y \cdot\nabla_y Z^\alpha h -Z^\alpha v \cdot \N=    {\mathcal{C}}^\alpha(h)
  \eeq
where
\beq
{\mathcal{C}}^\alpha(h)=- \[ Z^\alpha, v_{y}, \nabla_{y}h\].
  \eeq
This yields \eqref{bordh}, and the estimate \eqref{bordhC} follows from \eqref{comsym}.
One may further single out the highest order derivative terms according to \eqref{chtilde}, where $\tilde{\mathcal{C}}^\alpha(h)$ is defined by
\beq\label{tildech} \tilde{\mathcal{C}}^\alpha(h) = -\sum_{\beta+\gamma=\alpha\atop \beta\neq0,|\gamma|\ge 2}C_{\beta,\gamma} Z^{\beta} v_y\cdot
\nabla_{y} Z^{ \gamma}h  .\eeq
And the estimate \eqref{bordhC1} follows by using \eqref{gues}.
\end{proof}

\section{Pressure estimates}\label{sectionpressure}

In view of the equation \eqref{eqValpha}, one needs to estimate the pressure $q$. The first equation in \eqref{NSv} implies that
\beq\label{eqq} \Delta^\varphi q = - \nabla^\varphi \cdot (v\cdot \nabla^\varphi v) \text{ in }\Omega.\eeq
 Moreover, the dynamic boundary condition gives
\beq\label{boundary0d} q  =  2 \eps S^\varphi v  \,{\bf n} \cdot \,{\bf n} + gh-\sigma H \text{ on }\{z=0\}.
\eeq
Projecting the first equation in \eqref{NSv} along $\N$ onto $\{z=0\}$ and $\{z=-b\}$ yields
\beq\label{boundary0}\nabla^\varphi q\cdot\N = -\dt v\cdot \N-(v_y\cdot\nabla_y) v\cdot \N+\eps   \Delta^\varphi v \cdot\N \text{ on }\{z=0\}
\eeq
and
\beq\label{boundaryb}\nabla^\varphi q\cdot\N =  \eps   \Delta^\varphi v \cdot\N \text{ on }\{z=-b\}.\eeq
Here in \eqref{boundaryb} one has used the fact that $\N=e_3$ and $v_3=0$ on $\{z=-b\}$.

Note that to solve the pressure, one has two choices of boundary conditions on $\{z=0\}$, $i.e.,$ \eqref{boundary0d} and \eqref{boundary0}.
Without surface tension, one can use the elliptic problem \eqref{eqq}, \eqref{boundary0d} and \eqref{boundaryb} to establish the regularity
estimates for $q$. The subtlety lies in that the energy dissipation estimates of \eqref{NSv} in the case without surface tension  provide the
needed
estimates for those boundary terms. When there is surface tension, however, the energy dissipation estimates do not provide enough
estimates for the boundary term $ -\sigma H$ (which is of one half regularity less). This would suggest that the elliptic problem \eqref{eqq},
\eqref{boundary0d} and \eqref{boundaryb} is not the right choice for estimating the pressure $q$ in the case with surface tension. Our way to get
around this difficulty is to use instead the elliptic problem \eqref{eqq}, \eqref{boundary0} and \eqref{boundaryb}. It is then noticed that this
approach forces one to estimate the time derivatives of $v$, that is, one needs to perform energy estimates for the time derivatives of the
solution.
However, there is an essential difficulty arising: when doing energy estimates with time derivatives up to $m$ order,  we can only obtain the
estimates of time derivatives of $q$ up to $m-1$ order due to the presence of $\dt v$ in \eqref{boundary0}. Thus the energy estimates cannot be
closed since it seems that $m$ order time derivative of $q$ is involved. We will explain this and our way to overcome it in more details in
Section
\ref{secconormal2}.

It follows from the definition of $\partial_i^\varphi$ that
 \beq
 \label{graddiv} \nabla^\varphi \cdot v= {1 \over \partial_{z}\varphi}  \nabla\cdot \big( P v \big) , \quad  \nabla^\varphi f= {1 \over
 \partial_{z} \varphi} P^* \nabla f,
  \quad P= \left( \begin{array}{ccc}  \partial_{z} \varphi & 0 & 0 \\ 0 & \partial_{z} \varphi & 0 \\ -\partial_{1} \varphi  & - \partial_{2}
  \varphi & 1 \end{array}
   \right).\eeq
And then $\Delta^\varphi$ can be expressed as
\beq
\label{deltaphi} \Delta^{\varphi}f = \nabla^\varphi \cdot (\nabla^\varphi f) = {1 \over \partial_{z} \varphi} \nabla \cdot \big( E \nabla f
\big)\eeq
 with the matrix $E$ defined by
 $$ E =  \frac1{\partial_{z} \varphi} PP^* \equiv \left( \begin{array}{ccc} \partial_{z} \varphi & 0 & {- \partial_{1} \varphi} \\
  0 & \partial_{z} \varphi & -\partial_{2} \varphi  \\ -\partial_{1} \varphi & -\partial_{2} \varphi &
   {1 + (\partial_{1} \varphi)^2} + (\partial_{2} \varphi)^2 \over \partial_{z} \varphi \end{array} \right).
  $$
   Note that $E$ is symmetric positive and that if
   $ \| \nabla_y \varphi \|_{L^\infty} \leq { 1 \over c_{0}} $ and
 $  \partial_{z} \varphi \geq c_{0}>0$ then
   there exists  $\delta(c_{0})>0$ such that
   \beq
   \label{Eminor}
   E X \cdot X \geq \delta |X|^2, \quad \forall X \in \mathbb{R}^3.
   \eeq
   Moreover,
   \beq
   \label{ELinfty}
\norm{ E}_{\tilde{\Y}^{k}} \leq \Lambda\({1 \over c_{0}}, \abs{h}_{\Y^{k+1}}\).
   \eeq
One can write
   $$ E= \mbox{Id}  + \tilde{E}, \quad
\tilde  E= \left( \begin{array}{ccc}  \partial_{z}\eta & 0 & {- \partial_{1} \eta} \\
  0 &  \partial_{z}\eta  & -\partial_{2} \eta  \\ -\partial_{1} \eta & -\partial_{2} \eta &
   { (\partial_{1} \eta)^2} + (\partial_{2} \eta)^2  - \partial_{z} \eta  \over \partial_{z} \varphi \end{array} \right), $$
where
 \beq
 \label{EHs} \norm{ \tilde{E}}_{\tilde{\X}^{k}} \leq \Lambda\({1 \over c_{0}}, \abs{h}_{\Y^{{\[\!\frac{k}{2}\!\]}+1}} \) \abs{h}_{\X^{k,\hal}} .
 \eeq
Here $\tilde{\X}^k$ and $\tilde{\Y}^k$ are referred to the usual {\it spatial-time} Sobolev spaces as defined similarly as \eqref{xym}.

Since $ \N=  P^\ast e_3$ on $\{z=0,-b\}$, the equations \eqref{eqq}--\eqref{boundaryb} can be rewritten as
 \begin{alignat}{3}
 \label{eqq12} &- \nabla \cdot \big( E \nabla q\big) =F:=  \partial_{z}\varphi \nabla^\varphi v \cdot \nabla^\varphi v  &&\quad\text{in
 }\Omega,\\
\label{boundary0d12}& q  =G^1:=  2 \eps S^\varphi v  \,{\bf n} \cdot \,{\bf n} + gh-\sigma H &&\quad\text{on }\{z=0\},
\\\label{boundary012}&{E\nabla q\cdot e_3} =G^2:= -\dt v\cdot \N-(v_y\cdot\nabla_y) v\cdot \N+\eps   \Delta^\varphi v \cdot\N &&\quad\text{on
}\{z=0\},
\\\label{boundaryb12}&{E\nabla q\cdot e_3} =G^3:= \eps   \Delta^\varphi v \cdot\N &&\quad\text{on }\{z=-b\}.
\end{alignat}

We shall now prove the estimates for the pressure $q$.
  \begin{prop}\label{proppressure}
The following estimates hold:
    \begin{align}\label{qpressureest}
  &    \norm{   q }_{\X^{k}}+\norm{ \nabla q }_{\X^{k}}+\norm{ \pa_{zz} q}_{\X^{k-1}}
  \\\nonumber&  \quad\leq \Lambda\( {1 \over c_{0}},   \abs{h}_{\Y^{{\[\!\frac{k}{2}\!\]}+3}} +\abs{h}_{{\X^{{{\[\!\frac{k+7}{2}\!\]}}}}}  +\norm{  v
  }_{\Y^{{\[\!\frac{k+3}{2}\!\]}}}  +\norm{\nabla v}_{\Y^{{\[\!\frac{k}{2}\!\]}+2}} +\norm{v}_{\X^{{\[\!\frac{k+7}{2}\!\]}}}  + \norm{\nabla v}_{\X^{{\[\!\frac{k+5}{2}\!\]}}}  \)
  \\\nonumber&\qquad\times\(\abs{h}_{\X^{k+1,-{1 \over 2 }}}+\sigma \abs{h}_{\X^{k,2}}+\norm{v}_{\X^{k+1}}  + \norm{\nabla v}_{\X^{k}} +
  \eps\abs{  v }_{{\X^{{k}, \frac{3}{2}}}}+\eps\abs{h}_{{\X^{{k},\frac{3}{2}}}}\), \text{ for } k\ge 3;
  \end{align}
\begin{align} \label{qinfty}
&\norm{   q }_{\Y^{k}}+\norm{ \nabla q }_{\Y^{k}}+ \norm{ \pa_{zz} q }_{\Y^{k-1}}
\\\nonumber&\quad\leq \Lambda\( {1 \over c_{0}},   \abs{h}_{\Y^{{\[\!\frac{k}{2}\!\]}+4}}  +\abs{h}_{\X^{k+4}}+\norm{  v
}_{\Y^{{\[\!\frac{k+5}{2}\!\]}}}+\norm{\nabla v}_{\Y^{{\[\!\frac{k}{2}\!\]}+3}} +  \norm{v}_{\X^{k+4}}  + \norm{\nabla v}_{\X^{k+3}}   \), \text{ for }k\ge 1.
\end{align}
\end{prop}
     \begin{proof}
Multiplying the equation \eqref{eqq12} by $q$  and then integrating by parts over $\Omega$, using
\eqref{boundary012} and \eqref{boundaryb12}, one obtains
  \beq
  \(E \nabla q,\nabla q\)_{\Omega} =\(F, q\)_{\Omega}+\(G^2, q\)_{z=0}-\(G^3, q\)_{z=-b}.
  \eeq
Here $\(\cdot,\cdot\)_{\Omega},\ \(\cdot,\cdot\)_{z=0}$ and $ \(\cdot,\cdot\)_{z=-b}$ denote the $L^2$ inner products on $\Omega$, $\{z=0\}$ and
$\{z=-b\}$, respectively. It follows from the trace estimate $\abs{q}_{\hal}\ls \norm{  q }_{H^1}$ and Cauchy's inequality that
  \beq
  \label{rhoH1'0}
  \norm{\nabla q }_{0} ^2 \leq {1 \over \eta}\Lambda_0  \(\norm{F}^2+\abs{G^1}_{0}^2+\abs{G^2}_{-\hal}^2+\abs{G^3}_{-\hal}^2\)+\eta\norm{  q
  }_{H^1}^2
  \eeq
  for any $\eta>0.$
 By the Poincar\'e inequality \eqref{poin} and \eqref{boundary0d12}, one can get by taking $\eta$
  sufficiently small that
    \beq
  \label{rhoH1'}
 \norm{  q }_{H^1}  \leq  \Lambda_0  \(\norm{F} +\abs{G^1}_{0} +\abs{G^2}_{-\hal} +\abs{G^3}_{-\hal} \).
  \eeq
Note that if one uses solely the problem \eqref{eqq}, \eqref{boundary0d} and \eqref{boundaryb} to estimate $p$, then one needs
$\abs{G^1}_{\hal}$.
This half less regularity requirement enables us to control the surface tension term by the energy dissipation estimates.

   Next, applying $Z^\alpha$ with $|\alpha|=k$ to the equation \eqref{eqq12} and using \eqref{boundary012}--\eqref{boundaryb12} lead to
  \beq
  \(Z^{\alpha}(E \cdot \nabla q),\nabla Z^{\alpha} q\)_{\Omega} =\(Z^{\alpha} F, Z^{\alpha} q\)_{\Omega}+\(Z^{\alpha} G^2, Z^{\alpha}
  q\)_{z=0}-\(Z^{\alpha} G^3 Z^{\alpha} q\)_{z=-b}.
  \eeq
Then as for \eqref{rhoH1'}, one can derive after using \eqref{boundary0d12} that
\begin{align}
       \label{elprov3}
      \norm{   q }_{\X^k}+\norm{ \nabla  q }_{\X^k}\leq \Lambda_0  & \(
      \norm{F}_{\X^k}+\abs{G^1}_{{\X^{k}}}+\abs{G^2}_{{\X^{k,-\hal}}}+\abs{G^3}_{{\X^{k,-\hal}}}\right.
        \\\nonumber& \ + \norm{ E\cdot \[Z^\alpha , \nabla\] q }_{0}  +  \norm{\[Z^\alpha,  E\]\cdot \nabla q}_{0} \Big).
      \end{align}
We then estimate the commutators in the right hand side of \eqref{elprov3}. First, \eqref{idcom} implies that
 \beq\label{al11}
  \norm{  E\cdot \[Z^\alpha , \nabla\] q}_{0}  \ls    \norm{ E}_{L^\infty}   \norm{  \nabla q }_{\X^{k-1}}\ls \abs{h}_{ \Y^{1}} \norm{\nabla q
  }_{\X^{k-1}} . \eeq
However, the other communtator needs more attentions according to $1\le k\le 3$ or $k\ge 4$. Indeed, for $1\le k\le 3$, direct estimates by
controlling the $\nabla q$ terms in $L^2$ and $E$ terms in $L^\infty$ yield
    \begin{align} \label{al1122}
\norm{ \[Z^\alpha,  E\]\cdot \nabla q} _{0}    \lesssim      \norm{ {E}}_{\tilde\Y^{k}} \norm{\nabla q }_{\X^{k-1}}
\lesssim      \abs{h}_{ \Y^{k+1}} \norm{\nabla q }_{\X^{k-1}} .
      \end{align}
Plugging \eqref{al11}--\eqref{al1122} into \eqref{elprov3}, by an induction argument and \eqref{rhoH1'}, one can deduce that
\begin{align}
\label{elprov31112}
     \nonumber \norm{   q }_{\X^k}+\norm{ \nabla  q }_{\X^k}&\leq \Lambda\({1 \over c_{0}},\norm{ {E}}_{\tilde\Y^{3}}  \)  \(
     \norm{F}_{\X^k}+\abs{G^1}_{{\X^{k}}}+\abs{G^2}_{{\X^{k,-\hal}}}+\abs{G^3}_{{\X^{k,-\hal}}}+\norm{\nabla q }_{\X^{k-1}}\)
      \\&\leq \Lambda\({1 \over c_{0}},\abs{ h}_{\Y^{4}}  \)  \(
      \norm{F}_{\X^k}+\abs{G^1}_{{\X^{k}}}+\abs{G^2}_{{\X^{k,-\hal}}}+\abs{G^3}_{{\X^{k,-\hal}}}\).
              \end{align}
Since the equation \eqref{eqq12}   gives
      \beq
      \label{eqdzzrho} \partial_{zz}  q ={1 \over E_{33}} \(    F -  \partial_{z}\( \sum_{j<3}E_{3,j}\partial_{j} q  \)-
       \sum_{i<3, \, j} \partial_{i}\( E_{i j} \partial_{j} q\) \),\eeq
in a similar way and by \eqref{elprov31112}, one can also obtain
 \begin{align}\label{q11}
\norm{ \pa_{zz} q}_{\X^{k-1}} &\leq \Lambda\({1 \over c_{0}},  \norm{ E }_{\tilde\Y^k}\)\(\norm{  F}_{\X^{k-1}}+\norm{\nabla q }_{\X^{k}}\)
 \\\nonumber&\leq \Lambda\({1 \over c_{0}},| h|_{\Y^{4}} \)\(
 \norm{F}_{\X^k}+\abs{G^1}_{{\X^{k}}}+\abs{G^2}_{{\X^{k,-\hal}}}+\abs{G^3}_{{\X^{k,-\hal}}}\).
 \end{align}
It then follows from \eqref{elprov31112} and \eqref{q11} that for $1 \le k\le 3$,
  \begin{align}\label{qclaim}
  &    \norm{   q }_{\X^k}+\norm{ \nabla q }_{\X^k}+\norm{ \pa_{zz} q}_{\X^{k-1}}
  \\\nonumber&  \quad\leq \Lambda \({1 \over c_{0}}, \abs{ h}_{\Y^{4}} +\abs{ h}_{\Y^{{\[\!\frac{k+3}{2}\!\]}}}+\abs{ h}_{\X^{{\[\!\frac{k+3}{2}\!\]},\hal}}+\norm{F}_{\X^{{\[\!\frac{k+3}{2}\!\]}}} +\abs{G^1}_{{\X^{{\[\!\frac{k+3}{2}\!\]}}}}+\abs{G^2}_{{\X^{{\[\!\frac{k+3}{2}\!\]},-\hal}}}+\abs{G^3}_{{\X^{{\[\!\frac{k+3}{2}\!\]},-\hal}}}\)
   \\\nonumber& \qquad\times\(\abs{ h}_{\X^{k,{1 \over 2}}}+
   \norm{F}_{\X^k}+\abs{G^1}_{{\X^{k}}}+\abs{G^2}_{{\X^{k,-\hal}}}+\abs{G^3}_{{\X^{k,-\hal}}}\).
  \end{align}

We now claim that \eqref{qclaim} holds for all $k\ge 1$. This will be proved by induction. Assume that $k\ge 4$
and that \eqref{qclaim} holds for $k-1$. The commutator estimate \eqref{com} yields that
\begin{align}
   \label{eses112}
   \norm{ \[Z^\alpha,  E\]\cdot \nabla q}_{0}  &  \lesssim  \norm{Z{E}}_{\X^{k-1}} \norm{ \nabla q }_{\Y^{{\[\!\frac{k-1}{2}\!\]}}} +
   \norm{Z{E}}_{\Y^{{\[\!\frac{k-1}{2}\!\]}}} \norm{\nabla q }_{\X^{k-1}}  \\
     \nonumber &   \le     \Lambda \({1 \over c_{0}}, \abs{ h}_{\Y^{{\[\!\frac{k+3}{2}\!\]}}}+\norm{ \nabla q
     }_{\Y^{{\[\!\frac{k-1}{2}\!\]}}}\)\(\abs{h}_{\X^{k,\hal}}+\norm{\nabla q }_{\X^{k-1}}\).
      \end{align}
Plugging \eqref{eses112} and \eqref{al11} into \eqref{elprov3} and using \eqref{eqdzzrho} again lead to
\begin{align}
       \label{elprov31212}
      \norm{   q }_{\X^k}+\norm{ \nabla  q }_{\X^k}+\norm{  \pa_{zz} q}_{\X^{k-1}} \leq  & \Lambda \({1 \over c_{0}}, \abs{
      h}_{\Y^{{\[\!\frac{k+3}{2}\!\]}}}+\norm{ \nabla q }_{\Y^{{\[\!\frac{k-1}{2}\!\]}}}\) \Big(\abs{h}_{\X^{k,\hal}}+\norm{\nabla q }_{\X^{k-1}}
        \\\nonumber&\qquad \left. + \norm{F}_{\X^k}+\abs{G^1}_{{\X^{k}}}+\abs{G^2}_{{\X^{k,-\hal}}}+\abs{G^3}_{{\X^{k,-\hal}}}\).
      \end{align}
To remove the dependence of $\Lambda$ on $\norm{ \nabla q }_{\Y^{{\[\!\frac{k-1}{2}\!\]}}}$, one applies the anisotropic Sobolev embedding estimate
\eqref{emb} and the induction assumption to obtain
\begin{align}\label{gogogo}
&\norm{ \nabla q }_{\Y^{{\[\!\frac{k-1}{2}\!\]}}}\ls  \norm{ \partial_{z} \nabla q }_{\X^{{\[\!\frac{k+1}{2}\!\]}}} \norm{   \nabla q }_{\X^{{\[\!\frac{k+3}{2}\!\]}}}
\\\nonumber&\quad\le \Lambda \({1 \over c_{0}}, \abs{ h}_{\Y^{4}} +\abs{ h}_{\Y^{{\[\!\frac{\[\frac{k+3}{2}\]+3}{2}\!\]}}}+\abs{ h}_{\X^{{\[\!\frac{\[\frac{k+3}{2}\]+3}{2}\!\]},\hal}}+\norm{ F}_{\X^{{\[\!\frac{\[\frac{k+3}{2}\]+3}{2}\!\]}}} \right.
\\\nonumber&\qquad\qquad\left.+\abs{G^1}_{{\X^{{\[\!\frac{\[\frac{k+3}{2}\]+3}{2}\!\]}}}}+\abs{G^2}_{{\X^{{\[\!\frac{\[\frac{k+3}{2}\]+3}{2}\!\]},-\hal}}}+\abs{G^3}_{{\X^{{\[\!\frac{\[\frac{k+3}{2}\]+3}{2}\!\]},-\hal}}} \)
\\\nonumber&\qquad  \times\( \abs{ h}_{\X^{{\[\!\frac{k+3}{2}\!\]},{1 \over 2}}}+  \norm{F}_{\X^{{\[\!\frac{k+3}{2}\!\]}}}
+\abs{G^1}_{{\X^{{\[\!\frac{k+3}{2}\!\]}}}}+\abs{G^2}_{{\X^{{\[\!\frac{k+3}{2}\!\]},-\hal}}}+\abs{G^3}_{{\X^{{\[\!\frac{k+3}{2}\!\]},-\hal}}} \)
\\\nonumber&\quad\le \Lambda \({1 \over c_{0}}, \abs{ h}_{\Y^{4}} + \abs{ h}_{\Y^{{\[\!\frac{k+3}{2}\!\]}}}+\abs{ h}_{\X^{{\[\!\frac{k+3}{2}\!\]},{1 \over 2}}}+
\norm{F}_{\X^{{\[\!\frac{k+3}{2}\!\]}}} +\abs{G^1}_{{\X^{{\[\!\frac{k+3}{2}\!\]}}}}+\abs{G^2}_{{\X^{{\[\!\frac{k+3}{2}\!\]},-\hal}}}+\abs{G^3}_{{\X^{{\[\!\frac{k+3}{2}\!\]},-\hal}}} \)
.
\end{align}
Here one has used the fact that ${\[\!\frac{k+3}{2}\!\]}\le k-1$ since $k\ge 4$, which allows one to use the induction assumption. Plugging the estimate
\eqref{gogogo} into \eqref{elprov31212} and using the induction assumption to estimate $\norm{\nabla q }_{\X^{k-1}}$, one thus concludes
 \eqref{qclaim} for all $k\ge 4$.

We now estimate the right hand side of \eqref{qclaim} for $k\ge 3$.  It follows by the product estimate \eqref{gues} that
\begin{align}\label{Fest1}
 \norm{F}_{\X^{k}}&=
 \norm{\partial_{z}\varphi \nabla^\varphi v \cdot \nabla^\varphi v }_{\X^{k}}
 \\\nonumber&\leq   \Lambda \( {1 \over c_{0}}, \abs{h}_{\Y^{{\[\!\frac{k}{2}\!\]}+1}}+\norm{\nabla v}_{\Y^{{\[\!\frac{k}{2}\!\]}}}   \)
  \( \abs{h}_{\X^{k,\hal}}  + \norm{\nabla v}_{\X^{k}} \).
  \end{align}
Similarly, since $k\ge 3$,
\begin{align}\label{Fest2}
\norm{F}_{\X^{{\[\!\frac{k+3}{2}\!\]}}}
 &  \leq  \Lambda \( {1 \over c_{0}}, \abs{h}_{\Y^{ {{\[\!\frac{k+3}{4}\!\]}}+1}}+\norm{\nabla v}_{\Y^{{{\[\!\frac{k+3}{4}\!\]}}}}   \)
  \( \abs{h}_{\X^{{\[\!\frac{k+3}{2}\!\]} ,\hal}}  + \norm{\nabla v}_{\X^{{\[\!\frac{k+3}{2}\!\]}  }} \)
 \\\nonumber&\le \Lambda \( {1 \over c_{0}}, \abs{h}_{\Y^{{\[\!\frac{k}{2}\!\]}+1}}+\abs{h}_{\X^{{{\[\!\frac{k}{2}\!\]}}+2}}+\norm{\nabla v}_{\Y^{{\[\!\frac{k}{2}\!\]}}} +
 \norm{\nabla v}_{\X^{{{{\[\!\frac{k+3}{2}\!\]}}}}}  \).
  \end{align}
By \eqref{gues} and Lemma \ref{lembord},
   \begin{align}\label{G1est1}
 \abs{ G^1 }_{\X^{k}} &= \abs{2 \eps S^\varphi v  \,{\bf n} \cdot \,{\bf n} + gh-\sigma H}_{\X^{k}}
   \\\nonumber&\leq   \Lambda\( {1 \over c_{0}},  \abs{h}_{\Y^{{\[\!\frac{k}{2}\!\]}+2}} +\norm{\nabla v}_{\Y^{{\[\!\frac{k}{2}\!\]}}}  \) \( \eps \abs{ \nabla v
   }_{\X^{k}}
    +  | h |_{\X^{k}}+ \sigma | h |_{\X^{k,2}}\)  \\\nonumber
     & \leq   \Lambda\( {1 \over c_{0}},  \abs{h}_{\Y^{{\[\!\frac{k}{2}\!\]}+2}} +\norm{\nabla v}_{\Y^{{\[\!\frac{k}{2}\!\]}}}  \) \( \eps \abs{   v  }_{\X^{k,1}}
    +  \abs{h}_{\X^{k}}+ \sigma \abs{h}_{\X^{k,2}}\) .
    \end{align}
 Similarly, since $k\ge 3$,  the trace estimate implies that
    \begin{align}\label{G1est2}
\abs{G^1}_{{\X^{{\[\!\frac{k+3}{2}\!\]} }}}
     & \leq   \Lambda\( {1 \over c_{0}},  \abs{h}_{\Y^{{{\[\!\frac{k+3}{4}\!\]}}+2}} +\norm{\nabla v}_{\Y^{{{\[\!\frac{k+3}{4}\!\]}}}}  \) \( \eps \abs{   v
     }_{\X^{{\[\!\frac{k+3}{2}\!\]},1}}
    +  \abs{h}_{\X^{{\[\!\frac{k+3}{2}\!\]}}}+ \sigma \abs{h}_{\X^{{\[\!\frac{k+3}{2}\!\]},2}}\)
     \\\nonumber
     & \leq   \Lambda\( {1 \over c_{0}},  \abs{h}_{\Y^{{\[\!\frac{k}{2}\!\]}+2}} +  \abs{h}_{\X^{{\[\!\frac{k+7}{2}\!\]}}}+\norm{\nabla v}_{\Y^{{\[\!\frac{k}{2}\!\]}}}
     +\norm{   v  }_{\X^{{{\[\!\frac{k}{2}\!\]}}+3}}+\norm{ \nabla  v  }_{\X^{{{\[\!\frac{k}{2}\!\]}}+2}}\).
    \end{align}

To estimate the most delicate term $G^2$, we start with $\eps\Delta^\varphi v\cdot\N$. Note that
    \begin{align}
\Delta^\varphi v\cdot\N=2\(\nabla^\varphi\cdot  S^\varphi v\)\cdot\N = 2\nabla^\varphi\cdot \( S^\varphi v \N \)- 2   S^\varphi v: \nabla^\varphi
\N
    \end{align}
and
    \begin{align}
 \nabla^\varphi\cdot \( S^\varphi v \N \)=\pa_1 \( S^\varphi v \N \)_1+\pa_2 \( S^\varphi v \N \)_2+ \frac{1}{\pa_z\varphi}\pa_z\( S^\varphi v \N
 \)\cdot\N .
    \end{align}
Hence by the estimate \eqref{gues2} with $s=-\hal,\hal$, one can get
\begin{align}\label{g3es1}
\abs{\Delta^\varphi v\cdot\N}_{{\X^{{k},-\hal}}}&\ls \abs{\frac{1}{\pa_z\varphi}\pa_z\( S^\varphi v \N
\)\cdot\N}_{{\X^{{k},-\hal}}}+\abs{S^\varphi v \N}_{{\X^{{k},\hal}}}+ \abs{S^\varphi v: \nabla^\varphi \N }_{{\X^{{k},-\hal}}}
 \\\nonumber
     & \leq   \Lambda\( {1 \over c_{0}},  \abs{h}_{\Y^{{\[\!\frac{k}{2}\!\]}+3}} +\norm{\nabla v}_{\Y^{{\[\!\frac{k}{2}\!\]}+1}}+\norm{\pa_z\( S^\varphi v \N
     \)\cdot\N}_{\Y^{{\[\!\frac{k}{2}\!\]}+1}} \) \\\nonumber
     & \quad\times\(\abs{ \pa_z\( S^\varphi v \N \)\cdot\N}_{{\X^{{k},-\hal}}}+\abs{\nabla v
     }_{{\X^{{k},\hal}}}+\abs{h}_{{\X^{{k},\frac{3}{2}}}}\).
\end{align}
But
    \begin{align}
\pa_z\( S^\varphi v \N \)\cdot\N=\pa_z\( S^\varphi v \N\cdot\N \)-  S^\varphi v \N  \cdot \pa_z\N,
    \end{align}
and recalling the matrices $P$ and $E$,
    \begin{align}
 \pa_z\( S^\varphi v \N\cdot\N \)&=\pa_z\( \nabla^\varphi v \N\cdot\N \)=\pa_z\(  \frac{1}{\pa_z\varphi}P^\ast \nabla v P^\ast e_3\cdot P^\ast
 e_3 \)
 \\\nonumber&
  = \nabla^\varphi (\pa_z v) \N\cdot\N  - \pa_j v_i\pa_z(E_{3i}P_{3j}), \end{align}
one can then deduce from \eqref{g3es1} that
\begin{align}\label{g3es2}
\abs{\Delta^\varphi v\cdot\N}_{{\X^{{k},-\hal}}}& \leq   \Lambda\( {1 \over c_{0}},  \abs{h}_{\Y^{{\[\!\frac{k}{2}\!\]}+3}} +\norm{\nabla
v}_{\Y^{{\[\!\frac{k}{2}\!\]}+1}}+\norm{\nabla^\varphi (\pa_z v) \N\cdot\N }_{\Y^{{\[\!\frac{k}{2}\!\]}+1}} \) \\\nonumber
     & \quad\times\(\abs{ \nabla^\varphi (\pa_z v) \N\cdot\N }_{{\X^{{k},-\hal}}}+\abs{\nabla v
     }_{{\X^{{k},\hal}}}+\abs{h}_{{\X^{{k},\frac{3}{2}}}}\).
\end{align}
Note further that
    \begin{align}
 &  \nabla^\varphi (\pa_z v) \N\cdot\N  = \nabla^\varphi \(\pa_z v\cdot \N\)\cdot\N-\(\N\cdot\nabla^\varphi\)\N\cdot \pa_z v, \end{align}
and one can compute by using $\nabla^\varphi\cdot v = 0$ that
    \begin{align}
 &   \nabla^\varphi \(\pa_z v\cdot \N\)\cdot\N= -\nabla^\varphi \(\pa_z\varphi(\pa_1v_1+\pa_2 v_2)\)\cdot\N.\end{align}
It follows from these, \eqref{g3es2} and Lemma \ref{lembord} that
 \begin{align}\label{gg1es1}
\abs{\Delta^\varphi v\cdot\N}_{{\X^{{k},-\hal}}}& \leq   \Lambda\( {1 \over c_{0}},  \abs{h}_{\Y^{{\[\!\frac{k}{2}\!\]}+3}} +\norm{\nabla
v}_{\Y^{{\[\!\frac{k}{2}\!\]}+2}}  \)  \(\abs{\nabla v }_{{\X^{{k},\hal}}}+\abs{h}_{{\X^{{k},\frac{3}{2}}}}\)
\\\nonumber& \leq   \Lambda\( {1 \over c_{0}},  \abs{h}_{\Y^{{\[\!\frac{k}{2}\!\]}+3}} +\norm{\nabla v}_{\Y^{{\[\!\frac{k}{2}\!\]}+2}}  \)  \(\abs{  v }_{{\X^{{k},
\frac{3}{2}}}}+\abs{h}_{{\X^{{k},\frac{3}{2}}}}\).
\end{align}
Next, one easily has
 \begin{align}\label{gg1es2}
\abs{(v_y\cdot\nabla_y) v\cdot \N}_{{\X^{{k},-\hal}}}\leq   \Lambda\( {1 \over c_{0}},  \abs{h}_{\Y^{{\[\!\frac{k}{2}\!\]}+2}} +\norm{
v}_{\Y^{{\[\!\frac{k}{2}\!\]}+2}}  \)  \(\abs{  v }_{{\X^{{k}, \frac{1}{2}}}}+\abs{h}_{{\X^{{k},\frac{1}{2}}}}\).
\end{align}
Finally, we estimate the remaining time derivative term $\dt v\cdot\N$. One first has
 \begin{align}\label{gggesrt00}
 \abs{\dt v\cdot\N}_{{\X^{k-1,\hal}}} & \leq\Lambda\( {1 \over c_{0}}, \abs{\nabla h }_{\Y^{{\[\!\frac{k-1}{2}\!\]}+1 }}  + \norm{ \dt v
 }_{\Y^{{\[\!\frac{k-1}{2}\!\]}+1}}    \)
   \(\abs{\dt v}_{\X^{k-1,\hal}}   + \abs{h}_{\X^{k-1, {3 \over 2 }}}\)
\\\nonumber&\leq\Lambda\( {1 \over c_{0}}, \abs{  h }_{\Y^{{\[\!\frac{k+3}{2}\!\]} }}  + \norm{  v }_{\Y^{{\[\!\frac{k+3}{2}\!\]}}}    \)
   \(\abs{ v}_{\X^{k,\hal}}   + \abs{h}_{\X^{k-1, {3 \over 2 }}}\).
\end{align}
It then suffices to estimate $\abs{\dt^{k }(\dt v\cdot\N)}_{-\hal} $. However, this will lead to some difficulties since $k$ can be $m-1$, and
energy estimates yield only $\dt^{m} v\in L^2(\Omega)$ which cannot ensure the control of the
$H^{-\hal}\{z=0\}$ norm of $\dt^m v$. The key observation is that $\dt^m v\cdot \N$ is indeed in $H^{-\hal}(\{z=0\})$. The way of achieving this
is to
use the Alinhac good unknown, returning back to $k$, $V^{k+1}=\dt^{k+1}v-\D_zv\dt^{k+1}\eta$. Indeed, since $\nabla^\varphi\cdot V^{k+1}=-
C^{k+1}(d)$ with $C^{k+1}(d)$ defined as in \eqref{divValpha} with $Z^\alpha=\dt^{k+1}$,  by Lemma \ref{-1/2b}, and using the estimate
\eqref{Cd} and \eqref{etaharm}, one gets that
 \begin{align}
\abs{V^{k+1}\cdot \N}_{-\hal}&\leq   \Lambda\( {1 \over c_{0}} \)  \(\norm{ V^{k+1} }_{0} +\norm{ \nabla^\varphi\cdot V^{k+1} }_{0}  \)
\\\nonumber&\leq   \Lambda\( {1 \over c_{0}},\norm{\nabla v}_{L^\infty} \)  \(\norm{\dt^{k+1} v}_{0}   +   \norm{\dt^{k+1}\eta}_{0}  +\norm{ C^{k+1}(d) }_{0}
\)
\\\nonumber&\leq\Lambda\( {1 \over c_{0}}, |h |_{\Y^{{{\[\!\frac{k+1}{2}\!\]}}+1}}   + \norm{ \nabla v }_{\Y^{{{\[\!\frac{k+1}{2}\!\]}}}} \)
   \(\norm{\dt^{k+1} v}_{0}   + \abs{h}_{\X^{{k+1},-{1 \over 2 }}}+ \norm{\nabla v}_{\X^{k}}\).
\end{align}
This in turn implies
 \begin{align}\label{tttt}
\abs{\dt^{k+1} v \cdot \N}_{-\hal}&\leq   \abs{V^{k+1}\cdot \N}_{-\hal}+\abs{  \N\cdot \D_zv\dt^{k+1}\eta}_{-\hal}
\\\nonumber&\leq\Lambda\( {1 \over c_{0}}, |h |_{\Y^{{{\[\!\frac{k+3}{2}\!\]}} }}   + \norm{ \nabla v }_{\Y^{{{\[\!\frac{k+1}{2}\!\]}}}} \)
   \(\norm{\dt^{k+1} v}_{0}  + \abs{h}_{\X^{{k+1},-{1 \over 2 }}} + \norm{\nabla v}_{\X^{k}} \).
\end{align}
Hence, it follows from \eqref{gggesrt00}, \eqref{tttt} and the trace estimate $|\cdot|_{-\hal}\le |\cdot|_{\hal}\ls \norm{\cdot}_{H^1}$ that
 \begin{align}\label{gggesrt}
&\abs{\dt v\cdot\N}_{{\X^{k,-\hal}}} \leq   \abs{\dt v\cdot\N}_{{\X^{k-1,\hal}}}+\abs{\dt^{k }(\dt v\cdot\N)}_{-\hal}
\\\nonumber&\quad\leq\Lambda\( {1 \over c_{0}}, \abs{h }_{\Y^{{\[\!\frac{k+3}{2}\!\]}  }}  + \norm{  v }_{\Y^{{\[\!\frac{k+3}{2}\!\]}}}  + \norm{ \nabla v
}_{\Y^{{\[\!\frac{k+1}{2}\!\]}}} \)
   \(\norm{v}_{\X^{k+1}}  + \norm{\nabla v}_{\X^{k}} + \abs{h}_{\X^{k+1,-{1 \over 2 }}}\).
\end{align}
Therefore, we conclude from \eqref{gg1es1}, \eqref{gg1es2} and \eqref{gggesrt} that
 \begin{align}\label{G2est1}
\abs{G^2}_{{\X^{{k},-\hal}}} \leq &  \Lambda\( {1 \over c_{0}},  \abs{h}_{\Y^{{\[\!\frac{k}{2}\!\]}+3}}  + \norm{  v }_{\Y^{{\[\!\frac{k+3}{2}\!\]}}} +\norm{\nabla
v}_{\Y^{{\[\!\frac{k}{2}\!\]}+2}}  \)
\\\nonumber&\times\(\norm{v}_{\X^{k+1}}  + \norm{\nabla v}_{\X^{k}} + \abs{h}_{\X^{k+1,-{1 \over 2 }}}+\eps\abs{  v }_{{\X^{{k},
\frac{3}{2}}}}+\eps\abs{h}_{{\X^{{k},\frac{3}{2}}}}\).
\end{align}
Similarly, since $k\ge 3$, due to the trace estimates, one can get
 \begin{align}\label{G2est2}
\abs{G^2}_{{\X^{{{\[\!\frac{k+3}{2}\!\]}},-\hal}}} &\leq \Lambda\( {1 \over c_{0}},  \abs{h}_{\Y^{\frac{k+3}{4}+3}}  + \norm{  v
}_{\Y^{{\[\!\frac{\[\frac{k+3}{2}\]+3}{2}\!\]}}} +\norm{\nabla v}_{\Y^{\frac{k+3}{4}+2}}  \)
\\\nonumber&\quad\times\(\norm{v}_{\X^{{\[\!\frac{k+3}{2}\!\]}+1}}  + \norm{\nabla v}_{\X^{{\[\!\frac{k+3}{2}\!\]}}} + \abs{h}_{\X^{{\[\!\frac{k+3}{2}\!\]}+1,-{1 \over 2
}}}+\eps\abs{  v }_{{\X^{{{\[\!\frac{k+3}{2}\!\]}}, \frac{3}{2}}}}+\eps\abs{h}_{{\X^{{{\[\!\frac{k+3}{2}\!\]}},\frac{3}{2}}}}\)
\\\nonumber&\leq\Lambda\( {1 \over c_{0}},  \abs{h}_{\Y^{{\[\!\frac{k}{2}\!\]}+3}} +\abs{h}_{{\X^{{{\[\!\frac{k+7}{2}\!\]}}}}} + \norm{  v }_{\Y^{{\[\!\frac{k+3}{2}\!\]}}}
+\norm{\nabla v}_{\Y^{{\[\!\frac{k}{2}\!\]}+2}} +\norm{v}_{\X^{{\[\!\frac{k+7}{2}\!\]}}}  + \norm{\nabla v}_{\X^{{\[\!\frac{k+5}{2}\!\]}}}  \)  .
\end{align}

Consequently, plugging the estimates \eqref{Fest1}--\eqref{Fest2}, \eqref{G1est1}--\eqref{G1est2} and  \eqref{G2est1}--\eqref{G2est2}, along with
doing the the same estimates for $G^3$ as for $G^2$, into \eqref{qclaim}, we can obtain that for $k\ge 3$,
  \begin{align}\label{qkxes}
  &    \norm{   q }_{\X^{k}}+\norm{ \nabla q }_{\X^{k}}+\norm{ \pa_{zz} q}_{\X^{k-1}}
  \\\nonumber&  \quad\leq \Lambda\( {1 \over c_{0}},   \abs{h}_{\Y^{{\[\!\frac{k}{2}\!\]}+3}} +\abs{h}_{{\X^{{{\[\!\frac{k+7}{2}\!\]}}}}} +\norm{  v
  }_{\Y^{{\[\!\frac{k+3}{2}\!\]}}}  +\norm{\nabla v}_{\Y^{{\[\!\frac{k}{2}\!\]}+2}} +\norm{v}_{\X^{{\[\!\frac{k+7}{2}\!\]}}}  + \norm{\nabla v}_{\X^{{\[\!\frac{k+5}{2}\!\]}}}  \)
  \\\nonumber&\qquad\times\(\abs{h}_{\X^{k+1,-{1 \over 2 }}}+\sigma \abs{h}_{\X^{k,2}}+\norm{v}_{\X^{k+1}}  + \norm{\nabla v}_{\X^{k}} +
  \eps\abs{  v }_{{\X^{{k}, \frac{3}{2}}}}+\eps\abs{h}_{{\X^{{k},\frac{3}{2}}}}\).
  \end{align}
This proves the estimate \eqref{qpressureest}.

To prove \eqref{qinfty}, one can use the anisotropic Sobolev embedding estimate \eqref{emb} and the trace estimates to have that for $k\ge 1$, by
\eqref{qpressureest},
  \begin{align}
  &\norm{   q }_{\Y^{k}}+\norm{ \nabla q }_{\Y^{k}}\ls \norm{ \nabla q }_{\X^{k+2}}+\norm{ \pa_{zz} q}_{\X^{k+1}}
  \\\nonumber&  \quad\leq \Lambda\( {1 \over c_{0}},   \abs{h}_{\Y^{{{\[\!\frac{k}{2}\!\]}}+4}} +\abs{h}_{{\X^{{\frac{k+9}{2}}}}} +\norm{  v
  }_{\Y^{{\[\!\frac{k+5}{2}\!\]}}}+\norm{\nabla v}_{\Y^{{{\[\!\frac{k}{2}\!\]}}+3}} +\norm{v}_{\X^{\frac{k+9}{2}}}   + \norm{\nabla v}_{\X^{{\[\!\frac{k+7}{2}\!\]}}}  \)
  \\\nonumber&\qquad\times\(\abs{h}_{\X^{k+3,-{1 \over 2 }}}+\sigma \abs{h}_{\X^{k+2,2}}+\norm{v}_{\X^{k+3}}  + \norm{\nabla v}_{\X^{k+2}} +
  \eps\abs{  v }_{{\X^{{k+2}, \frac{3}{2}}}}+\eps\abs{h}_{{\X^{{k+2},\frac{3}{2}}}}\)
    \\\nonumber&  \quad\leq \Lambda\( {1 \over c_{0}},   \abs{h}_{\Y^{{\[\!\frac{k}{2}\!\]}+4}}  +\abs{h}_{\X^{k+4}}+\norm{  v
    }_{\Y^{{\[\!\frac{k+5}{2}\!\]}}}+\norm{\nabla v}_{\Y^{{\[\!\frac{k}{2}\!\]}+3}} +  \norm{v}_{\X^{k+4}}  + \norm{\nabla v}_{\X^{k+3}}   \).
  \end{align}
  Note that one has used the fact that $k+2\ge 3$ so that \eqref{qpressureest} can be used with $k+2$. This, together with
 \eqref{eqdzzrho} again, proves \eqref{qinfty}.
\end{proof}

\section{Smoothing estimates of $h$}\label{sectionprelimeta}

We first show the smoothing regularity estimates of $h$ coming from viscosity.
\begin{prop}\label{heps}
For every $m\in \mathbb{N}$,  $\eps \in(0,1)$, it holds that
\beq\label{hepses1}
\eps\abs{h(t)}_{\X^{m-1, \frac{3}{2}}}^2\leq   \eps\abs{h(0)}_{\X^{m-1, \frac{3}{2}}}^2+\int_{0}^t \Lambda\(\abs{ h}_{\Y^{{\[\!\frac{m}{2}\!\]}+2}}
+\norm{ v }_{\Y^{{\[\!\frac{m}{2}\!\]}+2}}  \)\( \eps\abs{  v} _{\X^{m-1, \frac{3}{2}}}^2 +  \eps\abs{h}_{\X^{m-1, \frac{3}{2}}}^2 \)
\eeq
and
\beq\label{hepses2}
\eps\abs{h(t)}_{\X^{m,\hal}}^2\leq   \eps\abs{h(0)}_{\X^{m,\hal}}^2+\int_{0}^t \Lambda\(\abs{ h}_{\Y^{{\[\!\frac{m}{2}\!\]}+2}}  +\norm{ v
}_{\Y^{{\[\!\frac{m}{2}\!\]}+2}}  \)
\( \eps\abs{  v} _{\X^{m,\hal}}^2 +  \eps\abs{h}_{\X^{m,\hal}}^2 \).
\eeq
\end{prop}
\begin{proof}
We prove only the estimate \eqref{hepses2}, and the estimate \eqref{hepses1} follows in the same way.
Apply $Z^\alpha$ with $\alpha\in \mathbb{N}^{1+2}$, $|\alpha|\le m$, to the kinematic boundary condition to get that on $\{z=0\}$:
\beq\label{heq1}
\partial_{t} Z^\alpha h + v  \cdot \nabla_y  Z^\alpha h -Z^\alpha v_{3}  + \[ Z^\alpha, v_{y} \]\cdot \nabla_y h=0.
\eeq
Then applying further $\Lambda^\hal$, the tangential Fourier multiplier, to \eqref{heq1} gives
\beq\label{heq2}
\partial_{t} \Lambda^\hal Z^\alpha h + v  \cdot \nabla_y \Lambda^\hal Z^\alpha h -\Lambda^\hal Z^\alpha  v_{3} + \[\Lambda^\hal, v_{y} \]\cdot
\nabla_y Z^\alpha h + \Lambda^\hal \(\[ Z^\alpha, v_{y} \]\cdot \nabla_y h\)=0.
\eeq
A standard energy estimate on the equation \eqref{heq2} yields
\begin{align}\label{heq3}
 \dtt \abs{Z^\alpha h}_{ \hal }^2&  \ls \norm{\nabla_{y} v}_{L^\infty }  \abs{h}_{\X^{m,\hal}}^2
\\ \nonumber & \quad +  \(   \abs{v}_{\X^{m,\hal}}+\abs{\[\Lambda^\hal, v_{y} \]\cdot \nabla_y Z^\alpha h}_{0} +\abs{\[ Z^\alpha, v_{y} \]\cdot
\nabla_y h}_{\hal}\) \abs{h}_{\X^{m,\hal}}.
\end{align}
Due to the commutator estimate \eqref{comr2}, one has
\beq\label{heq4}
\abs{\[\Lambda^\hal, v_{y} \]\cdot \nabla_y Z^\alpha h}_{0}\ls  \norm{\nabla_{y} v  }_{L^\infty} \abs{h}_{\X^{m,\hal}}
+  \norm{\nabla_{y} h  }_{L^\infty} \abs{v}_{\X^{m,\hal}}.
\eeq
And the estimate \eqref{gues2} leads to
\begin{align}\label{heq5}
\abs{\[ Z^\alpha, v_{y} \]\cdot \nabla_y h}_{\hal}&\le \sum_{|\beta|+|\gamma|=\alpha\atop |\beta'|=1}\abs{Z^{\beta-\beta'}Z^{\beta'} v_{y}
\cdot \nabla_y Z^{\gamma} h}_{\hal}
\\\nonumber&\ls\norm{ v }_{\Y^{{\[\!\frac{m}{2}\!\]}+2}}\abs{h}_{\X^{m,\hal}}+\abs{ h}_{\Y^{{\[\!\frac{m}{2}\!\]}+2}}\abs{  v} _{\X^{m,\hal}}.
\end{align}
Hence, plugging the estimates \eqref{heq4} and \eqref{heq5} into \eqref{heq3} and summing over $|\alpha|\le m$, by Cauchy's inequality, one can
deduce
that
\beq\label{heq6}
 \dtt  \abs{h}_{\X^{m,\hal}}^2\ls\(1+\abs{ h}_{\Y^{{\[\!\frac{m}{2}\!\]}+2}}+\norm{ v }_{\Y^{{\[\!\frac{m}{2}\!\]}+2}}\)\( \abs{  v} _{\X^{m,\hal}}^2 +
 \abs{h}_{\X^{m,\hal}}^2\).
\eeq
Integrating the inequality \eqref{heq6} directly in time yields \eqref{hepses2}.
\end{proof}

Next, we show the smoothing estimates of $h$ due to surface tension.
\begin{prop}\label{hsig}
For every $m\in \mathbb{N}$,  $\es \in(0,1)$, it holds that
\begin{align}\label{hsiges}
 \sigma^2\abs{h}_{\X^{m-1,\frac{5}{2}}}^2\le& \Lambda\( {1 \over c_{0}}, \abs{h }_{\Y^{{{\[\!\frac{m+3}{2}\!\]}}}}   + \norm{  \nabla v
 }_{\Y^{{\[\!\frac{m+1}{2}\!\]}}} \)
\(\abs{ q}_{\X^{m-1,\hal}}^2 +\abs{ h}_{\X^{m-1,\hal}}^2 \right.
 \\\nonumber&\qquad\qquad\quad\left.   +\eps^2 \abs{  v }_{\X^{m-1,\frac{3}{2}}}^2+\eps^2\abs{h }_{\X^{m-1,\frac{3}{2}}}^2   \).
\end{align}
\end{prop}
\begin{proof}
Apply $Z^\alpha$ with $\alpha\in \mathbb{N}^{1+2}$, $|\alpha|\le m,\alpha_0\le m-1$, to the dynamic boundary condition to have
\beq
-\sigma Z^\alpha H=Z^\alpha  q  -2 \eps Z^\alpha( S^\varphi v  \,{\bf n} \cdot \,{\bf n})- gZ^\alpha h .
\eeq
Note that
\begin{align}
     Z^\alpha H &=\nabla_y\cdot \(Z^\alpha\(\frac{\nabla_y h}{\sqrt{1+|\nabla_y h|^2}}\) \)
  \\& \nonumber=\nabla_y\cdot \( \frac{ \nabla_y Z^\alpha h}{\sqrt{1+|\nabla_y h|^2}}+\nabla_y h Z^\alpha\(\frac{1}{\sqrt{1+|\nabla_y h|^2}}\)
  +\[Z^\alpha, \nabla_y h, \frac{1}{\sqrt{1+|\nabla_y h|^2}}\] \)
      \end{align}
and
\begin{align}
Z^\alpha\(\frac{1}{\sqrt{1+|\nabla_y h|^2}}\)&=Z^{\alpha-\alpha'} Z^{\alpha'}\(\frac{1}{\sqrt{1+|\nabla_y h|^2}}\)
=-Z^{\alpha-\alpha'} \(\frac{\nabla_y h\cdot \nabla_y Z^{\alpha'} h}{\sqrt{1+|\nabla_y h|^2}^3}\)
\\\nonumber& =- \frac{\nabla_y h\cdot \nabla_y Z^\alpha h}{\sqrt{1+|\nabla_y h|^2}^3}-\[Z^{\alpha-{\alpha'}},\frac{\nabla_y h}{\sqrt{1+|\nabla_y
h|^2}^3}\]\cdot \nabla_y Z^{\alpha'} h
\end{align}
for any $|\alpha'|=1$. Hence,
  \begin{align}\label{sigsur1}
 Z^\alpha H  = \sigma \nabla_y\cdot \( \frac{ \nabla_y Z^\alpha h}{\sqrt{1+|\nabla_y h|^2}}- \frac{\nabla_y h\cdot \nabla_y Z^\alpha
 h}{\sqrt{1+|\nabla_y h|^2}^3}\nabla_y h
  +\mathcal{C}(\mathcal{B}_\sigma^\alpha) \),
    \end{align}
where
\begin{align}\label{sigsur2}
 \mathcal{C}(\mathcal{B}_\sigma^\alpha)=-\[Z^{\alpha-\alpha'},\frac{\nabla_y h}{\sqrt{1+|\nabla_y h|^2}^3}\]\cdot \nabla_y Z^{\alpha'} h
 \nabla_y h+\[Z^\alpha, \frac{1}{\sqrt{1+|\nabla_y h|^2}}, \nabla_y h\]
\end{align}
for any $|\alpha'|=1$. It follows that
\beq\label{fkdal}
-\sigma   \nabla_y\cdot \( \frac{ \nabla_y Z^\alpha h}{\sqrt{1+|\nabla_y h|^2}}- \frac{\nabla_y h\cdot \nabla_y Z^\alpha h}{\sqrt{1+|\nabla_y
h|^2}^3}\nabla_y h  \) =\sigma   \nabla_y\cdot  \mathcal{C}(\mathcal{B}_\sigma^\alpha)+Z^\alpha  q  -2 \eps Z^\alpha( S^\varphi v  \,{\bf n}
\cdot \,{\bf n})- gZ^\alpha h .
\eeq

Then apply further $\Lambda^\hal$ to \eqref{fkdal} to get
\begin{align}
&-\sigma   \nabla_y\cdot \( \frac{ \nabla_y \Lambda^\hal Z^\alpha h}{\sqrt{1+|\nabla_y h|^2}}- \frac{\nabla_y h\cdot \nabla_y \Lambda^\hal
Z^\alpha h}{\sqrt{1+|\nabla_y h|^2}^3}\nabla_y h  \)
 \\\nonumber&\quad=\sigma   \nabla_y\cdot \( \[\Lambda^\hal, \frac{ 1}{\sqrt{1+|\nabla_y h|^2}}-1\]\nabla_y  Z^\alpha h- \[\Lambda^\hal,
 \frac{\nabla_y h}{\sqrt{1+|\nabla_y h|^2}^3}\nabla_y h \cdot \]\nabla_y  Z^\alpha h \)
   \\\nonumber&\qquad+\Lambda^\hal\( \sigma   \nabla_y\cdot  \mathcal{C}(\mathcal{B}_\sigma^\alpha)+Z^\alpha  q  -2 \eps Z^\alpha( S^\varphi v
   \,{\bf n} \cdot \,{\bf n})- gZ^\alpha h\).
\end{align}
It follows from a standard energy estimate for this elliptic equation that
\begin{align}
& \sigma^2\int_{z=0}  \( \frac{  |\nabla_y \Lambda^\hal Z^\alpha h|^2 }{\sqrt{1+|\nabla_y h|^2}}-  \frac{ |\nabla_y h\cdot \nabla_y \Lambda^\hal
Z^\alpha h|^2  }{\sqrt{1+|\nabla_y h|^2}^3}   \)\,dy
 \\\nonumber&\quad=- \sigma^2\( \[\Lambda^\hal, \frac{ 1}{\sqrt{1+|\nabla_y h|^2}}-1\]\nabla_y  Z^\alpha h- \[\Lambda^\hal, \frac{\nabla_y
 h}{\sqrt{1+|\nabla_y h|^2}^3}\nabla_y h \cdot \]\nabla_y  Z^\alpha h, \nabla_y \Lambda^\hal Z^\alpha h \)
 \\\nonumber&\qquad-\sigma^2  \(\Lambda^\hal\mathcal{C}(\mathcal{B}_\sigma^\alpha),\nabla_y \Lambda^\hal Z^\alpha h \)+\(\Lambda^\hal\(Z^\alpha
 q  -2 \eps Z^\alpha( S^\varphi v  \,{\bf n} \cdot \,{\bf n})- gZ^\alpha h\), \sigma \Lambda^\hal Z^\alpha h \)  .
\end{align}
Since for any vector $\mathbf{a}\in\mathbb{R}^2$,
\beq\label{sigsur3}
\frac{  |\mathbf{a}|^2 }{\sqrt{1+|\nabla h|^2}}-  \frac{ |\nabla_y h\cdot \mathbf{a}|^2  }{\sqrt{1+|\nabla_y h|^2}^3}
\ge   \frac{ 1  }{\sqrt{1+|\nabla_y h|^2}^3} |\mathbf{a}|^2,
\eeq
so Cauchy's inequality, together with \eqref{cont2D} and \eqref{gues2}, since $|\alpha_1|+|\alpha_2|\ge 1$, give
\begin{align}\label{la1}
 \sigma^2\abs{\nabla_yh}_{\X^{m-1,\frac{5}{2}}}^2&\le \Lambda_0
\( \abs{ q}_{\X^{m-1,\hal}}^2 +\eps^2\abs{S^\varphi v  \,{\bf n} \cdot \,{\bf n}}_{\X^{m-1,\hal}}^2   +\abs{ h}_{\X^{m-1,\hal}}^2\right.
 \\\nonumber&\quad  +\sigma^2\abs{ \mathcal{C}(\mathcal{B}_\sigma^\alpha)}_\hal^2+\sigma^2\abs{ \[\Lambda^\hal, \frac{ 1}{\sqrt{1+|\nabla_y
 h|^2}}-1\]\nabla_y  Z^\alpha h}_0^2
  \\\nonumber&\quad \left. +\sigma^2\abs{ \[\Lambda^\hal, \frac{\nabla_y h}{\sqrt{1+|\nabla_y h|^2}^3}\nabla_y h \cdot \]\nabla_y  Z^\alpha
  h}_0^2\)
 \\&\nonumber\le \Lambda\( {1 \over c_{0}}, \abs{h }_{\Y^{{{\[\!\frac{m+3}{2}\!\]}}}}   + \norm{  \nabla v }_{\Y^{{\[\!\frac{m+1}{2}\!\]}}} \)
\(\abs{ q}_{\X^{m-1,\hal}}^2 +\abs{ h}_{\X^{m-1,\hal}}^2\right.
   \\&\nonumber\left.\quad+\sigma^2 \abs{h}_{\X^{m-1,\frac{3}{2}}}^2 +\eps^2\abs{S^\varphi v  \,{\bf n} \cdot \,{\bf n}}_{\X^{m-1,\hal}}^2\).
\end{align}
Using the similar arguments in the previous section, one can have
$$
\abs{S^\varphi v  \,{\bf n} \cdot \,{\bf n}}_{\X^{m-1,\hal}}^2  \leq    \Lambda\( {1 \over c_{0}}, \abs{h }_{\Y^{{\[\!\frac{m+1}{2}\!\]}}}   + \norm{
\nabla v }_{\Y^{{\[\!\frac{m+1}{2}\!\]}}} \)  \( \abs{  v }_{\X^{m-1,\frac{3}{2}}}
   +\abs{h }_{\X^{m-1,\frac{3}{2}}} \).
   $$
Then \eqref{hsiges} follows from \eqref{la1}, the Sobolev interpolation and Young's inequality.
\end{proof}

\section{Conormal estimates}\label{secconormal}

We shall derive a priori estimates on a time interval $[0, T^\es]$ on which it is assumed that
\beq\label{apriori}
\partial_{z} \varphi \geq \frac{{c_{0}}}{2} , \  |h|_{2, \infty} \leq {1 \over c_{0}}\text{ and }  g  -   \partial_{z}^\varphi q   \geq
{c_{0}\over
2}\text{ on } \{z=0\}.
\eeq
Note in particular that this will allow one to use Lemmas \ref{mingrad} and \ref{Korn}.

To derive the higher order energy estimates, we shall use the good unknown $ V^\alpha= Z^\alpha v-\D_{z}v Z^\alpha \eta$, $\alpha\neq0$. A
key point is that the control of $V^\alpha$ and $Z^\alpha h$ will yield a control of $Z^\alpha v$:
\beq\label{equiv1}
\norm{Z^\alpha v} \ls \norm{V^\alpha} +  \Lambda\({1 \over c_{0}} , \norm{\nabla v}_{L^\infty}\) \abs{Z^\alpha h}_{-\hal},\ \norm{V^\alpha} \ls
\norm{Z^\alpha v} +  \Lambda\({1 \over c_{0}} , \norm{\nabla v}_{L^\infty}\) \abs{Z^\alpha h}_{-\hal}.
\eeq

Define
\beq\label{deflambdainfty1}
\Lambda_{\infty}(t)= \Lambda \( {1 \over c_{0}},  \abs{h(t)}_{\X^{{\[\!\frac{m}{2}\!\]}+5}} +\norm{v(t) }_{\X^{{\[\!\frac{m}{2}\!\]}+5}}+\norm{\pa_z v (t)
}_{\X^{{\[\!\frac{m}{2}\!\]}+4}} +\norm{\pa_z v(t) }_{\Y^{{\[\!\frac{m}{2}\!\]}+2}}+ \eps^{1 \over 2} \norm{\partial_{zz}v(t)}_{L^\infty} \).
\eeq
It will be shown then that those functions $\Lambda(\cdot,\cdot)$ defined in the previous three sections can be bounded by $\linf$ for
sufficiently large
$m$, and also the elliptic estimates of $q$ and the smoothing estimates of $h$ will be restated along the way. First, taking $k={\[\!\frac{m}{2}\!\]}$
for $m\ge 6$ in the estimate \eqref{qinfty} yields
\begin{align}\label{qminfty}
&\norm{   q }_{\Y^{{\[\!\frac{m}{2}\!\]}}}+ \norm{ \nabla q }_{\Y^{{\[\!\frac{m}{2}\!\]}}}+ \norm{ \pa_{zz} q }_{\Y^{{\[\!\frac{m}{2}\!\]}-1}}
\\\nonumber&\quad\le \Lambda\( {1 \over c_{0}},   \abs{h}_{\Y^{{{\[\!\frac{m}{4}\!\]}}+4}}  +\abs{h}_{\X^{{\[\!\frac{m}{2}\!\]}+4}}+\norm{  v
}_{\Y^{{{\[\!\frac{m+10}{4}\!\]}} }}+\norm{\nabla v}_{\Y^{{{\[\!\frac{m}{4}\!\]}}+3}} +  \norm{v}_{\X^{{\[\!\frac{m}{2}\!\]}+4}}  + \norm{\nabla v}_{\X^{{\[\!\frac{m}{2}\!\]}+3}}
\)\leq \linf,
\end{align}
while taking $k=m-1$ in the estimate \eqref{qpressureest} gives
\begin{align}\label{qpressurem}
  &    \norm{   q }_{\X^{m-1}}+\norm{ \nabla q }_{\X^{m-1}}+\norm{ \pa_{zz} q}_{\X^{m-2}}
  \\\nonumber&  \quad\leq \Lambda\( {1 \over c_{0}},   \abs{h}_{\Y^{{\[\!\frac{m+5}{2}\!\]}}}  +\abs{h}_{\X^{{\[\!\frac{m}{2}\!\]}+3}} +\norm{  v
  }_{\Y^{{\[\!\frac{m}{2}\!\]}+1}} +\norm{\nabla v}_{\Y^{{{\[\!\frac{m+3}{2}\!\]}} }} +  \norm{v}_{\X^{{\[\!\frac{m}{2}\!\]}+3}}  + \norm{\nabla v}_{\X^{{\[\!\frac{m}{2}\!\]}+2}}   \)
  \\\nonumber&\qquad\times\(\abs{h}_{\X^{m,-{1 \over 2 }}}+\sigma \abs{h}_{\X^{m-1,2}}+\norm{v}_{\X^{m}}  + \norm{\nabla v}_{\X^{m-1}} +
  \eps\abs{  v }_{{\X^{m-1, \frac{3}{2}}}}+\eps\abs{h}_{{\X^{m-1,\frac{3}{2}}}}\)
   \\\nonumber&  \quad\leq \linf\(\abs{h}_{\X^{m,-{1 \over 2 }}}+\sigma \abs{h}_{\X^{m-1,2}}+\norm{v}_{\X^{m}}  + \norm{\pa_z v}_{\X^{m-1}} +
   \eps\abs{  v }_{{\X^{m-1, \frac{3}{2}}}}+\eps\abs{h}_{{\X^{m-1,\frac{3}{2}}}}\).
  \end{align}
Here one has required $m\ge 6$ so that by Sobolev's inequality
$$  \abs{h}_{\Y^{{{\[\!\frac{m}{4}\!\]}}+4}}  \ls \abs{h}_{\X^{{\[\!\frac{m}{2}\!\]}+5}},
$$
and that by the anisotropic Sobolev embedding estimate \eqref{emb}
 $$\norm{  v }_{\Y^{{{\[\!\frac{m+3}{2}\!\]}}+1}}\ls \norm{  v }_{\X^{{\[\!\frac{m}{2}\!\]}+5}}+\norm{\pa_z  v }_{\X^{{\[\!\frac{m}{2}\!\]}+4}}.$$
It can be checked easily that all the functions $\Lambda( \cdot,\cdot)$ defined in Propositions \ref{heps} and \ref{hsig} and Lemmas
\ref{lemValpha},
\ref{lembordV} and \ref{lembordh} are bounded by $\Lambda_{\infty}$, due to \eqref{qminfty}. Moreover, Proposition \ref{heps} implies
that
\beq\label{mheps1}
\eps\abs{h(t)}_{\X^{m-1, \frac{3}{2}}}^2\leq   \eps\abs{h(0)}_{\X^{m-1, \frac{3}{2}}}^2+\int_{0}^t \linf\( \eps\abs{  v} _{\X^{m-1,
\frac{3}{2}}}^2 +  \eps\abs{h}_{\X^{m-1, \frac{3}{2}}}^2 \)
\eeq
and
\beq \label{mheps2}
\eps\abs{h(t)}_{\X^{m,\hal}}^2\leq   \eps\abs{h(0)}_{\X^{m,\hal}}^2+\int_{0}^t \linf
\( \eps\abs{  v} _{\X^{m,\hal}}^2 +  \eps\abs{h}_{\X^{m,\hal}}^2 \).
\eeq
Proposition \ref{hsig}, together with \eqref{qpressurem} and the trace estimate $\abs{q}_{\X^{m-1,\hal}}\ls\norm{   q
}_{\X^{m-1}}+\norm{ \nabla q }_{\X^{m-1}}$, yields that
\begin{align}\label{mhsig1}
 \sigma^2\abs{h}_{\X^{m-1,\frac{5}{2}}}^2\le& \linf
\(\abs{h}_{\X^{m,-{1 \over 2 }}}^2+\sigma^2 \abs{h}_{\X^{m-1,2}}^2+ \norm{v}_{\X^{m}}^2  + \norm{\pa_z v}_{\X^{m-1}}^2 +  \eps^2\abs{  v
}_{{\X^{m-1, \frac{3}{2}}}}^2+\eps^2\abs{h}_{{\X^{m-1,\frac{3}{2}}}}^2 \).
\end{align}
Since $\abs{h}_{\X^{m-1,2}}^2\le \abs{h}_{\X^{m-1}}^\frac{1}{5}\abs{h}_{\X^{m-1,\frac{5}{2}}}^\frac{4}{5}$, one may improve \eqref{mhsig1} by
using Young's inequality to get
\begin{align}\label{mhsig}
 \sigma^2\abs{h}_{\X^{m-1,\frac{5}{2}}}^2\le& \linf
\(\abs{h}_{\X^{m,-{1 \over 2 }}}^2+ \norm{v}_{\X^{m}}^2  + \norm{\pa_z v}_{\X^{m-1}}^2 +  \eps^2\abs{  v }_{{\X^{m-1,
\frac{3}{2}}}}^2+\eps^2\abs{h}_{{\X^{m-1,\frac{3}{2}}}}^2 \).
\end{align}

Note that in the following we will use frequently these $L^\infty$ bounds involved in $\linf$.

\subsection{Basic $L^2$ estimate}\label{secconormal0}

We start with the estimates of $(v,h)$ itself, that is, the case $\alpha=0$.
  \begin{prop}
  \label{basicL2}
  For any smooth solution of \eqref{NSv}, it holds that
\begin{align}\label{l2estimate}
 \norm{v(t)}_0^2 +g\abs{h(t)}_0^2+\sigma\abs{h(t)}_1^2 + \eps \int_{0}^t \norm{ \nabla v}_0^2
 \le \Lambda_0  \(\norm{v_{0}}_0^2+ \abs{h_0}_0^2+\sigma\abs{h_0}_1^2+ \int_{0}^t \norm{v}_0^2\).
 \end{align}
  \end{prop}

\begin{proof}
Standard energy identity yields
\beq\label{l20}
\hal\dtt \int_{\Omega} |v|^2 \, d\V +   \eps  \int_{\Omega} |S^\varphi v|^2\, d\V=
\int_{z=0} \(2 \eps S^\varphi v - q \mbox{I} \) \N \cdot v \, dy-\int_{z=-b} \(2 \eps S^\varphi v - q \mbox{I} \) e_3 \cdot v \, dy.
\eeq
The Navier slip boundary condition implies that
\beq\label{l21}
-\int_{z=-b}\(2 \eps S^\varphi v - q \mbox{I} \) e_3 \cdot v \, dy=-\int_{z=-b}  2 \eps (S^\varphi v   e_3)_i  v_i \, dy=-2 \kappa
\eps\int_{z=-b}|v|^2\, dy.
\eeq
While the dynamic boundary condition and the kinematic boundary condition give
\begin{align}\label{l22}
\int_{z=0} \(2 \eps S^\varphi v - q \mbox{I} \) \N \cdot v \, dy& =  -  \int_{z=0} (g h-\sigma H) \N  \cdot v \, dy=  -  \int_{z=0} (g h-\sigma
H) \partial_t h \, dy
\\\nonumber&=   - \hal\dtt \int_{z=0} g|h|^2+2\sigma\(\sqrt{1+|\nabla h|^2}-1\) \,dy.
\end{align}
Consequently,
\begin{align}\label{l23}
& \hal\dtt \(  \int_{\Omega} |v|^2 \, d\V +   \int_{z=0} g|h|^2+2\sigma\(\sqrt{1+|\nabla h|^2}-1\) \,dy  \)
 \\\nonumber&\qquad\qquad\qquad\qquad\qquad\qquad+ 2 \eps  \int_{\Omega} |S^\varphi v|^2\, d\V+2 \kappa \eps\int_{z=-b}|v|^2\, dy=0.
 \end{align}

Note that
$$\sqrt{1+|\nabla h|^2}-1\ge \frac{1}{2}\frac{1}{\sqrt{1+({1 \over c_0})^2}}|\nabla h|^2$$
due to \eqref{apriori}, and the trace estimate
$$\abs{v}_{L^2(\{z=-b\})}^2\ls \norm{\nabla v}\norm{v}+\norm{v}^2.$$
Hence, \eqref{l2estimate} follows from \eqref{l23}, Lemma \ref{Korn}  and Cauchy's inequality.
\end{proof}

\subsection{Estimate of $(Z^\alpha v, Z^\alpha h)$ for $\alpha_0\le m-1$}\label{secconormal1}

Next, we derive the energy estimates of $(Z^\alpha v, Z^\alpha h)$ for $1 \leq | \alpha |\leq m$ and $\alpha_0\le m-1$, that is, except the cases
$\alpha=0$ or $\alpha_0=m$.

\begin{prop}\label{conormv1}
Any smooth solution of \eqref{NSv} satisfies the estimate
\begin{align}\label{conormales1}
\nonumber&\norm{v(t)}_{\X^{m-1,1}}^2+\abs{h(t)}_{\X^{m-1,1}}^2 +\sigma\abs{h(t)}_{\X^{m-1,2}}^2 +\eps  \abs{h(t)}_{\X^{m-1,
\frac{3}{2}}}^2+\int_0^t\eps  \norm{\nabla v}_{\X^{m-1,1}}^2+\sigma^2\abs{h}_{\X^{m-1,\frac{5}{2}}}^2
\\&\quad \le \Lambda_0\( \norm{v(0)}_{\X^{m-1,1}}^2+\abs{h(0)}_{\X^{m-1,1}}^2 +\sigma\abs{h(0)}_{\X^{m-1,2}}^2+ \eps \abs{h(0)}_{\X^{m,
\hal}}^2\)
\\\nonumber&\qquad+\int_0^t\linf \(\abs{h}_{\X^{m-1,1}}^2+\abs{h}_{\X^{m,-\hal}}^2+\sigma \abs{h}_{\X^{m,1}}^2+ \eps  \abs{h}_{\X^{m-1,
\frac{3}{2}}}^2+\norm{v}_{\X^{m}}^2+  \norm{\pa_z v}_{\X^{m-1}}^2 \).
\end{align}
\end{prop}
\begin{proof}
The energy identity for the equations \eqref{eqValpha}--\eqref{divValpha} yields
\beq \label{conormen}
\hal\dtt\int_{\Omega}\abs{V^\alpha }^2 d\V + 2 \eps \int_{\Omega} \abs{S^\varphi V^\alpha}^2\, d\V
\\=\mathcal{I}_0^\alpha+\mathcal{I}_b^\alpha+ \mathcal{R}_{C}^\alpha+\mathcal{R}_{S}^\alpha,
\eeq
where
\begin{align}
\label{I0}&\mathcal{I}_0^\alpha=\int_{z= 0}\( 2 \eps S^\varphi V^\alpha  - Q^\alpha \mbox{I} \) \N \cdot V^\alpha\, dy
\\
\label{Ib}&\mathcal{I}_b^\alpha=-\int_{z= -b}\( 2 \eps S^\varphi V^\alpha  - Q^\alpha \mbox{I} \) e_3 \cdot V^\alpha\, dy,
\\\label{RC} & \mathcal{R}_{C}^\alpha=\int_{\Omega}\( \( \D_{z}v  \cdot \nabla^\varphi v Z^\alpha \eta-  \mathcal{C}^\alpha(\mathcal{T})-
\mathcal{C}^\alpha(q)\) \cdot V^\alpha-\mathcal{C}^\alpha (d) Q^\alpha \) \, d\V,
\\
\label{RS}&\mathcal{R}_{S}^\alpha=\int_{\Omega} \( \eps \mathcal{D}^\alpha(S^\varphi v)  + \eps \nabla^\varphi \cdot \big( \mathcal{E}^\alpha (v)
\) \cdot V^\alpha\, d\V.
\end{align}

We first estimate $\mathcal{I}_0^\alpha$. The boundary condition \eqref{bordV} implies
\begin{align}\label{I0'}
\mathcal{I}_0^\alpha& =  \int_{z= 0} \big( 2 \eps S^\varphi V^\alpha  - Z^\alpha q \mbox{I} +\D_z q Z^\alpha\eta I\big) \N \cdot V^\alpha\,dy
\\& \nonumber =  \int_{z= 0}  - (g-\D_zq) Z^\alpha h \N\cdot V^\alpha\, dy+\int_{z= 0}  \sigma Z^\alpha H \N\cdot
V^\alpha\, dy+ \int_{z=0}  \eps \mathcal{C}^\alpha(\mathcal{B}_\eps)\cdot V^\alpha \, dy
\\\nonumber  &  \quad- \int_{z= 0} 2 \eps S^\varphi v\Pi Z^\alpha \N \cdot V^\alpha \, dy
-\int_{z=0} 2 \eps Z^\alpha h\, \D_{z} \big( S^\varphi v\big) \N  \cdot V^\alpha \, dy.
\end{align}
By \eqref{CalphaB}, the third term in the right hand side of \eqref{I0'} can be bounded by
\beq\label{m11}
\abs{\int_{z=0}  \eps \mathcal{C}^\alpha(\mathcal{B}_\eps)\cdot V^\alpha \, dy}\le \Lambda_{\infty}\eps \(\abs{v }_{\X^{m-1,1}}  + \abs{h
}_{\X^{m-1,1}}\)\abs{V^\alpha}_0.
 \eeq
Due to \eqref{cont2D}, it holds that
\begin{align}\label{m12}
\abs{ \int_{z= 0}2 \eps  S^\varphi v\Pi Z^\alpha \N \cdot V^\alpha \, dy}&\le 2\eps\abs{  Z^\alpha \nabla h }_{-{1 \over 2 }} \abs{S^\varphi v\Pi
V^\alpha}_{{1 \over 2}}
\\\nonumber&
\le \Lambda_{\infty}  \eps \abs{h}_{\X^{m-1,{3 \over 2}}} \abs{V^\alpha}_{\hal}.
\end{align}
Note that $\Lambda_\infty$ involves $\sqrt{\varepsilon}||\partial_{zz} v||_{L^\infty}$, one has
\begin{align}\label{m13}
\abs{ \int_{z=0} 2 \eps Z^\alpha h\, \D_{z} \( S^\varphi v\) \N  \cdot V^\alpha \, dy}
&\leq 2 \eps  \abs{Z^\alpha h}_{0}\norm{\D_{z} \( S^\varphi v\) \N }_{L^\infty} \abs{V^\alpha}_{0}
\\\nonumber&\leq \linf\eps^{\hal}  |h|_{\X^{m-1,1}} \abs{V^\alpha}_{0}.
\end{align}

For the first gravity term, one may use the boundary condition \eqref{bordh} to rewrite it as
 \begin{align}\label{gravity0}
  &\int_{z= 0}  - (g-\D_zq) Z^\alpha h \N\cdot
   V^\alpha\,dy
 \\\nonumber&\quad= \int_{z=0}  - (g-\D_zq) Z^\alpha h \(\partial_{t} Z^\alpha h+ v_y \cdot \nabla_{y}Z^\alpha h +\D_z v \cdot \N
 Z^\alpha h- \mathcal{C}^\alpha(h)\)\,dy.
    \end{align}
Integrating by parts in $t$ and using \eqref{qminfty} lead to
\beq\label{gravity1}
  \int_{z=0} -\(g-\D_zq\) \, Z^\alpha h  \partial_{t} Z^\alpha h\,dy
     \le -\hal \dtt  \int_{z=0} \( g- \D_{z} q\) \abs{Z^\alpha h }^2\,dy+\Lambda_\infty \abs{h}_{\X^{m-1,1}}^2.
\eeq
The integration by parts in $y$ gives
\beq\label{gravity2}
\abs{ \int_{z=0}  \(g-\D_zq\) Z^\alpha h v_y \cdot \nabla_{y}Z^\alpha h\, dy}\le   \Lambda_\infty\abs{h}_{\X^{m-1,1}}^2.
\eeq
Due to \eqref{bordhC}, one has
\beq\label{gravity3}
\abs{ \int_{z=0} \(g-\D_zq\) Z^\alpha h \(\D_z v \cdot \N Z^\alpha h- \mathcal{C}^\alpha(h)\)\, dy}
\leq  \Lambda_\infty    \abs{h}_{\X^{m-1,1}}\( \abs{h}_{\X^{m-1,1}} + \abs{ v}_{\X^{m-1}}\).
\eeq
Hence, in light of the estimates \eqref{gravity1}--\eqref{gravity3}, one may conclude from \eqref{gravity0} that
 \begin{align}\label{m14}
 &\int_{z= 0}  - (g-\D_zq) Z^\alpha h \N\cdot
   V^\alpha\,dy
 \\\nonumber&\quad\le - {1 \over 2} {d \over dt}  \int_{z=0} \( g- \D_{z} q\) \abs{Z^\alpha h }^2\,dy
+\Lambda_\infty\( \abs{h}_{\X^{m-1,1}}^2 +  \abs{ v}_{\X^{m-1}}^2\).
    \end{align}

To deal with the second term involving surface tension, one has by \eqref{sigsur1} that
  \begin{align}\label{sigmaes0}
   \int_{z= 0}  \sigma Z^\alpha H \N\cdot
   V^\alpha\,dy=\int_{z= 0}\sigma \nabla_y\cdot \( \frac{ \nabla_y Z^\alpha h}{\sqrt{1+|\nabla_y h|^2}}- \frac{\nabla_y h\cdot \nabla_y Z^\alpha
   h}{\sqrt{1+|\nabla_y h|^2}^3}\nabla_y h
  +\mathcal{C}(\mathcal{B}_\sigma^\alpha) \)\N\cdot  V^\alpha\,dy,
    \end{align}
where $ \mathcal{C}(\mathcal{B}_\sigma^\alpha)$ is defined by \eqref{sigsur2}. By Lemma \ref{lembordh}, one may deduce
  \begin{align}\label{sigmaes1}
   \int_{z= 0}  \sigma\nabla_y\cdot\mathcal{C}(\mathcal{B}_\sigma^\alpha) \N\cdot  V^\alpha\,dy
  \nonumber &\le \sigma\abs{  \nabla_y\cdot\mathcal{C}(\mathcal{B}_\sigma^\alpha)}_0\abs{\partial_{t} Z^\alpha h + v_y \cdot\nabla_y Z^\alpha h
  + \D_z v \cdot \N
 Z^\alpha h- \mathcal{C}^\alpha(h)}_0
  \\& \le \Lambda_\infty\sigma \abs{h}_{\X^{m-1,2}}\(\abs{ h}_{\X^{m ,1}}+\abs{h}_{\X^{m-1,2}}+\abs{h}_{\X^{m-1,1}}+\abs{v}_{\X^{m-1}}\)\nonumber
  \\& \le \Lambda_\infty \sigma \abs{h}_{\X^{m-1,2}}\(\abs{ h}_{\X^{m ,1}} +\abs{v}_{\X^{m-1}}\).
    \end{align}
To study the other two terms, one rewrite it as, by using the boundary condition \eqref{bordh} again,
 \begin{align} & \int_{z= 0}
   \sigma \nabla_y\cdot \(  \frac{ \nabla_y Z^\alpha h}{\sqrt{1+|\nabla_y h|^2}}- \frac{\nabla_y h\cdot \nabla_y Z^\alpha h}{\sqrt{1+|\nabla_y
   h|^2}^3}\nabla_y h\)\N\cdot V^\alpha  \,dy
   \\\nonumber&\quad = \int_{z=0}   \sigma \nabla_y\cdot \(  \frac{ \nabla_y Z^\alpha h}{\sqrt{1+|\nabla_y h|^2}}- \frac{\nabla_y h\cdot \nabla_y
   Z^\alpha h}{\sqrt{1+|\nabla_y h|^2}^3}\nabla_y h \)\(\partial_{t} Z^\alpha h+v_y \cdot \nabla_{y}Z^\alpha h\)\,dy
    +\mathcal{R}_{\mathcal{B}_\sigma^1}^\alpha,
    \end{align}
where
\beq
\mathcal{R}_{\mathcal{B}_\sigma^1}^\alpha= \int_{z=0}   \sigma \nabla_y\cdot \(  \frac{ \nabla_y Z^\alpha h}{\sqrt{1+|\nabla_y h|^2}}+
\frac{\nabla_y h\cdot \nabla_y Z^\alpha h}{\sqrt{1+|\nabla_y h|^2}^3}\nabla_y h \) \(\D_zv\cdot\N Z^\alpha h-   \mathcal{C}^\alpha(h) \) dy.
 \eeq
It follows from an integration by parts, \eqref{chtilde} and \eqref{bordhC1} that
 \begin{align}\label{sigmaes2}
 \mathcal{R}_{\mathcal{B}_\sigma^1}^\alpha&\ls \sigma\abs{   \frac{ \nabla_y Z^\alpha h}{\sqrt{1+|\nabla_y h|^2}}- \frac{\nabla_y h\cdot \nabla_y
 Z^\alpha h}{\sqrt{1+|\nabla h|^2}^3}\nabla_y h  }_{0}\abs{\nabla_y\(\D_zv\cdot\N Z^\alpha h-\tilde{\mathcal{C}}^\alpha(h)\)}_{0}
  \\\nonumber&\quad + \sigma \abs{\nabla_y\cdot \(  \frac{ \nabla_y Z^\alpha h}{\sqrt{1+|\nabla_y h|^2}}- \frac{\nabla_y h\cdot \nabla_y Z^\alpha
  h}{\sqrt{1+|\nabla h|^2}^3}\nabla_y h \)}_{-\hal}\abs{Z^{\alpha-\alpha_1} v_y\cdot \nabla_{y} Z^{ \alpha_1}h }_{\hal}
 \\&\nonumber \le \Lambda_\infty  \(\sigma\abs{h}_{\X^{m-1,2}}\(\abs{h}_{\X^{m-1,2}}+\abs{v}_{\X^{m-2,1}}\) +\sigma\abs{h}_{\X^{m-1,\frac{5}{2}}}
 \abs{v}_{\X^{m-1,\hal}} \).
    \end{align}
Integrating by parts in both $y$ and $t$, one finds that
\begin{align}\label{sigmaes3}
&\int_{z=0}   \sigma \nabla_y\cdot \(  \frac{ \nabla_y Z^\alpha h}{\sqrt{1+|\nabla_y h|^2}}- \frac{\nabla_y h\cdot \nabla_y Z^\alpha
h}{\sqrt{1+|\nabla h|^2}^3}\nabla_y h \)\partial_{t} Z^\alpha h\,dy
\\\nonumber&\quad=-\frac{1}{2}\frac{d}{dt}\int_{z=0}  \sigma\( \frac{  |\nabla_y Z^\alpha h|^2 }{\sqrt{1+|\nabla_y h|^2}}-  \frac{ |\nabla_y
h\cdot \nabla_y Z^\alpha h|^2  }{\sqrt{1+|\nabla_y h|^2}^3}   \)\,dy+\mathcal{R}_{\mathcal{B}_\sigma^2}^\alpha,
   \end{align}
where
\begin{align}\label{sigmaes4}
\mathcal{R}_{\mathcal{B}_\sigma^2}^\alpha
&=\hal\int_{z=0}  \sigma\(\partial_t\( \frac{ 1  }{\sqrt{1+|\nabla_y h|^2}}\)|\nabla_y Z^\alpha h|^2-  \partial_t\(\frac{ 1}{\sqrt{1+|\nabla_y
h|^2}^3}\) |\nabla_y h\cdot \nabla_y Z^\alpha h|^2\right.  \\\nonumber&\qquad\qquad\qquad\left. - \frac{\nabla_y h\cdot \nabla_y Z^\alpha
h}{\sqrt{1+|\nabla_y h|^2}^3} \nabla_y \partial_{t} h\cdot \nabla_y Z^\alpha h \)\,dy
\\\nonumber&\le \Lambda_\infty \sigma \abs{h}_{\X^{m-1,2}}^2.
   \end{align}
Similarly, the integration by parts twice yields
  \begin{align}\label{sigmaes5}
  - \int_{z=0}   \sigma \nabla_y\cdot \(  \frac{ \nabla_y Z^\alpha h}{\sqrt{1+|\nabla_y h|^2}}- \frac{\nabla_y h\cdot \nabla_y Z^\alpha
  h}{\sqrt{1+|\nabla h|^2}^3}\nabla_y h \)  v_y \cdot \nabla_{y}Z^\alpha h \, dy\le \Lambda_\infty \sigma \abs{h}_{\X^{m-1,2}}^2.
    \end{align}
Hence, by the estimates \eqref{sigmaes1}, \eqref{sigmaes2}--\eqref{sigmaes5}, one may conclude from \eqref{sigmaes0} that
 \begin{align}\label{m15}
 \int_{z= 0}  \sigma Z^\alpha H \N\cdot
   V^\alpha\, dy
 \le& -\frac{1}{2}\frac{d}{dt}\int_{z=0}  \sigma\( \frac{  |\nabla_y Z^\alpha h|^2 }{\sqrt{1+|\nabla_y h|^2}}-  \frac{ |\nabla_y h\cdot \nabla_y
 Z^\alpha h|^2  }{\sqrt{1+|\nabla_y h|^2}^3}   \)\,dy
 \\\nonumber&+\Lambda_\infty    \(\sigma\abs{h}_{\X^{m-1,2}} \abs{h}_{\X^{m,1}}+\sigma\abs{h}_{\X^{m-1,\frac{5}{2}}} \abs{v}_{\X^{m-1,\hal}} \).
    \end{align}
Note also that Lemma \ref{sobbord} implies that
\beq \label{equiv12}
\abs{V^\alpha}_0\ls \abs{v}_{\X^{m-1,1}}+\linf\abs{h}_{\X^{m-1,1}}\text{ and }\abs{V^\alpha}_\hal\ls
\abs{v}_{\X^{m-1,\frac{3}{2}}}+\linf\abs{h}_{\X^{m-1,\frac{3}{2}}}.
\eeq
Consequently, plugging the estimates \eqref{m11}--\eqref{m13}, \eqref{m14} and \eqref{m15} into \eqref{I0'}, by \eqref{equiv12} and Cauchy's
inequality, one may finish the
estimates of $\mathcal{I}_0^\alpha$ as:
 \begin{align}\label{I0estimate}
\mathcal{I}_0^\alpha
\le& -\frac{1}{2}\frac{d}{dt}\int_{z=0} \( g- \D_{z} q\) \abs{Z^\alpha h }^2+ \sigma\( \frac{  |\nabla_y Z^\alpha h|^2 }{\sqrt{1+|\nabla_y
h|^2}}-  \frac{ |\nabla_y h\cdot \nabla_y Z^\alpha h|^2  }{\sqrt{1+|\nabla_y h|^2}^3}   \)\,dy
 \\\nonumber&+\Lambda_\infty\( \abs{h}_{\X^{m-1,1}}^2 +    \abs{ v}_{\X^{m-1}}^2 + \sigma   \abs{h}_{\X^{m-1,2}}
 \abs{h}_{\X^{m,1}}+\sigma\abs{h}_{\X^{m-1,\frac{5}{2}}} \abs{v}_{\X^{m-1,\hal}}    \right.
   \\\nonumber& \qquad\quad\left.+
    \eps\(\abs{v }_{\X^{m-1,1}}^2
  +  \abs{h}_{\X^{m-1,{3 \over 2}}}^2+ \abs{h}_{\X^{m-1,\frac{3}{2}}} \abs{v}_{\X^{m-1,\frac{3}{2}}}\) \).
    \end{align}

It also follows from \eqref{eeest} and \eqref{equiv12} that $\mathcal{I}_b^\alpha$ admits the following bound:
\begin{align}\label{Ibestimate}
\mathcal{I}_b^\alpha&=-\int_{z= -b}\( 2 \eps S^\varphi V^\alpha e_3 \)_i \cdot V_i^\alpha\, dy=-\int_{z= -b}2\eps\( \kappa
V^\alpha_i-\mathcal{E}^\alpha(v)_{i3}\)  \cdot V_i^\alpha\, dy
 \\\nonumber&\le \linf \eps \( \abs{V^\alpha}_0+\abs{ v }_{\X^{m-1,1}} +  \abs{h}_{\X^{m-1,1}}\)\abs{V^\alpha}_0
  \\\nonumber&\le \linf \eps \( \abs{ v }_{\X^{m-1,1}}^2 +  \abs{h}_{\X^{m-1,1}}^2\).
\end{align}

Next, the commutator $\mathcal{R}_C$ is estimated by using \eqref{CT}, \eqref{Cq}, \eqref{Cd}, \eqref{equiv1}, \eqref{qpressurem} and
 \eqref{qminfty} as
  \begin{align}\label{RCestimate}
  \mathcal{R}_{C}^\alpha &\le \linf\(\( \norm{ Z^\alpha \eta}+\norm{\mathcal{C}^\alpha(\mathcal{T})}+\norm{\mathcal{C}^\alpha(q)}\) \norm{
  V^\alpha}+\norm{\mathcal{C}^\alpha (d)}\norm{Q^\alpha}\)
  \\\nonumber& \le \Lambda_\infty\(\abs{h}_{\X^{m,-\hal}}+ \norm{v}_{\X^{m-1}}+\norm{\nabla v}_{\X^{m-1}}+ \norm{\nabla
  q}_{\X^{m-1}}\)\(\norm{v}_{\X^{m-1,1}}+ \abs{h}_{\X^{m-1,\hal}}\)
  \\\nonumber&\quad+\Lambda_\infty \(\norm{\nabla v}_{\X^{m-1}}+ |h|_{\X^{m-1,{1 \over 2 }}}\)\(\norm{q}_{\X^{m-1,1}}+ \abs{h}_{\X^{m-1,\hal}}\)
   \\\nonumber &\le \linf\(\abs{h}_{\X^{m,-{1 \over 2 }}}+\sigma \abs{h}_{\X^{m-1,2}}+\norm{v}_{\X^{m}}  + \norm{\nabla v}_{\X^{m-1}} +
   \eps\abs{  v }_{{\X^{m-1, \frac{3}{2}}}}+\eps\abs{h}_{{\X^{m-1,\frac{3}{2}}}}\)
  \\\nonumber&\quad\times \(\norm{v}_{\X^{m-1,1}}+\norm{\nabla v}_{\X^{m-1}}+ \abs{h}_{\X^{m-1,\hal}}\).
  \end{align}

 It remains to estimate the commutator $\mathcal{R}_{S}^\alpha$. First, it follows from the integration by parts, \eqref{CE}
 and \eqref{eeest} that
\begin{align}\label{RSe}
&\int_{\Omega} \eps \nabla^\varphi( \mathcal{E}^\alpha(v) \big) \cdot V^\alpha d \V=- \int_{\Omega} \eps\, \mathcal{E}^\alpha(v) \cdot \nabla
V^\alpha d \V+ \int_{z=0}\eps\,  \mathcal{E}^\alpha(v) \N \cdot V^\alpha \, dy
\\\nonumber&\quad\leq \Lambda_{\infty}\eps\(\( \norm{\nabla v}_{\X^{m-1}}+  \abs{h}_{\X^{m-1,\hal}}\)  \norm{\nabla V^\alpha} + \(
\abs{h}_{\X^{m-1,1}} + \abs{v}_{\X^{m-1,1}}\)\abs{V^\alpha}_0\).
\end{align}
Next, for the first term, one actually has to estimate
\begin{align}\label{RSi}
\mathcal{R}_{Si}^\alpha  & =  \eps \int_{\Omega} \mathcal C^\alpha_{j}(S^\varphi v)_{ij}  V^\alpha_{j} d \V \\
&\nonumber=\eps \int_{\Omega} \mathcal C^\alpha_{j, 1}(S^\varphi v)_{ij}  V^\alpha_{j} d \V+ \eps \int_{\Omega} \mathcal C^\alpha_{j,
2}(S^\varphi v)_{ij}  V^\alpha_{j} d \V+  \eps \int_{\Omega} \mathcal C^\alpha_{j, 3}(S^\varphi v)_{ij}  V^\alpha_{j} d \V \\
&\nonumber: = \mathcal{R}_{Si}^{\alpha,1} + \mathcal{R}_{Si}^{\alpha,2} + \mathcal{R}_{Si}^{\alpha,3}
\end{align}
due to \eqref{Cialpha}. For $\mathcal{R}_{Si}^{\alpha,1}$, by \eqref{Cialpha1}, it suffices to estimate terms like
$$\eps  \int_{\Omega}  Z^\beta \big( {\partial_{j} \varphi \over \partial_{z} \varphi}\big) \big(Z^{\tilde{\gamma}} \partial_{z}(S^\varphi
v)_{ij} \big)V^\alpha_{j} d \V,$$
where $\beta $ and $\tilde\gamma$ are such that $\beta \neq 0, \, \tilde \gamma \neq 0$ and $|\beta | + |\tilde \gamma|=m.$ By using
\eqref{idcom}, one can reduce the problem to the estimate of
$$ \eps  \int_{\Omega}  c_{\gamma} Z^\beta \big( {\partial_{j} \varphi \over \partial_{z} \varphi}\big)  \partial_{z}\big( Z^\gamma (S^\varphi
v)_{ij} \big)V^\alpha_{j} d \V$$
with $\beta$ as before (thus $| \beta | \leq m-1$) and $| \gamma | \leq | \tilde \gamma|\leq m-1.$ The integration by parts shows that it
suffices
to estimate three types of terms:
$$  \mathcal{I}_{1}=  \eps\int_{\Omega} Z^\beta \big( {\partial_{j} \varphi \over \partial_{z} \varphi } \big)\,Z^\gamma (S^\varphi
v)_{ij}\,\partial_{z} V_{j}^\alpha d \V,$$
$$\mathcal{I}_{2}=\eps\int_{\Omega} \Big( \partial_{z}Z^\beta \big( {\partial_{j} \varphi \over \partial_{z} \varphi }  \big) \Big)Z^\gamma
(S^\varphi v)_{ij} V_{j}^\alpha d \V,$$
and
$$ \mathcal{I}_{3}= \eps \int_{z=0} Z^\beta \big( {\partial_{j} \varphi \over \partial_{z} \varphi } \big)\, Z^\gamma (S^\varphi v)_{ij}
V_{j}^\alpha dy.$$
For $\mathcal{I}_1$ and $\mathcal{I}_2$,  since $\beta \neq 0$, it follows from \eqref{gues}, \eqref{quot} and Lemma \ref{propeta} that
\beq\nonumber
\abs{\mathcal{I}_{1}} \leq\Lambda_{\infty}\eps\(  \norm{ \nabla v}_{\X^{m-1}}+\abs{h}_{\X^{m-1,{1 \over 2}}} \) \norm{ \nabla V^\alpha}
\eeq
and
$$\abs{ \mathcal{I}_{2}} \le\Lambda_{\infty}\eps\(   \norm{ \nabla v}_{\X^{m-1}}+\abs{h}_{\X^{m-1,{3 \over 2}}} \) \norm{   V^\alpha}.$$
By \eqref{gues}, $\beta \neq 0$ and Lemma \ref{lembord}, it holds that
$$\abs{\mathcal{I}_{3}}\leq\linf\eps\( \abs{h}_{\X^{m-1,1} } +  \abs{\nabla v}_{\X^{m-1}} \)\abs{V^\alpha}_0 \leq\linf\eps\( \abs{h}_{\X^{m-1,1}
} +  \abs{ v}_{\X^{m-1,1}} \)\abs{V^\alpha}_0 .$$
Consequently, one can get from the previous three estimates that
\begin{align}\label{RSi1}
\abs{\mathcal{R}_{Si}^{\alpha,1}}\leq\Lambda_{\infty}\eps&\(\(  \norm{ \nabla v}_{\X^{m-1}}+  \abs{h}_{\X^{m-1,\hal }} \) \norm{ \nabla V^\alpha}
+\(   \norm{ \nabla v}_{\X^{m-1}}+\abs{h}_{\X^{m-1,{3 \over 2}}} \) \norm{   V^\alpha}\right.
\\\nonumber&\left.+ \( \abs{h}_{\X^{m-1,1} } + \abs{ v}_{\X^{m-1,1}}  \)\abs{V^\alpha}_0 \) .
\end{align}
 The estimate of $\mathcal{R}_{Si}^{\alpha,2}$ is straightforward, one gets from the definition \eqref{Cialpha2} that
\beq\label{RSi2}
\abs{\mathcal{R}_{Si}^{\alpha,2}}\leq\Lambda_{\infty}\eps^{\hal}\abs{h}_{\X^{m-1,\hal }}\norm{V^\alpha}.
\eeq
To estimate $\mathcal{R}_{Si}^{\alpha,3}$, one derives from \eqref{Cialpha3} and \eqref{idcom} that
$$\eps \abs{\int_{\Omega}  {\partial_{i}\varphi \over (\partial_{z}\varphi)^2} \partial_{z}\big( S^\varphi v)  [Z^\alpha, \partial_{z}]\varphi
V_{j}^\alpha d\V} \leq\Lambda_{\infty}\eps^{\hal}\abs{h}_{\X^{m-1,\hal }}\norm{V^\alpha}.$$
Note that one has used again in the previous two estimates the fact that $\Lambda_\infty$ involves $
\eps^{\hal}\norm{\partial_{zz}v}_{L^\infty}$. For the term
$$\eps \int_{\Omega}{\partial_{i}\varphi \over \partial_{z} \varphi}  V^\alpha \, [Z^\alpha, \partial_{z}](S^\varphi v)  d\V, $$
performing an integration by parts and using a similar arguments as for $\mathcal{R}_{Si}^{1}$ show that
$$\abs{\int_{\Omega}{\partial_{i}\varphi \over \partial_{z} \varphi}  V^\alpha \, [Z^\alpha, \partial_{z}](S^\varphi v)  d\V}\leq\linf \eps
\norm{\nabla v}_{\X^{m-1}}\( \norm{V^\alpha} + \norm{\nabla  V^\alpha}\).$$
Consequently,
\beq\label{RSi3}
\abs{\mathcal{R}_{Si}^{\alpha,3}}\leq\linf\( \eps^{\hal}\abs{h}_{\X^{m-1,\hal }}\norm{V^\alpha}+\eps\norm{\nabla v}_{\X^{m-1}}\( \norm{V^\alpha}
+ \norm{\nabla  V^\alpha}\) \) .
\eeq
It then follows from \eqref{RSi1}--\eqref{RSi3} that
\begin{align}
\eps\abs{ \int_{\Omega} \mathcal D^\alpha(S^\varphi v) \cdot V^\alpha d \V}&\leq\linf\(\(\eps^{\hal}\abs{h}_{\X^{m-1,\hal}}+\eps
\abs{h}_{\X^{m-1, \frac{3}{2}}}+\eps\norm{\nabla v}_{\X^{m-1}}\)\norm{V^\alpha}\right.
\\\nonumber&\left.\quad +\eps\(  \norm{ \nabla v}_{\X^{m-1}}+  \abs{h}_{\X^{m-1,\hal }} \) \norm{ \nabla V^\alpha}+ \( \abs{h}_{\X^{m-1,1} } +
\abs{ v}_{\X^{m-1,1}} \)\abs{V^\alpha}_0 \).
\end{align}
This, together with \eqref{RSe}, \eqref{equiv1} and \eqref{equiv12}, implies that
\begin{align}\label{RSestimate}
\mathcal{R}_{S}^\alpha\leq&\linf\(\(\eps^{\hal}\abs{h}_{\X^{m-1,\hal}}+\eps \abs{h}_{\X^{m-1, \frac{3}{2}}}+\eps\norm{\nabla
v}_{\X^{m-1}}\)\(\norm{v}_{\X^{m-1,1}}+\abs{h}_{\X^{m-1,\hal}}\)\right.
\\\nonumber&\left.\qquad +\eps\(  \norm{ \nabla v}_{\X^{m-1}}+  \abs{h}_{\X^{m-1,\hal }} \) \norm{ \nabla V^\alpha}+ \abs{h}_{\X^{m-1,1} }^2 +
\abs{  v}_{\X^{m-1,1}}^2\).
\end{align}

We can now finish the proof of the proposition. By the estimates \eqref{I0estimate}--\eqref{RCestimate} and \eqref{RSestimate},  the trace
estimates
\beq
\abs{v}_{\X^{m-1,\hal}}\ls \norm{v}_{\X^{m-1}}+\norm{\nabla v}_{\X^{m-1}}\text{ and }\abs{v}_{\X^{m-1,1}}\ls \norm{\nabla
v}_{\X^{m-1,1}}^\hal\norm{v}_{\X^{m-1,1}}^\hal+\norm{v}_{\X^{m-1,1}} ,
\eeq
 using Cauchy's inequality, one may deduce from \eqref{conormen} that
\begin{align}\label{alphaend1}
&\hal\dtt \mathcal{E}^\alpha + 2 \eps \int_{\Omega} \abs{S^\varphi V^\alpha}^2\, d\V
\\\nonumber&\quad \le\Lambda_\infty\( \abs{h}_{\X^{m-1,1}}^2 + \sigma   \abs{h}_{\X^{m-1,2}}
\abs{h}_{\X^{m,1}}+\sigma\abs{h}_{\X^{m-1,\frac{5}{2}}}\(\norm{v}_{\X^{m-1,1}}+\norm{\nabla v}_{\X^{m-1}}\) \right.
  \\\nonumber&\qquad\qquad+\abs{h}_{\X^{m,-\hal}}^2+ \norm{\nabla v}_{\X^{m-1}}^2 + \norm{v}_{\X^{m}}^2 + \eps  \abs{h}_{\X^{m-1, \frac{3}{2}}}^2
  \\\nonumber&\qquad\qquad+ \eps\abs{  v }_{{\X^{m-1, \frac{3}{2}}}} \(\norm{v}_{\X^{m-1,1}}+\norm{\nabla
  v}_{\X^{m-1}}+\abs{h}_{{\X^{m-1,\frac{3}{2}}}}\)
\\\nonumber&\left.\qquad \qquad+\eps\(  \norm{ \nabla v}_{\X^{m-1}}+  \abs{h}_{\X^{m-1,\hal }} \) \norm{ \nabla V^\alpha}  \),
\end{align}
where
\beq
\mathcal{E}^\alpha:=\int_{\Omega}\abs{V^\alpha }^2 d\V +\int_{z=0} \( g- \D_{z} q\) \abs{Z^\alpha h }^2+ \sigma\( \frac{  |\nabla_y Z^\alpha h|^2
}{\sqrt{1+|\nabla_y h|^2}}-  \frac{ |\nabla_y h\cdot \nabla_y Z^\alpha h|^2  }{\sqrt{1+|\nabla_y h|^2}^3}   \)\,dy.
\eeq
It follows from \eqref{equiv1}, the Taylor sign condition in \eqref{apriori} and \eqref{sigsur3} that
$$\norm{Z^\alpha v}^2+\abs{Z^\alpha h}_0^2+\sigma\abs{Z^\alpha h}_1^2 \leq \Lambda\({1\over c_{0}}\)  \mathcal{E}^\alpha.$$
One can use the Korn inequality of Lemma \ref{Korn} and \eqref{equiv1} to get that
$$\norm{\nabla V^\alpha}^2 \leq \Lambda\({1\over c_{0}}\) \(\int_{\Omega} \abs{S^\varphi V^\alpha}^2\, d\V + \norm{v}_{\X^{m-1,1}}^2+
\abs{h}_{\X^{m-1,\hal}}^2\).$$
On the other hand, by the definition of $V^\alpha$,
$$\eps \norm{\nabla Z^\alpha v}^2\le \eps \norm{\nabla V^\alpha}^2+\linf\(\eps\norm{h}_{\X^{m-1,\frac{3}{2}}}^2+\norm{h}_{\X^{m-1,\hal}}^2\).$$
Then integrating \eqref{alphaend1} in time, using the trace estimate $\abs{  v }_{{\X^{m-1, \frac{3}{2}}}}\ls \norm{\nabla v }_{{\X^{m-1,
1}}}+\norm{  v }_{{\X^{m-1, 1}}}$ and Cauchy's inequality, together with \eqref{l2estimate}, one deduces that
\begin{align}\label{alphaend2}
&\norm{v(t)}_{\X^{m-1,1}}^2+\abs{h(t)}_{\X^{m-1,1}}^2 +\sigma\abs{h(t)}_{\X^{m-1,2}}^2+\eps\int_0^t  \norm{\nabla v}_{\X^{m-1,1}}^2
\\\nonumber&\quad\le \Lambda_0\( \norm{v(0)}_{\X^{m-1,1}}^2+\abs{h(0)}_{\X^{m-1,1}}^2 +\sigma\abs{h(0)}_{\X^{m-1,2}}^2\)
\\\nonumber&\qquad+ \int_0^t \Lambda_\infty\( \abs{h}_{\X^{m-1,1}}^2 + \sigma   \abs{h}_{\X^{m-1,2}}
\abs{h}_{\X^{m,1}}+\sigma\abs{h}_{\X^{m-1,\frac{5}{2}}}\(\norm{v}_{\X^{m-1,1}}+\norm{\nabla v}_{\X^{m-1}}\)  \right.
  \\\nonumber&\qquad\qquad\left.+\abs{h}_{\X^{m,-\hal}}^2+ \norm{\nabla v}_{\X^{m-1}}^2 + \norm{v}_{\X^{m}}^2 + \eps  \abs{h}_{\X^{m-1,
  \frac{3}{2}}}^2\).
\end{align}
This, together with \eqref{mheps1}, \eqref{mhsig} and Cauchy's inequality, leads to \eqref{conormales1}.
\end{proof}

\subsection{Estimate of $(\dt^m v, \dt^m h)$}\label{secconormal2}

We now derive the energy estimates of $(\dt^m v, \dt^m h)$, that is, the case $\alpha_0=m$.
\begin{prop}\label{conormvm}
Any smooth solution of \eqref{NSv} satisfies the estimate
 \begin{align}\label{alphaend2m1ll}
&\int_0^t\(\norm{\dt^m v }_0^2+\abs{\dt^m h }_{0}^2 +\sigma\abs{\dt^m h }_{1}^2+ \eps  \abs{h }_{\X^{m , \hal}}^2 \)^2+\int_0^t\(\eps\int_0^s
\norm{\nabla \dt^m v}_{0}^2\)^2
\\\nonumber&\quad\le t\Lambda_0\( \norm{\dt^m v(0)}_0^2+\abs{\dt^m h(0)}_{0}^2 +\sigma\abs{\dt^m h(0)}_{1}^2+ \eps  \abs{h(0)}_{\X^{m ,
\hal}}^2\)^2
\\\nonumber&\qquad+ t\(\int_0^t\Lambda_\infty\( \abs{h}_{\X^{m }}^2 + \sigma     \abs{h}_{\X^{m,1}}^2
  + \norm{v}_{\X^{m}}^2 + \norm{\pa_z v}_{\X^{m-1}}^2+ \eps  \norm{\nabla v}_{\X^{m-1,1}}^2+ \eps  \abs{h}_{\X^{m , \hal}}^2  \)\)^2
   \\\nonumber &\qquad +\int_0^t\linf \(\abs{h}_{\X^{m-1,\hal }}^2+\sigma \abs{h}_{\X^{m-1,2 }}^2+\norm{v}_{\X^{m-1,1 }}^2+\norm{\pa_z
   v}_{\X^{m-2 }}^2\)  \\\nonumber&\qquad\qquad \times\(\abs{h}_{\X^{m,-{1 \over 2 }}}^2+\sigma \abs{h}_{\X^{m,1}}^2 +\norm{v}_{\X^{m}}^2  +
   \norm{\pa_z v}_{\X^{m-1}}^2 +  \eps^2\abs{  v }_{{\X^{m-1, \frac{3}{2}}}}^2+\eps^2\abs{h}_{{\X^{m-1,\frac{3}{2}}}}^2\).
\end{align}

\end{prop}

\begin{proof}
In the current case, \eqref{conormen} can be restated as:
\beq\label{conormenm}
\hal\dtt\int_{\Omega} \abs{V^m }^2 d\V + 2 \eps \int_{\Omega} \abs{S^\varphi V^m}^2\, d\V \\
= \mathcal{I}_0^m+\mathcal{I}_b^m+ \mathcal{R}_{Q}^m+ \mathcal{R}_{C}^m+\mathcal{R}_{S}^m,
\eeq
where
\begin{align}
\label{Im0} &\mathcal{I}_0^m=  \int_{z= 0} \( 2 \eps S^\varphi V^m  - Q^m \mbox{I} \) \N \cdot V^m\, dy,\\
\label{Imb} &\mathcal{I}_b^m= - \int_{z= 0} \( 2 \eps S^\varphi V^m  - Q^m \mbox{I} \) e_3 \cdot V^m\, dy, \\
\label{RQm} & \mathcal{R}_{Q}^m=   - \int_{\Omega}    \mathcal{C}^m (d) \dt^m q  \, d\V,\\
\label{RCm} & \mathcal{R}_{C}^m=   \int_{\Omega}\( \( \D_{z}v  \cdot \nabla^\varphi v \dt^m \eta-  \mathcal{C}^m(\mathcal{T})- \mathcal{C}^m(q)\)
\cdot V^m+\mathcal{C}^m (d) \D_zq\dt^m\eta \) \, d\V,\\
\label{RSm} & \mathcal{R}_{S}^m=   \int_{\Omega} \( \eps \mathcal{D}^m(S^\varphi v)  + \eps \nabla^\varphi \cdot  \mathcal{E}^m (v) \) \cdot
V^m\, d\V.
\end{align}
Here $V^m=\dt^m v-\D_z v\dt^m\eta,\ Q^m=\dt^m q-\D_z v\dt^m\eta$, and $\mathcal{C}^m(\cdot)$ are those commutators $\mathcal{C}^\alpha(\cdot)$
for the case $\alpha_0=m$. Note that we have singled out the term $\mathcal{R}_{Q}^m$ from $\mathcal{R}_{C}^m$.

We first estimate $\mathcal{I}_0^m$, which can be rewritten as (similar to \eqref{I0'})
\begin{align}\label{I0'm}
\mathcal{I}_0^m&   =  \int_{z= 0}  - (g-\D_zq) \dt^m h \N\cdot V^m\, dy+\int_{z= 0}  \sigma \dt^m H \N\cdot V^m\, dy+ \int_{z=0}  \eps
\mathcal{C}^m(\mathcal{B}_\eps)\cdot V^m \, dy
\\\nonumber&  \quad- \int_{z= 0} 2 \eps S^\varphi v\Pi \dt^m \N \cdot V^m \, dy-\int_{z=0} 2 \eps \dt^m h\, \D_{z} \big( S^\varphi v\big) \N
\cdot V^m \, dy.
\end{align}
Following the analysis in \eqref{m11}--\eqref{m13}, one can bound the last three terms in \eqref{I0'm} by
\beq\label{m11m}
\Lambda_{\infty}\(\eps \(\abs{v }_{\X^{m-1,1}}  + \abs{h }_{\X^{m-1,1}}\)\abs{V^m}_0 +\eps \abs{ h}_{\X^{m,\hal}} \abs{V^m}_{\hal}+\eps^\hal
\abs{ h}_{\X^{m}}\abs{V^m}_{0}\).
\eeq
As \eqref{m14}, one deduces
\begin{align}\label{m11m1}
& \int_{z= 0}  - (g-\D_zq) \dt^m h \N\cdot V^m\,dy
\\\nonumber&\quad\le - \hal \dtt  \int_{z=0} \( g- \D_{z} q\) \abs{\dt^m h }^2\,dy
+\linf\(\abs{h}_{\X^{m}}^2  +  \abs{ v}_{\X^{m-1}}^2 \).
\end{align}
However, as explained in Section \ref{sec main}, one can not use the arguments leading to \eqref{m15} to estimate $\int_{z= 0}  \sigma \dt^m H
\N\cdot V^m\, dy$ since there is one half
regularity loss for $\dt^m h$ so that it is difficult to control the following term, after using the kinematic boundary condition,
\beq\label{surout}
 -\int_{z= 0} \sigma \pa_t^m  H\,  m\pa_t  \N  \cdot \pa_t^{m-1} v \, dy.
\eeq
The crucial observation here is that, this term will be cancelled out from estimating the term $\mathcal{R}^m_Q$ defined by \eqref{RQm}.
So the estimates of this term will be postponed till we estimate $\mathcal{R}^m_Q$.

Similarly as \eqref{Ibestimate}, $\mathcal{I}_b^m$ admits the bound
\beq\label{Ibestimatem}
\mathcal{I}_b^m \le \linf \eps \( \abs{V^m}_0+\abs{ v }_{\X^{m-1,1}}  +  \abs{h}_{\X^{m-1,1}}\)\abs{V^m}_0.
\eeq

We now estimate $\mathcal{R}^m_Q$. Note carefully that there is no any estimates of $\dt^m q$, so one needs to integrate by parts in $t$.
To continue, one needs more explicit expression of $\mathcal{C}^m (d)$. Indeed, we will use a variant of \eqref{Cialpha}. It follows from
the divergence free condition that
$$\pa_z\varphi\(\pa_1v_1+\pa_2v_2\)+\pa_zv\cdot\N=0.$$
Applying $\dt^m$ to the above and using the definition of $\mathcal{C}^m(d)$, one gets that
 \begin{align}\label{com1m}
  \pa_z\varphi \mathcal{C}^m(d)   =  \[ \pa_t^m ,  \N , \cdot\partial_{z}v  \]  +   \[ \pa_t^m ,  \partial_{z} \eta,\pa_1v_1+\pa_2v_2\].
     \end{align}
 Moreover, to integrate by parts in $t$, one needs to single out in $\mathcal{C}^m(d) $ the highest $m-1$ order time
 derivatives terms and use the following decomposition
\begin{align}
 \pa_z\varphi\mathcal{C}^m (d)=  \mathcal{C} ^{m} (d)_1+\mathcal{C} ^{m} (d)_2+\mathcal{C} ^{m} (d)_3+\mathcal{C} ^{m} (d)_4+\mathcal{C} ^{m}
 (d)_5
     \end{align}
   with
\begin{align}
\mathcal{C} ^{m} (d)_1= & m\pa_t  \N  \cdot \pa_t^{m-1} \partial_{z} v,
\\ \mathcal{C} ^{m} (d)_2 =&m\pa_t \partial_{z} \eta
\pa_t^{m-1}(\pa_1v_1+\pa_2v_2), \\
\mathcal{C} ^{m} (d)_3
  = & m \pa_t^{m-1}\N\cdot \pa_t \partial_{z} v,\\
\mathcal{C} ^{m} (d)_4
  = &m \pa_t^{m-1}\pa_z\eta \pa_t (\pa_1v_1+\pa_2v_2),\\
 \mathcal{C} ^{m} (d)_5
 =   & \sum_{\ell=2 }^{m-2} C_m^\ell\(\pa_t^\ell \N  \cdot\pa_t^{m-\ell}\partial_{z} v+\pa_t^\ell \partial_{z} \eta
 \cdot\pa_t^{m-\ell}(\pa_1v_1+\pa_2v_2)\).
     \end{align}
Accordingly,
 \beq\label{rmq0}
\mathcal{R}^m_Q= -\int_{\Omega} \(\mathcal{C} ^{m} (d)_1+\mathcal{C} ^{m} (d)_2+\mathcal{C} ^{m} (d)_3+\mathcal{C} ^{m} (d)_4+\mathcal{C} ^{m}
(d)_5\) \pa_t^m q   \, dydz.
  \eeq
The fifth term in \eqref{rmq0} can be easily treated by the integration by parts in $t$ as
 \begin{align}\label{rmq11}
 -\int_{\Omega}  \mathcal{C} ^{m} (d)_5 \pa_t^m q   \, dydz
 =-\dtt\int_{\Omega}   \mathcal{C} ^{m} (d)_5 \pa_t^{m-1} q\,dydz+\mathcal{R}_5^m,
  \end{align}
  with
   \begin{align}\label{rmq11'}
\mathcal{R}_5^m=\int_{\Omega} \dt \mathcal{C} ^{m} (d)_5   \pa_t^{m-1} q\,dydz \le \linf\( \abs{h}_{\X^{m-1,\hal}}+\norm{\pa_z
v}_{\X^{m-2}}\)\norm{ \pa_t^{m-1} q}.
  \end{align}
Integrate by parts in $t$ to write the fourth term as
 \beq\label{rmq12}
-\int_{\Omega}  \mathcal{C} ^{m} (d)_4\pa_t^m q   \, dydz
=-\dtt\int_{\Omega}   \mathcal{C} ^{m} (d)_4 \pa_t^{m-1} q\,dydz+\mathcal{R}_4^m,
\eeq
where, by further integrating by parts in $z$ and the trace theory,
 \begin{align}\label{rmq12'}
  \mathcal{R}_4^m&=\int_{\Omega}m \(\pa_t^{m}\pa_z\eta \pa_t (\pa_1v_1+\pa_2v_2) +  \pa_t^{m-1}\pa_z\eta \pa_t^2 (\pa_1v_1+\pa_2v_2) \)
  \pa_t^{m-1} q\,dydz
 \\\nonumber& = \int_{z=0}   m \pa_t^{m}h  \dt\( \pa_1 v_1+\pa_2 v_2\)\pa_t^{m-1} q  \, dy -\int_{\Omega}   m\pa_t^{m}\eta\pa_z\( \( \pa_1
 v_1+\pa_2 v_2\)\pa_t^{m-1} q \) \, dydz \\\nonumber&\quad +\int_{\Omega}m
  \pa_t^{m-1}\pa_z\eta \pa_t^2 (\pa_1v_1+\pa_2v_2)  \pa_t^{m-1} q\,dydz
\\\nonumber& \le  \linf\(\abs{\dt^m h}_{-\hal}\abs{\dt^{m-1} q}_{\hal}+\norm{\pa_t^{m} \eta}\norm{\pa_t^{m-1} q}_{H^1}+\norm{\pa_t^{m-1}\pa_z
\eta}\norm{\pa_t^{m-1} q} \)
\\\nonumber& \le \linf \abs{h}_{\X^{m,-\hal}} \norm{\pa_t^{m-1} q}_{H^1}  .
 \end{align}
Similarly, integrate by parts in both $t$ and $y$ to bound the second and third terms by
 \begin{align}\label{rmq13}
 -\int_{\Omega}  \(\mathcal{C} ^{m} (d)_2+\mathcal{C} ^{m} (d)_3\) \pa_t^m q   \, dydz
 =-\dtt\int_{\Omega}   \(\mathcal{C} ^{m} (d)_2+\mathcal{C} ^{m} (d)_3\)  \pa_t^{m-1} q\,dydz+\mathcal{R}_{2,3}^m,
  \end{align}
  where
   \begin{align}\label{rmq13'}
\mathcal{R}_{2,3}^m &=\int_{\Omega}  \dt \(\mathcal{C} ^{m} (d)_2+\mathcal{C} ^{m} (d)_3\)  \pa_t^{m-1} q\,dydz
\\\nonumber&
\le \linf\( \abs{h}_{\X^{m-1,\hal}}+\norm{  v}_{\X^{m}}+\norm{\nabla  v}_{\X^{2}}\)\norm{\pa_t^{m-1} q}_{H^1}.
  \end{align}
Finally, we turn to the most delicate term, the one involving $\mathcal{C}^m(d)_1$ in \eqref{rmq0}. Integrate by parts in $z$ first to get
 \beq\label{rmq14}
-\int_{\Omega}   \mathcal{C}^m (d)_1  \pa_t^{m} q\,dydz
 =-\int_{ {z=0}}   m\pa_t  \N  \cdot \pa_t^{m-1} v  \pa_t^{m} q \, dy+\int_{\Omega}   m  \pa_z\( \pa_t^{m} q\pa_t  \N\)\cdot \pa_t^{m-1} v  \,
 dydz.
  \eeq
 Then integrate by parts in $t$ to obtain
\beq
 \int_{\Omega}   m  \pa_z\( \pa_t^{m} q\pa_t  \N\)\cdot \pa_t^{m-1} v  \, dydz
=\dtt\int_{\Omega}   m \pa_z\( \pa_t^{m-1} q\pa_t  \N\)\cdot \pa_t^{m-1} v  \, dydz+\mathcal{R}^m_1,
\eeq
with
\begin{align}
\mathcal{R}^m_1&= -\int_{\Omega}   m \pa_z\( \pa_t^{m-1} q\pa_t  \N\)\cdot \pa_t^{m} v + m \pa_z\( \pa_t^{m-1} q\pa_t^2  \N\)\cdot \pa_t^{m-1} v
\, dydz
\\&\le   \linf \norm{  v}_{\X^{m}}
\norm{\pa_z\pa_t^{m-1} q} .\nonumber
  \end{align}
Note carefully that we integrate by parts in $z$ first rather than in $t$ since there is no estimates of $\dt^{m}v$ on the boundary.
This also indicates the difficulty in controlling the first term in the right hand side of \eqref{rmq14} since one can no longer integrate by
parts in $t$. Recall here that there was also
one term out of control, that is, \eqref{surout}.
Our crucial observation is that there is a  cancelation between them since $ q=g h-\sigma H+2\eps S^\varphi v  \,{\bf n} \cdot \,{\bf n} $ on
$\{z=0\}$. This motivates us to estimate together the first term in \eqref{rmq14} and the second surface tension term in \eqref{I0'm}, by the
kinematic boundary condition,
  \begin{align}\label{rmq15}
  & \int_{z= 0}   \sigma \pa_t^m H    \N \cdot V^m \, dy-\int_{ {z=0}}   m\pa_t  \N  \cdot \pa_t^{m-1} v  \pa_t^{m} q \, dy
  \\\nonumber&\quad= \int_{z= 0} \sigma \pa_t^m  H \(\N \cdot V^m+  m\pa_t  \N  \cdot \pa_t^{m-1} v\)\, dy
   \\\nonumber&\qquad-\int_{ {z=0}}   m\pa_t  \N  \cdot \pa_t^{m-1} v  \(g\dt^mh +2\eps \dt^m \( S^\varphi v  \,{\bf n} \cdot \,{\bf n}\)\)\, dy
  \\\nonumber&\quad= \int_{z= 0} \sigma \pa_t^m  H \(\pa_t^{m+1}h+v_y\cdot\nabla_y \pa_t^m h+\D_zv\cdot\N \dt^m h- \tilde{\mathcal{C}}^m(h) \)\,
  dy
  \\\nonumber&\qquad-\int_{ {z=0}}   m\pa_t  \N  \cdot \pa_t^{m-1} v  \(g\dt^mh +2\eps \dt^m \( S^\varphi v  \,{\bf n} \cdot \,{\bf n}\)\)\, dy,
    \end{align}
where $\tilde{\mathcal{C}}^m(h)$ is the commutator $\tilde{\mathcal{C}}^\alpha(h)$ defined by \eqref{tildech} for the case $\alpha_0=m$.
 Note that the last term in \eqref{rmq15} can be estimated as follows, thanks to Lemma \ref{lembord},
   \begin{align}\label{rmq151}
  & -\int_{ {z=0}}   m\pa_t  \N  \cdot \pa_t^{m-1} v  \(g\dt^mh +2\eps \dt^m \( S^\varphi v  \,{\bf n} \cdot \,{\bf n}\)\)\, dy
  \\\nonumber&\quad\le \Lambda_\infty \abs{\pa_t^{m-1} v}_{\hal}\(\abs{\dt^mh}_{-\hal}+\eps\abs{\dt^m \( S^\varphi v  \,{\bf n} \cdot \,{\bf
  n}\)}_{-\hal}\)
  \\\nonumber&\quad\le \Lambda_\infty \abs{\pa_t^{m-1} v}_{\hal}\(\abs{\dt^mh}_{-\hal}+\eps\abs{h}_{\X^{m,\hal}}+\eps\abs{v}_{\X^{m,\hal}}\).
    \end{align}
The integration by parts and \eqref{bordhC1} yield
 \begin{align}\label{rmq152}
 &\int_{z= 0} \sigma \pa_t^m  H \( \D_zv\cdot\N \dt^m h- \tilde{\mathcal{C}}^m(h) \)\, dy
 \\\nonumber&\quad\le \sigma\abs{  \dt^m\( \frac{ \nabla_y  h}{\sqrt{1+|\nabla_y h|^2}} \)}_{0}\abs{\nabla_y\(\D_zv\cdot\N \dt^m
 h-\tilde{\mathcal{C}}^m(h)\)}_{0}
 \\&\nonumber\quad \le \Lambda_\infty   \sigma\abs{h}_{\X^{m,1}}\(\abs{\dt^m h}_{1}+\abs{h}_{\X^{m-1,2}}+\abs{v}_{\X^{m-2,1}} \).
    \end{align}
It follows from \eqref{sigsur1} and \eqref{sigsur2} that
  \begin{align}\label{sigmaes0m}
  & \int_{z= 0}  \sigma Z^\alpha H \(\pa_t^{m+1}h+v_y\cdot\nabla_y \pa_t^m h\)\,dy
  \\\nonumber&\quad= \int_{z= 0}\sigma\nabla_y \cdot\( \frac{ \nabla_y \pa_t^{m}h }{\sqrt{1+|\nabla_y h|^2}}- \frac{\nabla_y h\cdot \nabla_y
  \pa_t^{m}h}{\sqrt{1+|\nabla_y h|^2}^3}\nabla_y h
  +\mathcal{C}(\mathcal{B}_\sigma^m) \)\(\pa_t^{m+1}h+v_y\cdot\nabla_y \pa_t^m h\)\,dy,
    \end{align}
where
\begin{align}\label{sigsur2m}
 \mathcal{C}(\mathcal{B}_\sigma^m)=-\[\dt^{m-1},\frac{\nabla_y h}{\sqrt{1+|\nabla_y h|^2}^3}\]\cdot \nabla_y \dt h \nabla_y h+\[\dt^m,
 \frac{1}{\sqrt{1+|\nabla h|^2}}, \nabla_y h\].
\end{align}
Similarly as \eqref{sigmaes3}--\eqref{sigmaes5}, one can deduce that
\begin{align}\label{rmq153}
&\int_{z= 0}\sigma \nabla_y\cdot \( \frac{ \nabla_y \pa_t^{m}h }{\sqrt{1+|\nabla_y h|^2}}- \frac{\nabla_y h\cdot \nabla_y
\pa_t^{m}h}{\sqrt{1+|\nabla_y h|^2}^3}\nabla_y h \)\(\pa_t^{m+1}h+v_y\cdot\nabla_y \pa_t^m h\)\,dy
\\\nonumber&\quad\le-\frac{1}{2}\frac{d}{dt}\int_{z=0}  \sigma\( \frac{  |\nabla_y \pa_t^{m} h|^2 }{\sqrt{1+|\nabla_y h|^2}}-  \frac{ |\nabla_y
h\cdot \nabla_y \pa_t^{m} h|^2  }{\sqrt{1+|\nabla_y h|^2}^3}   \)\,dy+\Lambda_\infty   \sigma\abs{h}_{\X^{m,1}}^2.
   \end{align}
Integrate by parts in both $t$ and $y$ to have
  \begin{align}\label{rmq154}
\nonumber \int_{z= 0}\sigma
 \nabla_y\cdot\( \mathcal{C}(\mathcal{B}_\sigma^m)\) \pa_t^{m+1}h\,dy&\le -\dtt\int_{z= 0}\sigma
  \mathcal{C}(\mathcal{B}_\sigma^m)\cdot \nabla_y  \pa_t^{m}h\,dy+
  \int_{z= 0}\sigma
 \dt \mathcal{C}(\mathcal{B}_\sigma^m) \cdot\nabla_y\pa_t^{m}h\,dy
 \\&\le -\dtt\int_{z= 0}\sigma
  \mathcal{C}(\mathcal{B}_\sigma^m)\cdot \nabla_y  \pa_t^{m}h\,dy+
 \Lambda_\infty   \sigma\abs{h}_{\X^{m,1}}^2.
    \end{align}
One easily has
  \begin{align}\label{rmq155}
  \int_{z= 0}\sigma\nabla_y \cdot\(  \mathcal{C}(\mathcal{B}_\sigma^m) \) v_y\cdot\nabla_y \pa_t^m h \,dy\le \Lambda_\infty
  \sigma\abs{h}_{\X^{m-1,2}}\abs{h}_{\X^{m,1}}.
    \end{align}
Hence, by the estimates \eqref{rmq151}, \eqref{rmq152}, \eqref{rmq153}--\eqref{rmq155}, one may conclude from \eqref{rmq15} that
\begin{align}\label{rmq15m}
  & \int_{z= 0}   \sigma \pa_t^m H    \N \cdot V^m \, dy-\int_{ {z=0}}   m\pa_t  \N  \cdot \pa_t^{m-1} v  \pa_t^{m} q \, dy
  \\\nonumber&\quad\le-\frac{1}{2}\frac{d}{dt}\int_{z=0}  \sigma\( \frac{  |\nabla_y \pa_t^{m} h|^2 }{\sqrt{1+|\nabla_y h|^2}}-  \frac{ |\nabla_y
  h\cdot \nabla_y \pa_t^{m} h|^2  }{\sqrt{1+|\nabla_y h|^2}^3}   \)\,dy-\dtt\int_{z= 0}\sigma
  \mathcal{C}(\mathcal{B}_\sigma^m)\cdot \nabla_y  \pa_t^{m}h\,dy
    \\\nonumber&\qquad+
 \Lambda_\infty  \( \sigma\abs{h}_{\X^{m,1}}^2 +\sigma\abs{h}_{\X^{m,1}} \abs{v}_{\X^{m-2,1}} + \abs{\pa_t^{m-1}
 v}_{\hal}\(\abs{\dt^mh}_{-\hal}+\eps\abs{h}_{\X^{m,\hal}}+\eps\abs{v}_{\X^{m,\hal}}\)\).
\end{align}
This in particular finishes the estimates of the second surface tension term in \eqref{I0'm} and $\mathcal{R}^m_Q$, which can be stated as
follows:
\begin{align}\label{msigQ}
  & \int_{z= 0}   \sigma \pa_t^m H    \N \cdot V^m \, dy+\mathcal{R}^m_Q
  \\\nonumber&\quad\le-\frac{1}{2}\frac{d}{dt}\int_{z=0}  \sigma\( \frac{  |\nabla_y \pa_t^{m} h|^2 }{\sqrt{1+|\nabla_y h|^2}}-  \frac{ |\nabla_y
  h\cdot \nabla_y \pa_t^{m} h|^2  }{\sqrt{1+|\nabla_y h|^2}^3}   \)\,dy-\dtt\mathcal{G}^m
    \\\nonumber&\qquad+
 \Lambda_\infty  \( \sigma\abs{h}_{\X^{m,1}}^2 +\sigma\abs{h}_{\X^{m,1}} \abs{v}_{\X^{m-2,1}} + \abs{\pa_t^{m-1}
 v}_{\hal}\(\abs{\dt^mh}_{-\hal}+\eps\abs{h}_{\X^{m,\hal}}+\eps\abs{v}_{\X^{m,\hal}}\)\)
    \\\nonumber&\qquad+\linf\( \abs{h}_{\X^{m-1,\hal}}+\norm{  v}_{\X^{m}}+\norm{\nabla  v}_{\X^{2}}\)\norm{\pa_t^{m-1} q}_{H^1},
\end{align}
where
\begin{align}
\mathcal{G}^m=&  \int_{\Omega} \( \mathcal{C} ^{m} (d)_2+\mathcal{C} ^{m} (d)_3+\mathcal{C} ^{m} (d)_4+\mathcal{C} ^{m} (d)_5\) \pa_t^{m-1} q
\, dydz
\\\nonumber&-\int_{\Omega}   m \pa_z\( \pa_t^{m-1} q\pa_t  \N\)\cdot \pa_t^{m-1} v  \, dydz+\int_{z= 0}\sigma
  \mathcal{C}(\mathcal{B}_\sigma^m)\cdot \nabla_y  \pa_t^{m}h\,dy.
\end{align}

It remains to estimates the commutators $\mathcal{R}_C^m$ and $\mathcal{R}_S^m$. It follows from \eqref{CT}, \eqref{Cq}, \eqref{Cd},
\eqref{equiv1}, \eqref{qpressurem} and \eqref{qminfty} that
  \begin{align}\label{RCestimatem}
  \mathcal{R}_{C}^m &\le\linf\( \( \norm{ \dt^m \eta}+\norm{   \mathcal{C}^m(\mathcal{T})}+ \norm{ \mathcal{C}^m(q)}\) \cdot V^m +  \norm{
  \mathcal{C}^m (d)} \norm{ \dt^m\eta} \)
  \\\nonumber& \le \Lambda_\infty\(\abs{h}_{\X^{m,-\hal}}+ \norm{v}_{\X^{m-1}}+\norm{\nabla v}_{\X^{m-1}}+ \norm{\nabla
  q}_{\X^{m-1}}\)\(\norm{V^m}+\norm{\dt^m\eta}\)
   \\\nonumber &\le \linf\(\abs{h}_{\X^{m,-{1 \over 2 }}}+\sigma \abs{h}_{\X^{m-1,2}}+\norm{v}_{\X^{m}}  + \norm{\nabla v}_{\X^{m-1}} +
   \eps\abs{  v }_{{\X^{m-1, \frac{3}{2}}}}+\eps\abs{h}_{{\X^{m-1,\frac{3}{2}}}}\)
  \\\nonumber&\quad\times \(\norm{v}_{\X^m}+\abs{h}_{\X^{m,-\hal}}\).
  \end{align}
Similarly as \eqref{RSestimate}, it holds that
\begin{align}\label{RSestimatem}
\mathcal{R}_{S}^m\leq&\linf\(\(\eps^{\hal}\abs{h}_{\X^{m-1,\hal}}+\eps \abs{h}_{\X^{m-1, \frac{3}{2}}}+\eps\norm{\nabla
v}_{\X^{m-1}}\)\(\norm{v}_{\X^m}+\abs{h}_{\X^{m,-\hal}}\)\right.
\\\nonumber&\left.\qquad +\eps\(  \norm{ \nabla v}_{\X^{m-1}}+  \abs{h}_{\X^{m-1,\hal }} \) \norm{ \nabla V^m}+ \(\abs{h}_{\X^{m-1,1} }  +  \abs{
v}_{\X^{m-1,1}}\)\(\abs{h}_{\X^{m} }  +  \abs{  v}_{\X^{m}}\)\).
\end{align}

We can now finish the proof of the proposition. As a consequence of \eqref{m11m}, \eqref{m11m1}, \eqref{Ibestimatem}, \eqref{msigQ},
\eqref{RCestimatem}, \eqref{RSestimatem} and Cauchy's inequality, one may deduce from \eqref{conormen} that
\begin{align}\label{alphaend1m}
&\hal\dtt \mathcal{E}^m+\dtt\mathcal{G}^m + 2 \eps \int_{\Omega} \abs{S^\varphi V^m}^2\, d\V
\\\nonumber&\quad \le\Lambda_\infty\( \abs{h}_{\X^{m }}^2 + \sigma     \abs{h}_{\X^{m,1}}^2
 + \norm{\nabla v}_{\X^{m-1}}^2 + \norm{v}_{\X^{m}}^2 + \eps  \abs{v}_{\X^{m}}^2+ \eps  \abs{h}_{\X^{m , \hal}}^2 \right.
  \\\nonumber&\qquad\qquad \left.+ \eps\abs{  v }_{{\X^{m , \hal}}} \(\norm{v}_{\X^{m }}+\norm{\nabla v}_{\X^{2}}+\abs{h}_{{\X^{m ,\hal}}}\)
+\eps\(  \norm{ \nabla v}_{\X^{m-1}}+  \abs{h}_{\X^{m-1,\hal }} \) \norm{ \nabla V^m}  \),
\end{align}
where
\beq
\mathcal{E}^m:=\int_{\Omega}\abs{V^m }^2 d\V +\int_{z=0} \( g- \D_{z} q\) \abs{\dt^m h }^2+ \sigma\( \frac{  |\nabla_y \pa_t^{m} h|^2
}{\sqrt{1+|\nabla_y h|^2}}-  \frac{ |\nabla_y h\cdot \nabla_y \pa_t^{m} h|^2  }{\sqrt{1+|\nabla_y h|^2}^3}   \)\,dy.
\eeq
Similarly as \eqref{alphaend2}, by the trace estimates
$$\abs{v}_{\X^{m,\hal}}\ls \norm{v}_{\X^{m}}+\norm{\nabla v}_{\X^{m}}\text{ and }
\abs{v}_{\X^{m}}\ls \norm{\nabla v}_{\X^{m}}^\hal\norm{v}_{\X^{m}}^\hal+\norm{v}_{\X^{m}}, $$
using Cauchy's inequality and
\eqref{mheps2}, one can then deduce that
\begin{align}\label{alphaend2m}
&\norm{\dt^m v(t)}^2+\abs{\dt^m h(t)}_{0}^2 +\sigma\abs{\dt^m h(t)}_{1}^2+ \eps  \abs{h(t)}_{\X^{m , \hal}}^2 +\eps\int_0^t  \norm{\nabla \dt^m
v}_{0}^2
\\\nonumber&\quad\le \Lambda_0\( \norm{\dt^m v(0)}^2+\abs{\dt^m h(0)}_{0}^2 +\sigma\abs{\dt^m h(0)}_{1}^2+ \eps  \abs{h(0)}_{\X^{m ,
\hal}}^2\)-\mathcal{G}^m
\\\nonumber&\qquad+ \int_0^t\Lambda_\infty\( \abs{h}_{\X^{m }}^2 + \sigma     \abs{h}_{\X^{m,1}}^2
 + \norm{v}_{\X^{m}}^2+ \norm{\pa_z v}_{\X^{m-1}}^2 + \eps  \norm{\nabla v}_{\X^{m-1,1}}^2+ \eps  \abs{h}_{\X^{m , \hal}}^2  \).
\end{align}
Note that
\begin{align}\label{ges}
 -\mathcal{G}^m\le \linf\(\(\abs{h}_{\X^{m-1,\hal }}+\abs{v}_{\X^{m-1,1 }}+\norm{\pa_z v}_{\X^{m-2
 }}\)\norm{\dt^{m-1}q}_{H^1}+\sigma\abs{h}_{\X^{m-1,2 }}\abs{h}_{\X^{m,1 }}\).
\end{align}
In contrast to the previous case, the difficulty here is that $\norm{\dt^{m-1}q}_{H^1}$ and hence $-\mathcal{G}^m$  are not in $L^\infty([0,T])$
but only in $L^2([0,T])$.
Our basic idea is to integrate in time twice. Indeed, we take the square and then integrate in time to have, by Cauchy's inequality,
 \begin{align}\label{alphaend2m1}
&\int_0^t\(\norm{\dt^m v }^2+\abs{\dt^m h }_{0}^2 +\sigma\abs{\dt^m h }_{1}^2+ \eps  \abs{h }_{\X^{m , \hal}}^2 \)^2+\int_0^t\(\eps\int_0^s
\norm{\nabla \dt^m v}_{0}^2\)^2
\\\nonumber&\quad\le t\Lambda_0\( \norm{\dt^m v(0)}^2+\abs{\dt^m h(0)}_{0}^2 +\sigma\abs{\dt^m h(0)}_{1}^2+ \eps  \abs{h(0)}_{\X^{m ,
\hal}}^2\)^2+\int_0^t\abs{\mathcal{G}^m}^2
\\\nonumber&\qquad+ t\(\int_0^t\Lambda_\infty\( \abs{h}_{\X^{m }}^2 + \sigma     \abs{h}_{\X^{m,1}}^2
  + \norm{v}_{\X^{m}}^2 + \norm{\pa_z v}_{\X^{m-1}}^2+ \eps  \norm{\nabla v}_{\X^{m-1,1}}^2+ \eps  \abs{h}_{\X^{m , \hal}}^2  \)\)^2.
\end{align}
It follows from \eqref{ges} and \eqref{qpressurem} that
\begin{align}\label{hles}
 \int_0^t\abs{\mathcal{G}^m}^2\le& \int_0^t\linf\(\(\abs{h}_{\X^{m-1,\hal }}^2+\abs{v}_{\X^{m-1,1 }}^2+\norm{\pa_z v}_{\X^{m-2
 }}^2\)\(\abs{h}_{\X^{m,-{1 \over 2 }}}+\sigma \abs{h}_{\X^{m-1,2}}\right. \right.
 \\\nonumber&\left.\left.+\norm{v}_{\X^{m}}  + \norm{\pa_z v}_{\X^{m-1}} +  \eps\abs{  v }_{{\X^{m-1,
 \frac{3}{2}}}}+\eps\abs{h}_{{\X^{m-1,\frac{3}{2}}}}\)^2+\sigma^2\abs{h}_{\X^{m-1,2 }}^2\abs{h}_{\X^{m,1 }}^2\)
 \\\nonumber\le& \int_0^t\linf \(\abs{h}_{\X^{m-1,\hal }}^2+\sigma \abs{h}_{\X^{m-1,2 }}^2+\abs{v}_{\X^{m-1,1 }}^2+\norm{\pa_z v}_{\X^{m-2 }}^2\)
 \\\nonumber& \times\(\abs{h}_{\X^{m,-{1 \over 2 }}}^2+\sigma \abs{h}_{\X^{m,1}}^2 +\norm{v}_{\X^{m}}^2  + \norm{\pa_z v}_{\X^{m-1}}^2 +
 \eps^2\abs{  v }_{{\X^{m-1, \frac{3}{2}}}}^2+\eps^2\abs{h}_{{\X^{m-1,\frac{3}{2}}}}^2\)  .
\end{align}
We thus conclude \eqref{conormales1} by plugging the estimate \eqref{hles} into \eqref{alphaend2m1}.
 \end{proof}

\section{Normal derivative estimates}\label{secnormal}

In view of the conormal estimates in Propositions \ref{conormv1} and \ref{conormvm} in Section \ref{secconormal}, the next main step is to
estimate
$\norm{\partial_{z}v}_{\X^{m-1}}$.

Recall the definition of $\linf$ from \eqref{deflambdainfty1} and all the facts of the $L^\infty$ controls elaborated in the beginning of
Section \ref{secconormal}. Note that $\Lambda_{\infty}$ involves only $\sqrt{\eps}  \norm{\partial_{zz}v}_{L^\infty}$. For the case without
surface tension \cite{MasRou} that involves only the spatial derivatives, this is sufficient for deriving the normal derivative estimates since
in
such situation applying the product or commutator estimates to control the commutators resulting from the viscosity term needs only the control
of $\sqrt{\eps}  \norm{\partial_{zz}v}_{L^\infty}$. However, in the current case that involves the time derivatives, following the arguments
of \cite{MasRou} would require the control of $\sqrt{\eps}  \norm{\partial_{zz}v}_{\Y^{k}}$ for some $k\ge 1$. Recall from Proposition 9.8 in
\cite{MasRou} that deriving the bound of $\sqrt{\eps}  \norm{\partial_{zz}v}_{L^\infty}$ requires a crucial use of the heat kernel and the first
order compatibility condition $S_{\n}|_{t=0}=0$ on the boundary. Hence, to control $\sqrt{\eps}  \norm{\partial_{zz}v}_{\Y^{k}}$, it seems to
involve much more delicate use of the various properties of the heat kernel for the time differentiated problems; furthermore, it should require
more compatibility conditions of initial data.

Our key observation here is that since in the vicinity of the boundary the solution behaves as $v(t,x)\sim
v^0(t,x)+\sqrt{\eps}U(t,y,z/\sqrt{\eps})$, it indicates that there may be better control of $\eps \partial_{zz}v $ (and even $\eps
\partial_{zzz}v$!)
in Sobolev conormal spaces. This is indeed the case as shown in the following lemma.
\begin{lem}\label{dzzlem}
It holds that
\begin{align}\label{dzzves11}
\eps \norm{\partial_{zz}v}_{\Y^{{\[\!\frac{m}{2}\!\]}+1}}+\eps \norm{\partial_{zzz}v}_{\Y^{{\[\!\frac{m}{2}\!\]}}}\le \linf
\end{align}
and for $k\le m-1$:
\begin{align}\label{dzzves221}
\eps \norm{\partial_{zz}v}_{\X^{k}}
\le  \linf& \( \eps\(\norm{\pa_z   v}_{\X^{k,1}}+\abs{ h}_{\X^{k,\thal}}\)  +\norm{\pa_zv}_{\X^{k}}+\norm{v}_{
{\X^{k+1}}}+\abs{h}_{\X^{k+1,-\hal}}+\norm{\nabla  q}_{\X^{k}}\)
\end{align}
and
\begin{align}\label{dzzves22'}
\eps \norm{\partial_{zzz}v}_{\X^{k-1}}
\le  \linf& \( \eps\(\norm{\nabla v}_{\X^{k-1,2}}+\abs{ h}_{\X^{k-1,\fhal}}\)+\abs{h}_{\X^{k,\hal}} \right.
\\&\nonumber\   +\norm{\pa_zv}_{\X^{k}} +\norm{v}_{ {\X^{k,1}}} +\norm{\nabla  q}_{\X^{k-1,1}} +\norm{\pa_z\nabla  q}_{\X^{k-1}}\Big) .
\end{align}
\end{lem}
\begin{proof}
It follows from the first equation in \eqref{NSv}, \eqref{transportW} and \eqref{deltaphi} that
\begin{align}\label{eqdzzv}
\eps\partial_{zz}  v ={1 \over E_{33}}  \( -  \eps \sum_{j<3} \partial_{z} \(E_{3,j}\partial_{j} v\)   -\eps\sum_{i<3, \, j} \partial_{i}\( E_{i
j} \partial_{j} v\) +\pa_z\varphi\(\dt v+v_y\cdot\nabla_y v\)+v_z\pa_z v +\pa_z\varphi \nabla^\varphi  q  \).
\end{align}
This implies that, since $v_z=0$ on the boundary,
\begin{align}\label{dzzves}
\eps \norm{\partial_{zz}v}_{\Y^{{\[\!\frac{m}{2}\!\]}+1}}
\le  \Lambda& \(\frac{1}{c_0},\eps\(\norm{\nabla \nabla_y v}_{\Y^{{\[\!\frac{m}{2}\!\]}+1}}+\abs{ h}_{\Y^{{\[\!\frac{m}{2}\!\]}+3}}\)+\abs{h}_{\Y^{{\[\!\frac{m}{2}\!\]}+2}}
\right. \\\nonumber&\qquad\ \left.+\norm{\pa_zv}_{\Y^{{\[\!\frac{m}{2}\!\]}+1}}+\norm{v}_{\Y^{{\[\!\frac{m}{2}\!\]}+2}}+\norm{\nabla  q}_{\Y^{{\[\!\frac{m}{2}\!\]}+1}}\)\le
\linf .
\end{align}
Applying $\pa_z$ to \eqref{eqdzzv}, using the estimate \eqref{dzzves}, one deduces
\begin{align}\label{dzzves2}
\eps \norm{\partial_{zzz}v}_{\Y^{{\[\!\frac{m}{2}\!\]}}}
&\le  \Lambda \(\frac{1}{c_0},\eps\(\abs{ h}_{\Y^{{\[\!\frac{m}{2}\!\]}+3}}+\norm{\nabla^2\nabla_y  v}_{\Y^{{\[\!\frac{m}{2}\!\]}}}\)+\abs{h}_{\Y^{{\[\!\frac{m}{2}\!\]}+2}}
\right. \\\nonumber&\qquad\qquad \left.+\norm{\pa_zv}_{\Y^{{\[\!\frac{m}{2}\!\]}+1}}+\norm{v}_{\Y^{{\[\!\frac{m}{2}\!\]}+2}}+\norm{\pa_z \nabla
q}_{\Y^{{\[\!\frac{m}{2}\!\]}}}+\norm{ \nabla  q}_{\Y^{{\[\!\frac{m}{2}\!\]}}}\)\le \linf ,
\end{align}
where one has used the fact that, since $v_z=0$ on the boundary,
\begin{align}\label{dzzves23}
\nonumber\norm{ \pa_z(v_z\pa_z  v)}_{\Y^{{\[\!\frac{m}{2}\!\]}}}&\le \norm{ \pa_z  v_z \pa_z  v }_{\Y^{{\[\!\frac{m}{2}\!\]}}}+\norm{ v_z\pa_{zz}  v
}_{\Y^{{\[\!\frac{m}{2}\!\]}}}=\norm{ \pa_z  v_z \pa_z  v }_{\Y^{{\[\!\frac{m}{2}\!\]}}}+\norm{ \frac{1}{z(z+b)}v_zZ_3 \pa_z v}_{\Y^{{\[\!\frac{m}{2}\!\]} }}
\\&\ls \norm{ \pa_z  v_z}_{\Y^{{\[\!\frac{m}{2}\!\]}}} \norm{\pa_z  v }_{\Y^{{\[\!\frac{m}{2}\!\]}}}+\norm{ \pa_zv_z}_{\Y^{{\[\!\frac{m}{2}\!\]} }}\norm{ \pa_z
v}_{\Y^{{\[\!\frac{m}{2}\!\]}+1 }}.
\end{align}
Combining \eqref{dzzves} and \eqref{dzzves23} proves the estimates \eqref{dzzves11}.

The estimates \eqref{dzzves221}--\eqref{dzzves22'} follow by a similar argument as that for \eqref{dzzves11}. Indeed, it follows from
\eqref{eqdzzv} and
\eqref{gues} that for $k\ge m-1$,
\begin{align}
\nonumber\eps \norm{\partial_{zz}v}_{\X^{k}}
\le  \linf  \( \eps\(\norm{\nabla \nabla_y v}_{\X^{k}}+\abs{ h}_{\X^{k,\thal}}\)+\abs{h}_{\X^{k,\hal}}  +\norm{\pa_zv}_{\X^{k}}+\norm{v}_{
{\X^{k+1}}}+\abs{h}_{\X^{k+1,-\hal}}+\norm{\nabla  q}_{\X^{k}}\) .
\end{align}
This yields \eqref{dzzves221}. Now applying $Z_i$, $i=1,2,3$ to \eqref{eqdzzv} and then using \eqref{gues} again lead to
\begin{align}\label{dzzvesh12}
\eps \norm{\partial_{zz}v}_{\X^{k-1,1}}
\le  \linf& \( \eps\(\norm{\nabla \nabla_y v}_{\X^{k-1,1}}+\abs{ h}_{\X^{k-1,\fhal}}\)+\abs{h}_{\X^{k-1,\thal}} \right.
\\&\nonumber\left. \ +\norm{\pa_zv}_{\X^{k-1,1}}+\norm{v}_{ {\X^{k,1}}}+\abs{h}_{\X^{k,\hal}}+\norm{\nabla  q}_{\X^{k-1,1}}\) .
\end{align}
Then applying $\pa_z$ to \eqref{eqdzzv} and using the estimate \eqref{dzzves11} and the similar observation as in \eqref{dzzves23} give
\begin{align}\label{dzzvesh13}
\eps \norm{\partial_{zzz}v}_{\X^{k-1}}
\le  \linf& \( \eps\(\norm{\nabla^2 \nabla_y v}_{\X^{k-1}}+\abs{ h}_{\X^{k-1,\fhal}}\)+\abs{h}_{\X^{k-1,\thal}}\right.
\\&\nonumber\ \left.  +\norm{\pa_zv}_{\X^{k}}+\norm{v}_{ {\X^{k-1,1}}}+\abs{h}_{\X^{k,\hal}}+\norm{\pa_z\nabla  q}_{\X^{k-1}}+\norm{ \nabla
q}_{\X^{k-1}}\) .
\end{align}
Combining \eqref{dzzvesh13} and \eqref{dzzvesh12} thus proves the estimate \eqref{dzzves22'}.
\end{proof}

Lemma \ref{dzzlem} then allows one to derive the normal derivative estimates.

\subsection{Estimate of  $\norm{\pa_z v}_{\X^{m-2}}$}

As in Section 8 of \cite{MasRou}, one defines
\beq\label{Sn}
S_{\n}= \Pi \(S^\varphi v   \n-\kappa\chi v\).
\eeq
The main advantages of this quantity are that
\beq\label{Snb}
S_{\n}=0\text{ on }\{z=0,-b\}
\eeq
and that the following estimates hold:
\begin{lem}\label{lemdzS}
For every $k \geq 0$:
\beq\label{dzvk}
\norm{ \partial_{z} v}_{\X^{k}} \le \linf \(\norm{ S_{\n}}_{\X^{k}} + \abs{h}_{\X^{k,\hal}} +  \norm{ v}_{\X^{k,1}}\)
\eeq
and
\beq\label{dzzvk}
\norm{\partial_{zz} v}_{\X^{k}} \le \linf\( \norm{\pa_z S_{\n} }_{\X^{k}} + \abs{h}_{\X^{k,\frac{3}{2}}} + \norm{v}_{\X^{k,2}}\).
\eeq
\end{lem}
\begin{proof}
We start with the estimate \eqref{dzvk}. The normal component of $\partial_{z}v$ is given by, due to the divergence free condition,
\beq\label{eqdzvnk}
\partial_{z} v \cdot  \n=-\frac{\partial_{z} \varphi}{\abs{\N}} \(\partial_{1} v_{1}+ \partial_{2} v_{2}\).
\eeq
Then it follows from \eqref{gues} and Lemma \ref{propeta} that
 \beq\label{vnorm}
\norm{\partial_{z}v \cdot \n }_{\X^{k}} \le \linf \( \norm{ v   }_{\X^{k,1}} + \abs{h}_{\X^{k,\hal}}\).
\eeq
It thus suffices to estimate the tangential components of $\partial_{z} v $. Recall from the derivation of \eqref{eqdzvPi'} that
\beq\label{eqdzvPin}
\Pi \partial_{z}v={\partial_{z} \varphi \over \abs{\N}}\(\partial_{1}\varphi\Pi \partial_{1} v + \partial_{2}\varphi \Pi \partial_{2}
v+2S_{\n}-g^{ij} \partial_{j}v \cdot \n \Pi \partial_{y^i} -\kappa\chi\Pi v\),
\eeq
which follows from the derivation of \eqref{eqdzvPi'}. Hence, by \eqref{gues} and Lemma \ref{propeta}, one has
\beq \label{vnormm}
\norm{\Pi \partial_{z}v }_{\X^{k}} \le \linf \( \norm{ v   }_{\X^{k,1}} + \abs{h}_{\X^{k,\hal}}+\norm{S_{\n}}_{\X^{k}}+\norm{\partial_{z}v \cdot
\n }_{\X^{k}} \).
\eeq
Thus the estimate \eqref{dzvk} follows from combining \eqref{vnorm} and \eqref{vnormm}.

The estimate \eqref{dzzvk} follows from applying $\pa_z$ to \eqref{eqdzvnk} and \eqref{eqdzvPin}, employing a similar argument as
for \eqref{dzvk} and using \eqref{dzvk}.
\end{proof}

In light of Lemma \ref{lemdzS}, we then turn to estimate $S_{\n}$ instead of $\partial_{z} v $. It follows from the first equation in \eqref{NSv}
that
\beq\label{eqS0}\D_{t} \nabla^\varphi v+ \(v \cdot \nabla^\varphi\) \nabla^\varphi v  + \(\nabla^\varphi v\)^2 + \(\nabla^\varphi\)^2 q  -\eps
\Delta^\varphi \nabla^\varphi v = 0.
\eeq
Taking the symmetric part of \eqref{eqS0} yields
\beq\label{eqS}
\D_{t} S^\varphi v+ \(v \cdot \nabla^\varphi\) S^\varphi v  + \hal\((\nabla^\varphi v)^2+ ((\nabla^\varphi v)^t)^2\) +  \(\nabla^\varphi\)^2 q
-\eps \Delta^\varphi (S^\varphi v)= 0.
\eeq
Consequently,
\beq\label{eqPiS}
\D_{t}  S_{\n}+ \(v \cdot \nabla^\varphi\) S_{\n}  -\eps \Delta^\varphi (S_{\n})= F_{S}
\eeq
where
\beq\label{FSdef}
F_{S}= F_{S}^1 + F_{S}^2+ F_{S}^3+ F_{S}^4
\eeq
with
\begin{align}
\label{FS1}& F_{S}^1 =   - \hal  \Pi\( (\nabla^\varphi v)^2+ ((\nabla^\varphi v)^t)^2\) \n + \( \partial_{t} \Pi + v \cdot \nabla^\varphi \Pi
\)S^\varphi v  \n +    \Pi S^\varphi v \( \partial_{t} \n + v \cdot \nabla^\varphi \n\)
\\&\qquad+\kappa \(\D_{t}+  v \cdot \nabla^\varphi\) (\chi \Pi v),\nonumber   \\
\label{FS4} &F_{S}^2 =- \Pi\( \(\nabla^\varphi\)^2 q \)\n,  \\
\label{FS2}& F_{S}^3 =    - \eps \(\Delta^\varphi \Pi  \)S^\varphi v  \n  - \eps  \Pi S^\varphi v \Delta^\varphi \n, \\
\label{FS3}& F_{S}^4 =   - 2 \eps  \D_{i} \Pi \D_{i} \( S^\varphi v  \n \)  -  2 \eps  \Pi \( \D_{i}\( S^\varphi v \)  \D_{i} \n \)-\eps
\kappa\Delta^\varphi (\chi\Pi v).
\end{align}
Note that the pressure estimates in \eqref{qpressurem} indicate that at most, one can estimate only $\norm{S_{\n}}_{\X^{m-2}}$ and hence
$\norm{\pa_z v}_{\X^{m-2}}$ at this stage. However, we shall prove a control of $\norm{\pa_z v}_{\X^{m-1}}$ based on the vorticity equation in
the next
subsection.

\begin{prop}\label{propdzvm-2}
Any smooth solution of \eqref{NSv} satisfies the estimate
 \begin{align} \label{estSnm-2}
& \norm{S_{\n}}_{\X^{m-2}}^2    + \eps \int_{0}^t \norm{ \nabla S_{\n}}_{\X^{m-2}}^2  \\
&\nonumber\le \Lambda_0\norm{S_{\n}(0)}_{\X^{m-2}}^2+\int_0^t \linf \eps^2\(\norm{\nabla v}_{\X^{m-2,2}}^2+\abs{h}_{\X^{m-2,\fhal}}^2\)
\\\nonumber&\quad  +\int_0^t \linf \abs{h}_{\X^{m-1,\hal}}^2+\sigma^2 \abs{h}_{\X^{m-2,3}}^2+\norm{v}_{\X^{m-1,1}}^2  + \norm{\pa_z
v}_{\X^{m-2,1}}^2 +\norm{ S_{\n} }_{\X^{m-2}}^2  .
\end{align}
\end{prop}
\begin{proof}
We start with the estimates of $F_S$. It follows from \eqref{gues} and Lemma \ref{propeta} that
\begin{align}
\label{FS1es}&
\norm{ F_{S}^1}_{\X^{m-2}} \le\linf\( \norm{\nabla  v }_{\X^{m-2}}  + \abs{h}_{\X^{m-1,\hal}} + \norm{v}_{\X^{m-2}} \),\\
\label{FS2es}&
\norm{ F_{S}^2 }_{\X^{m-2}} \le \linf \norm{\nabla q}_{\X^{m-2,1}},\\
\label{FS3es}&
\norm{ F_{S}^3 }_{\X^{m-2}} \le\linf \eps   \(\norm{\nabla v}_{\X^{m-2}}  + \abs{h}_{\X^{m-2, \frac{5}{2} }}\),\\
\label{FS4es}&
\norm{ F_{S}^4 }_{\X^{m-2}}
\le \linf \eps\(\abs{h}_{\Y^{{\[\!\frac{m}{2}\!\]-1}+2}}\norm{\nabla^2 v}_{\X^{m-2}}+\abs{h}_{\X^{m-2,\frac{3}{2}}}\norm{\nabla^2 v}_{\Y^{{\[\!\frac{m}{2}\!\]-1}}}\)
\\&\nonumber \  \le \linf \( \eps \norm{\nabla v}_{\X^{m-2,1}} +\norm{\pa_zv}_{\X^{m-2}}+\norm{v}_{\X^{m-1}}+\abs{h}_{\X^{m-1,-\hal}}+
\norm{\nabla q}_{\X^{m-2}}+\abs{h}_{\X^{m-2,\frac{3}{2}}}\).
\end{align}
Here in the second inequality of \eqref{FS4es}, one has used \eqref{dzzves11} and \eqref{dzzves221} with $k=m-2$. Hence,
\beq\label{FSes}
\norm{ F_{S}}_{\X^{m-2}} \le\linf\(\eps\(\norm{\nabla
v}_{\X^{m-2,1}}+\abs{h}_{\X^{m-2,\fhal}}\)+\norm{\pa_zv}_{\X^{m-2}}+\norm{v}_{\X^{m-1}}+\abs{h}_{\X^{m-1,\hal}} + \norm{\nabla q}_{\X^{m-2,1}}
\).
\eeq

It follows from applying $\frac{1}{\pa_z\varphi}Z^\alpha(\pa_z\varphi\cdot) $ for $| \alpha | \leq m-2 $ to \eqref{eqPiS}, \eqref{transportW} and
\eqref{deltaphi} that
\beq\label{SNalpha}
\D_{t}  Z^\alpha S_{\n}+ \(v \cdot \nabla^\varphi\)Z^\alpha S_{\n}  -\eps \Delta^\varphi Z^\alpha  S_{\n} = Z^\alpha F_{S} + \mathcal{C}_{S},
\eeq
where the commutator is given by
\beq\label{CS}
\mathcal{C}_{S}= \mathcal  C_{S}^1 + \mathcal C_{S}^2
\eeq
with
\begin{align}
\mathcal{C}_{S}^1&=\frac{1}{\pa_z\varphi}\[Z^\alpha,\pa_z\varphi\]\(\pa_t+v_y\cdot\nabla_y\)S_\n+ \[Z^\alpha, v_{y}\]\cdot  \nabla_{y} S_{\n} +
\[Z^\alpha, v_{z}\] \partial_{z} S_{\n}
\\\nonumber&:= \mathcal{C}_{S1}^1 + \mathcal{C}_{S2}^1 + \mathcal{C}_{S3}^1
\end{align}
and
\begin{align}
\mathcal{C}_{S}^2& =- \eps {1 \over \partial_{z} \varphi} [Z^\alpha, \nabla] \cdot \big( E \nabla S_{\n}\big)-\eps {1 \over \partial_{z}\varphi}
\nabla\cdot  \big( \[Z^\alpha, E \]\nabla S_{\n}\big)-\eps {1 \over \partial_{z}\varphi} \nabla\cdot  \big( E [Z^\alpha,  \nabla \big]
S_{\n}\big)
\\\nonumber&:= \mathcal{C}_{S1}^2 + \mathcal{C}_{S2}^2 + \mathcal{C}_{S3}^2.
\end{align}
Since $Z^\alpha S_{\n}=0$ on the boundary, one can obtain
\beq\label{m-21}
\hal\dtt\int_{\Omega}  \abs{Z^\alpha S_{\n}}^2 \, d\V + \eps \int_{\Omega} \abs{\nabla^\varphi Z^\alpha S_{\n}}^2 \, d\V
=\int_{\Omega}  \(Z^\alpha F_{S} + \mathcal{C}_{S} \) \cdot Z^\alpha S_{\n}\, d\V.
\eeq
The right hand side of \eqref{m-21} can be estimated as follows. \eqref{FSes} implies immediately that
\begin{align}\label{FSmes}
\int_{\Omega}  Z^\alpha F_{S} \cdot Z^\alpha S_{\n}\, d\V
\le \linf&\(\eps\(\norm{\pa_z v}_{\X^{m-2,1}}+\abs{h}_{\X^{m-2,\fhal}}\)+\norm{\pa_zv}_{\X^{m-2}}\right.
\\\nonumber&\ \left.+\norm{v}_{\X^{m-1}}+\abs{h}_{\X^{m-1,\hal}} + \norm{\nabla q}_{\X^{m-2,1}} \)\norm{ S_{\n} }_{\X^{m-2}}.
\end{align}
Next, we estimate the part involving $C_S^1$. \eqref{com} yields
\beq\label{CSy}
\norm{\mathcal  C_{S1}^1}+\norm{\mathcal  C_{S1}^2}\le\linf\(  \norm{v}_{\X^{m-2}}+ \norm{S_{\n}}_{\X^{m-2 }} + \abs{h}_{\X^{m-2,\hal }}  \).
\eeq
Additional care is needed to estimate $\norm{\mathcal  C_{S3}^1}$, since $\pa_{z}S_{\n}$ can not be
controlled. By expanding the commutator and using \eqref{idcom}, one sees clearly that it suffices to estimate terms of the form
$$\norm{ Z^\beta v_{z} \partial_{z} Z^\gamma S_{\n} }$$
with $|\beta |+ | \gamma |\leq m-2$, $ | \gamma |  \leq m-3$. Since
$$ Z^\beta v_{z} \partial_{z} Z^\gamma S_{\n}= {1 \over z(z+b) } Z^\beta v_{z} Z_{3} Z^\gamma S_{\n},$$
which can be further rewritten as a sum of terms of the form
\beq\label{CSz0}
c_{\tilde \beta }Z^{\tilde \beta }\( {1 \over z(z+b) } v_{z} \)   \, Z_{3} Z^\gamma S_{\n},
\eeq
where $c_{\tilde \beta}$ are harmless bounded functions and $|\tilde \beta | \leq | \beta |$.  Indeed, this comes from the fact that $ Z_{3}\( {1
\over z(z+b) }\)= -{2z+b \over z(z+b) } $. If $\tilde \beta = 0$,  since \eqref{Wdef} implies that $v_{z}=0$ on the boundary, then
$$
\norm{ c_{\tilde \beta }Z^{\tilde \beta }\(  {1 \over z(z+b) } v_{z} \)   \, Z_{3} Z^\gamma S_{\n}}\ls  \norm{\pa_z v_z}_{L^\infty}   \norm{
S_{\n}}_{\X^{m-2}}.
$$
If $\tilde \beta \neq 0$, one can use \eqref{gues} to obtain that
\begin{align*}
&\norm{ c_{\tilde \beta }Z^{\tilde \beta }\(  {1 \over z(z+b) } v_{z} \)   \, Z_{3} Z^\gamma S_{\n}}
\\&\quad\ls \norm{Z \(  {1 \over z(z+b) } v_{z} \)}_{\Y^{\[\!\frac{m-3}{2}\!\]}} \norm{Z_{3} S_{\n} }_{\X^{m-3}}
+ \norm{Z \(  {1 \over z(z+b) } v_{z} \)}_{\X^{m-3}} \norm{Z_{3} S_{\n} }_{\Y^{\[\!\frac{m-3}{2}\!\]}}.
\end{align*}
Observe that by again that $v_{z}=0$ on the boundary,
$$ \norm{Z \(  {1 \over z(z+b) } v_{z} \)}_{\Y^{\[\!\frac{m-3}{2}\!\]}}  \ls \norm{ \pa_z v_{z} }_{\Y^{\[\!\frac{m}{2}\!\]}} .$$
On the other hand, since
$$\norm{Z \(  {1 \over z(z+b) } v_{z} \)}_{\X^{m-3}} \ls  \norm{   {1 \over z(z+b) } Z v_{z}  }_{\X^{m-3}}  +  \norm{   {1 \over z(z+b) } v_{z}
}_{\X^{m-3}},
$$
it suffices to estimate
$$ \norm{   {1 \over z(z+b) } Z^\beta Z v_{z}  },\quad \norm{   {1 \over z(z+b) } Z^\beta v_{z}  }, \quad |\beta| \leq m-3.$$
Indeed, since $v_{z}=0$ on the boundary, it follows from the Hardy inequality that
$$ \norm{   {1 \over z(z+b) } Z^\beta Z v_{z}  }\ls \norm{\pa_z Z^\beta Z v_{z}  }\text{ and } \norm{   {1 \over z(z+b) } Z^\beta v_{z}  }\ls
\norm{\pa_z Z^\beta v_{z}  }.$$
We have thus proven that
$$ \norm{\mathcal{C}_{S3}^1} \le  \norm{ \pa_z v_{z} }_{\Y^{\[\!\frac{m}{2}\!\]}}  \norm{ S_{\n}}_{\X^{m-2}} + \norm{ S_{\n}
}_{\Y^{{\[\!\frac{m}{2}\!\]}}}\norm{ \pa_z v_{z} }_{\X^{m-2}}.
$$
By Lemma \ref{propeta}, it holds that
\begin{align*}
\norm{ \pa_z v_{z} }_{\Y^{\[\!\frac{m}{2}\!\]}}\le \linf\text{ and }\norm{ \pa_z v_{z}}_{\X^{m-2}} \le  \linf\(\norm{v}_{\X^{m-2}} +\norm{\pa_z
v}_{\X^{m-2}} +\abs{h }_{\X^{m-1 ,\hal}} \).
\end{align*}
Hence,
\beq\label{CSz}
\norm{\mathcal{C}_{S3}^1}  \le \linf \(\norm{v}_{\X^{m-2}} +\norm{\pa_z v}_{\X^{m-2}} +\abs{h }_{\X^{m-1 ,\hal}}+\norm{ S_{\n}}_{\X^{m-2}}\).
\eeq
This together with \eqref{CSy} yields
\beq\label{CS1}
\int_{\Omega}   \mathcal{C}_{S}^1  \cdot Z^\alpha S_{\n}\, d\V \le \linf \(\norm{v}_{\X^{m-2}} +\norm{\pa_z v}_{\X^{m-2}} +\abs{h }_{\X^{m-1
,\hal}}+\norm{ S_{\n}}_{\X^{m-2}}\)\norm{ S_{\n}}_{\X^{m-2}}.
\eeq

Next, we shall estimate the term involving $\mathcal{C}_{S}^2$. As mentioned in the beginning of this section, we need to employ a different
argument from Proposition 8.3 in \cite{MasRou}. Due to \eqref{idcom} to handle the term involving $\mathcal{C}_{S1}^2$, it suffices to
estimate
$$ \eps \int_{\Omega } Z^\beta\partial_{z}\big( E\nabla S_{\n}\big) \cdot Z^\alpha S_{\n} \,dy dz$$
with $|\beta | \leq m-3$, which can be reduced to the estimate of
$$ \eps \int_{\Omega } \(Z^{\beta'}\partial_{z}  E Z^{\beta''}\nabla S_{\n}+Z^{\beta'} E Z^{\beta''} \partial_{z} \nabla S_{\n}  \) \cdot
Z^\alpha S_{\n} \,dy dz$$
with $\beta'+\beta''=\beta$.
It follows from \eqref{gues}, \eqref{dzzves11} and \eqref{dzzves221} with $k=m-3$ that
\begin{align}\label{keyiss11}
&\eps\norm{Z^{\beta'}\partial_{z}  E Z^{\beta''}\nabla S_{\n}}
 \ls  \norm{\partial_{z}  E}_{\Y^{\[\!\frac{m-3}{2}\!\]}}\eps\norm{\nabla S_{\n}}_{\X^{m-3}}
+\norm{\partial_{z}  E}_{\X^{m-3}}\eps\norm{\nabla S_{\n}}_{\Y^{\[\!\frac{m-3}{2}\!\]}}
\\&\nonumber \quad\le \linf \(\eps \norm{\pa_z   v}_{\X^{m-3,1}}   +\norm{\pa_zv}_{\X^{m-3}}+\norm{v}_{
{\X^{m-2}}}+\abs{h}_{\X^{m-2,-\hal}}+\norm{\nabla  q}_{\X^{m-3}}+ \abs{h}_{\X^{m-3,\frac{3}{2}}}\),
\end{align}
and \eqref{dzzves22'} with $k=m-2$ implies that
\begin{align}\label{keyiss12}
&\eps\norm{Z^{\beta'}  E Z^{\beta''}\partial_{z}\nabla S_{\n}}
\ls  \norm{   E}_{\Y^{\[\!\frac{m-3}{2}\!\]}}\eps\norm{\partial_{z} \nabla S_{\n}}_{\X^{m-3}}
+\norm{   E}_{\X^{m-3}}\eps\norm{\partial_{z} \nabla S_{\n}}_{\Y^{\[\!\frac{m-3}{2}\!\]}}
\\&\nonumber\quad \le \linf \(\eps\(\norm{\nabla v}_{\X^{m-3,2}}+\abs{ h}_{\X^{m-3,\fhal}}\)+\abs{h}_{\X^{m-2,\hal}} \right.
\\&\nonumber\qquad\qquad\   +\norm{\pa_zv}_{\X^{m-2}} +\norm{v}_{ {\X^{m-2,1}}} +\norm{\nabla  q}_{\X^{m-3,1}} +\norm{\pa_z\nabla
q}_{\X^{m-3}}\Big).
\end{align}
Thus,
\begin{align}\label{CS22}
 \int_{\Omega} \mathcal{C}_{S1}^2 \cdot Z^\alpha S_{\n}\, d\V
\le& \linf \(\eps\(\norm{\nabla v}_{\X^{m-3,2}}+\abs{ h}_{\X^{m-3,\fhal}}\)+\abs{h}_{\X^{m-2,\hal}} \right.
\\&\nonumber\qquad     +\norm{\pa_zv}_{\X^{m-2}} +\norm{v}_{ {\X^{m-2,1}}} +\norm{\nabla  q}_{\X^{m-3,1}} +\norm{\pa_z\nabla  q}_{\X^{m-3}}\Big).
\end{align}
By expanding the second commutator $\mathcal{C}_{S2}^2$, one sees that it suffices to estimate terms of the form
$$\eps \int_{\Omega }  \nabla \cdot \(  Z^{\beta} E Z^{ \gamma} \nabla S_{\n}\) \, \cdot Z^\alpha S_{\n} \, dydz$$
with $\beta +   \gamma = \alpha$, $\beta \neq 0$. If $\gamma=0$ and hence $\beta=\alpha$, since $Z^\alpha S_{\n}=0$ on the boundary, one can then
integrate by parts to get
\begin{align}
&\abs{\eps \int_{\Omega }  \nabla \cdot \(  Z^{\alpha} E   \nabla S_{\n}\) \, \cdot Z^\alpha S_{\n} \, dydz}
=
\abs{\eps \int_{\Omega }       Z^{\alpha} E   \nabla S_{\n}  \cdot \nabla Z^\alpha S_{\n} \, dydz}
\\&\nonumber\quad\le \eps \norm{Z^{\alpha} E}\norm{\nabla S_{\n}}_{L^\infty}\norm{\nabla Z^\alpha S_{\n}}\le \linf
\abs{h}_{\X^{m-2,\hal}}\eps^\hal \norm{\nabla Z^\alpha S_{\n}},
\end{align}
where in the last inequality one has used the fact that $\Lambda_{\infty}$ involves $\sqrt{\eps}  \norm{\partial_{zz}v}_{L^\infty}$.
If $\gamma\neq 0$ and hence $1\le |\beta|\le m-3$, then one can expand $\nabla\cdot$, by \eqref{gues} and Lemma \ref{dzzlem}, to estimate
$$\eps \int_{\Omega }   \(       Z^{\beta} \nabla E Z^{ \gamma} \nabla S_{\n}+  Z^{\beta} E   Z^{ \gamma}  \nabla^2 S_{\n}  \)\, \cdot Z^\alpha
S_{\n} \, dydz.$$
It follows from \eqref{gues}, \eqref{dzzves11} and \eqref{dzzves221} with $k=m-3$ that
\begin{align}\label{keyiss112}
&\eps\norm{ Z^{\beta}\nabla E Z^{ \gamma} \nabla S_{\n}}
 \ls  \norm{Z\nabla  E}_{\Y^{\[\!\frac{m}{2}\!\]-2}}\eps\norm{Z\nabla S_{\n}}_{\X^{m-4}}
+\norm{Z\nabla  E}_{\X^{m-4}}\eps\norm{Z\nabla S_{\n}}_{\Y^{\[\!\frac{m}{2}\!\]-2}}
\\&\nonumber\quad \le \linf \(\eps \norm{\pa_z   v}_{\X^{m-3,1}}   +\norm{\pa_zv}_{\X^{m-3}}+\norm{v}_{
{\X^{m-2}}}+\abs{h}_{\X^{m-2,-\hal}}+\norm{\nabla  q}_{\X^{m-3}}+ \abs{h}_{\X^{m-3,\frac{3}{2}}}\),
\end{align}
and \eqref{dzzves22'} with $k=m-2$ that
\begin{align}
&\eps\norm{Z^{\beta}  E   Z^{\gamma} \nabla^2 S_{\n}} \ls  \norm{Z   E}_{\Y^{\[\!\frac{m}{2}\!\]-2}}\eps\norm{Z\nabla^2 S_{\n}}_{\X^{m-4}}
+\norm{ Z  E}_{\X^{m-4}}\eps\norm{Z\nabla^2 S_{\n}}_{\Y^{\[\!\frac{m}{2}\!\]-2}}\\&\nonumber\quad \le \linf \(\eps\(\norm{\nabla v}_{\X^{m-3,2}}+\abs{
h}_{\X^{m-3,\fhal}}\)+\abs{h}_{\X^{m-2,\hal}} \right.
\\&\nonumber\qquad\qquad\   +\norm{\pa_zv}_{\X^{m-2}} +\norm{v}_{ {\X^{m-2,1}}} +\norm{\nabla  q}_{\X^{m-3,1}} +\norm{\pa_z\nabla
q}_{\X^{m-3}}\Big).
\end{align}
Hence,
\begin{align}\label{CS222}
  \nonumber\int_{\Omega} \mathcal{C}_{S2}^2 \cdot Z^\alpha S_{\n}\, d\V
&\le \linf  \(\eps\(\norm{\nabla v}_{\X^{m-3,2}}+\abs{ h}_{\X^{m-3,\fhal}}\)+\abs{h}_{\X^{m-2,\hal}} +\norm{\pa_zv}_{\X^{m-2}} +\norm{v}_{
{\X^{m-2,1}}}\right.
\\&\qquad  +\norm{\nabla  q}_{\X^{m-3,1}} +\norm{\pa_z\nabla  q}_{\X^{m-3}}\Big)\norm{ S_{\n}}_{\X^{m-2}}+\linf  \abs{h}_{\X^{m-2,\hal}}\eps^\hal
\norm{\nabla Z^\alpha S_{\n}}  .
\end{align}
To handle the term involving $\mathcal{C}_{S3}^2$, due to \eqref{idcom}, one needs to estimate
$$ \eps \int_{\Omega }\nabla\cdot\( E Z^\beta\partial_{z}  S_{\n} \)\cdot Z^\alpha S_{\n} \,dy dz$$
with $|\beta | \leq m-3$, which can be bounded easily by using \eqref{dzzves221} with $k=m-3$ and \eqref{dzzves22'}
with $k=m-2$ so that
\begin{align}\label{CS2223}
  \int_{\Omega} \mathcal{C}_{S3}^2 \cdot Z^\alpha S_{\n}\, d\V
  &\le \linf  \(\eps\(\norm{\nabla v}_{\X^{m-3,2}}+\abs{ h}_{\X^{m-3,\fhal}}\)+\abs{h}_{\X^{m-2,\hal}} +\norm{\pa_zv}_{\X^{m-2}}\right.
\\&\nonumber\qquad\quad  +\norm{v}_{ {\X^{m-2,1}}}+\norm{\nabla  q}_{\X^{m-3,1}} +\norm{\pa_z\nabla  q}_{\X^{m-3}}\Big)\norm{ S_{\n}}_{\X^{m-2}}
.
\end{align}
In view of \eqref{CS22}, \eqref{CS222}, \eqref{CS2223}, one has actually proven that
\begin{align} \label{CS2}
 \nonumber \int_{\Omega} \mathcal{C}_{S}^2 \cdot Z^\alpha S_{\n}\, d\V
&\le \linf  \(\eps\(\norm{\nabla v}_{\X^{m-3,2}}+\abs{ h}_{\X^{m-3,\fhal}}\)+\abs{h}_{\X^{m-2,\hal}} +\norm{\pa_zv}_{\X^{m-2}} +\norm{v}_{
{\X^{m-2,1}}}\right.
\\&\qquad\quad  +\norm{\nabla  q}_{\X^{m-3,1}} +\norm{\pa_z\nabla  q}_{\X^{m-3}}\Big)\norm{ S_{\n}}_{\X^{m-2}}+\linf
\abs{h}_{\X^{m-2,\hal}}\eps^\hal \norm{\nabla Z^\alpha S_{\n}}  .
\end{align}

Consequently, by \eqref{FSmes}, \eqref{CS1} and \eqref{CS2}, one deduces from \eqref{m-21} that
 \begin{align}
& \hal\dtt \int_{\Omega} \abs{Z^\alpha S_{\n}}^2 \, d\V    +\eps \int_{\Omega}\abs{\nabla^\varphi Z^\alpha S_{\n}}^2 \, d\V \\
&\nonumber\le \linf\(\eps\(\norm{\nabla v}_{\X^{m-2,1}}+\abs{h}_{\X^{m-2,\fhal}}\)+\norm{ S_{\n}
}_{\X^{m-2}}+\norm{\pa_zv}_{\X^{m-2}}+\norm{v}_{\X^{m-1}} \right.
\\\nonumber&\quad \left.+\abs{h}_{\X^{m-1,\hal}} + \norm{\nabla q}_{\X^{m-2,1}} +\norm{\pa_z\nabla  q}_{\X^{m-3}}\)\norm{ S_{\n}
}_{\X^{m-2}}+\linf  \abs{h}_{\X^{m-2,\hal}}\eps^\hal \norm{\nabla Z^\alpha S_{\n}} .
\end{align}
To conclude, one can use Lemma \ref{mingrad} to replace $ \norm{\nabla^\varphi S_{\n}}_{\X^{m-2}}$ by $\norm{\nabla  S_{\n}}_{\X^{m-2}}$ in the
left hand side and then use Cauchy's inequality to absorb the last term. On the other hand, one can follow the derivation of \eqref{qpressurem}
to get that
  \begin{align}\label{qpressurem'}
  \nonumber&    \norm{   q }_{\X^{m-2,1}}+\norm{ \nabla q }_{\X^{m-2,1}}+\norm{ \pa_{zz} q}_{\X^{m-2}}
      \\\nonumber&  \quad\leq \linf\(\abs{h}_{\X^{m-1,{1 \over 2 }}}+\sigma \abs{h}_{\X^{m-2,3}}+\norm{v}_{\X^{m-1,1}}  + \norm{\nabla
      v}_{\X^{m-2,1}} +  \eps\norm{ \nabla v }_{{\X^{m-2, 2}}}+\eps\abs{h}_{{\X^{m-2,\frac{5}{2}}}}\),
  \end{align}
which reduces the order of time derivatives to the spatial derivatives.
Finally, we integrate the resulting inequality in time to obtain \eqref{estSnm-2}.
\end{proof}

\subsection{Estimate of  $\norm{\pa_z v}_{\X^{m-1}}$}\label{sectionnorm2}

Note that Proposition \ref{propdzvm-2} only provides the control of $\norm{\pa_z v}_{\X^{m-2}}$, and this is due to the appearance of
$(\nabla^\varphi)^2 q$ in the source term in the equation of $S_{\n}$. To get around this, a natural way is to use the vorticity instead of
$S_{\n}$.

Set $\omega = \nabla^\varphi \times v$. Then
\beq\label{omega1}
\omega \times \n= \hal \( \nabla^\varphi v\, \n - (\nabla^\varphi v)^t\n\)= S^\varphi v\, \n -   (\nabla^\varphi v)^t\n,
\eeq
and hence by \eqref{Su},
$$ \omega \times \n =   \hal  \partial_{\n} u  - g^{ij}\big( \partial_{j}v \cdot \n \big) \partial_{y^i} .$$
Consequently, as in Lemma \ref{lemdzS} one can get that
\beq\label{dzvm-1O}
\norm{\partial_{z} v}_{\X^{m-1}} \le \linf\(\norm{v}_{\X^{m-1,1}}+\abs{h}_{\X^{m-1,\hal} } + \norm{\omega}_{\X^{m-1}}\).
\eeq
{It then suffices to estimate $\norm{\omega}_{\X^{m-1}}$. Before proceeding further, one has the following
observation.
\begin{lem}\label{iii}
For any smooth cut-off function $\chi$ such that $\chi= 0$ in a vicinity of $\partial\Omega$, we have
\beq\label{inftyint}
\norm{\chi\omega}_{ H^k}\ls\norm{f}_{k+1}.
\eeq
\end{lem}
\begin{proof}
It follows directly by the fact that away from the boundary
$\partial\Omega$ the conormal Sobolev norm $\norm{\cdot}_{k}$ is equivalent to the $H^k$ norm.
\end{proof}

Hence, one needs only to estimate $   \omega  $ near the boundary
$\partial\Omega$.  For sake of brevity, we consider only the
estimates near $\{z=0\}$, and the estimates near $\{z=-b\}$ may follow in the same way and a bit simpler. One notes that the first equation in \eqref{NSv} implies
\beq\label{eqomega1}
\D_{t} \omega + v \cdot \nabla^\varphi \omega - \omega \cdot \nabla^\varphi v = \eps \Delta^\varphi \omega.
\eeq
We choose the cut-off function $\chi(z) \in [0, 1]$ which is smooth compactly supported near $\{z=0\}$ and takes the value $1$ in a vicinity of $\{z=0\}$. Then we have
\beq\label{eqomega}
\D_{t} (\chi\omega) + v \cdot \nabla^\varphi (\chi\omega)  - (\chi\omega)  \cdot \nabla^\varphi v = \eps \Delta^\varphi (\chi\omega) +\big(V_{z} \partial_{z} \chi\big) \omega -  2\eps \nabla^\varphi \chi \cdot \nabla^\varphi \omega - \eps \Delta^\varphi
\chi \, \omega.
\eeq
As in the previous subsection, applying $ \frac{1}{\pa_z\varphi}Z^\alpha(\pa_z\varphi\cdot)$ for $|\alpha|\le m-1$ to
\eqref{eqomega} yields
\beq\label{eqomegaalpha}
\D_{t} Z^\alpha  (\chi\omega) + (v \cdot \nabla^\varphi) \, Z^\alpha  (\chi\omega) - \eps \Delta Z^\alpha (\chi\omega)= F.\eeq
Here $F$ is given by
   \beq
   \label{Fomega}
    F= Z^\alpha\big(\chi\omega \cdot \nabla^\varphi v+\big(V_{z} \partial_{z} \chi\big) \omega -  2\eps \nabla^\varphi \chi \cdot \nabla^\varphi \omega - \eps \Delta^\varphi
\chi \, \omega)  + \mathcal{C}_{S},
   \eeq
   where} $\mathcal{C}_{S}$  is given, as in \eqref{CS}, by
    \beq
       \label{CSomega}
       \mathcal{C}_{S}= \mathcal  C_{S}^1 + \mathcal C_{S}^2\eeq
with
\begin{align}
\mathcal{C}_{S}^1&=\frac{1}{\pa_z\varphi}\[Z^\alpha,\pa_z\varphi\]\(\pa_t+v_y\cdot\nabla_y\) {(\chi\omega)} + \[Z^\alpha, v_{y}\]\cdot  \nabla_{y} {(\chi\omega)}  +
\[Z^\alpha, v_{z}\] \partial_{z}{(\chi\omega)}
\\\nonumber&:= \mathcal{C}_{S1}^1 + \mathcal{C}_{S2}^1 + \mathcal{C}_{S3}^1
\end{align}
and
\begin{align}
\mathcal{C}_{S}^2& =- \eps {1 \over \partial_{z} \varphi} [Z^\alpha, \nabla] \cdot \big( E \nabla {(\chi\omega)} \big)-\eps {1 \over \partial_{z}\varphi}
\nabla\cdot  \big( \[Z^\alpha, E \]\nabla {(\chi\omega)} \big)-\eps {1 \over \partial_{z}\varphi} \nabla\cdot  \big( E [Z^\alpha,  \nabla \big]
{(\chi\omega)} \big)\nonumber
\\&:= \mathcal{C}_{S1}^2 + \mathcal{C}_{S2}^2 + \mathcal{C}_{S3}^2.
\end{align}

{Since $\chi(z)$ is compactly supported near $\{z=0\}$, we may regard the equation \eqref{eqomegaalpha}  as to be defined in  $\{z<0\}$ by extending the functions smoothly outside $\Omega$. The main difficulty lies in that the vorticity does not vanish on the boundary $\{z=0\}$.} Thus, set
\beq\label{splitomega}
Z^\alpha {(\chi\omega)}= \omega^\alpha_{nh}+ \omega^{\alpha}_{h},
\eeq
where $\omega^\alpha_{nh} $ sloves
\beq\label{eqomeganh}
\begin{cases}
\D_{t}  \omega^\alpha_{nh} + (v \cdot \nabla^\varphi)   \omega^\alpha_{nh} - \eps \Delta^\varphi  \omega^\alpha_{nh}= F &\text{in } {\{z<0\}}
\\ \omega^\alpha_{nh} = 0 &\text{on }{\{z=0\}}
\\  \omega^\alpha_{nh}|_{t=0}= { Z^\alpha(\chi \omega)}(0)&
\end{cases}
\eeq
and $ \omega^{\alpha}_{h}$ solves
\beq\label{eqomegah}
\begin{cases}
\D_{t}  \omega^\alpha_{h} + (v \cdot \nabla^\varphi)   \omega^\alpha_{h} - \eps \Delta^\varphi  \omega^\alpha_{h}= 0 &\text{in }{\{z<0\}}
\\ \omega^\alpha_{h} = Z^\alpha \omega &\text{on }{\{z=0\}}
\\  \omega^\alpha_{h}|_{t=0}= 0.&
\end{cases}
\eeq

The estimate of $\omega^\alpha_{nh}$ is  given as follows.
\begin{prop}
For $|\alpha|\le m-1$, the solution $\omega_{nh}^\alpha$ of \eqref{eqomeganh} satisfies the estimate
\begin{align}\label{propomeganh}
\norm{\omega^\alpha_{nh}(t)}_0^2 + \eps \int_{0}^t \norm{\nabla \omega^\alpha_{nh} }_0^2 \le& \Lambda_0\norm{ \omega(0)}_{\X^{m-1}}^2+\int_0^t \linf
\eps^2\({\norm{\nabla v}_{\X^{m-1,1}}^2+\abs{ h}_{\X^{m-1,\thal}}^2}\)
\\&\nonumber +\int_0^t \linf \(\abs{h}_{\X^{m-1,\hal}}^2+\norm{\pa_zv}_{\X^{m-1}}^2 +\norm{v}_{ {\X^{m-1,1}}}^2+\norm{ \omega^\alpha_{nh}}_0^2 \)
.
\end{align}
\end{prop}
\begin{proof}
Since $\omega^\alpha_{nh}=0$ on the boundary, it follows from \eqref{eqomeganh} that
\beq\label{omeganh1}
\hal\dtt \int_{\Omega}  \abs{\omega^\alpha_{nh}}^2 \, d\V
+ \eps \int_{\Omega}  \abs{ \nabla^\varphi \omega^\alpha_{nh}}^2 \, d\V
=\int_{\Omega}F \cdot \omega^\alpha_{nh}\, d\V.
\eeq
We now estimate the right hand side of \eqref{omeganh1}. \eqref{gues} implies that
\beq\label{cssss}
\norm{Z^\alpha \({\chi}\omega \cdot \nabla^\varphi v \)}
\le\linf \( \norm{\omega }_{\X^{m-1}}+ \norm{ \nabla^\varphi v}_{\X^{m-1}}\)
\le \linf\( \norm{\omega }_{\X^{m-1}}+ \norm{ \nabla  v}_{\X^{m-1}}+\abs{h}_{\X^{m-1,\hal}}\).
\eeq
{By the cut-off function $\chi$, one has  similarly
\beq\label{cssss23}
\norm{Z^\alpha\big( \big(V_{z} \partial_{z} \chi\big) \omega -  2\eps \nabla^\varphi \chi \cdot \nabla^\varphi \omega - \eps \Delta^\varphi
\chi \, \omega) }\le \linf\( \norm{v }_{\X^{m-1,1}}+ \varepsilon\norm{    v}_{\X^{m-1,2}}+\varepsilon\abs{h}_{\X^{m-1,\frac{3}{2}}}\).
   \eeq}
To estimate the part involving $\mathcal{C}_{S}$, one first can change $S_{\n}$ into ${\chi}\omega$ and $m$ into $m+1$ in \eqref{CS2} to get
\begin{align}\label{CS2omega}
\int_{\Omega} \mathcal{C}_{S}^2 \cdot \omega^\alpha_{nh}\, d\V
&\le \linf  \(\eps\(\norm{\nabla v}_{\X^{m-2,2}}+\abs{ h}_{\X^{m-2,\fhal}}\)+\abs{h}_{\X^{m-1,\hal}} \right.
\\&\nonumber\qquad\quad  +\norm{\pa_zv}_{\X^{m-1}} +\norm{v}_{ {\X^{m-1,1}}} \Big)\norm{ \omega^\alpha_{nh}} +\linf
\abs{h}_{\X^{m-1,\hal}}\eps^\hal \norm{\nabla \omega^\alpha_{nh}}  .
\end{align}
For the part involving $\mathcal{C}_{S}^1$, one can change $S_{\n}$ into $\omega$ and $m$ into $m+1$ in \eqref{CSy} to obtain
\beq\label{CSyomega}
\norm{\mathcal  C_{S1}^1}+\norm{\mathcal  C_{S1}^2}\le\linf\(  \norm{v}_{\X^{m-1}}+ \norm{\omega}_{\X^{m-1 }} + \abs{h}_{\X^{m-1,\hal }}  \).
\eeq
The commutator $\mathcal{C}_{S3}^1$ requires much more care. Indeed, one can not change $m$ into $m+1$ in \eqref{CSz} since it would involve
$|h|_{\X^{m,\hal} }$. As in \eqref{CSz0}, this commutator can be expanded into a sum of terms of the form
$$ c_{  \beta }Z^{  \beta }\(  {1  \over {z (z+b)}} v_{z} \)   \, Z_{3} Z^\gamma \omega$$
such that  $|  \beta| + |  \gamma| \leq m-1$, $| \gamma | \leq m-2$. As the arguments in previous sections, by \eqref{gues} and $v_{z}=0$ on the
boundary that, one deduces
\begin{align}\label{tqqq0}
\norm{Z^{  \beta }\(  {1  \over {z (z+b)}} v_{z} \)   \, Z_{3} Z^\gamma \omega}
&\le\linf\( \norm{ \omega }_{\X^{m-2,1}}+  \norm{ Z\( {1  \over {z (z+b)}} v_{z}\)}_{\X^{m-2}}\) \\
&\le\linf\( \norm{ \omega }_{\X^{m-2,1}}+ \norm{{1  \over {z (z+b)}}  v_{z}}_{\X^{m-1}} \).\nonumber
\end{align}
It follows from $v_z=0$ on the boundary and the Hardy inequality that, since $  v_{z}  =  v\cdot \N - \partial_{t} \eta  $,
\begin{align}\label{tqqq00}
&\norm{{1  \over {z (z+b)}}  v_{z}}_{\X^{m-1}} \ls  \norm{\partial_{z}  v_{z}}_{\X^{m-1}}
 \\\nonumber&\quad\le \linf \(\norm{\pa_z v}_{\X^{m-1}} +\norm{  v}_{\X^{m-1}} + \abs{  h}_{\X^{m-1,\hal}}\)+ \sum_{| \alpha |= m- 1} \norm{ v
 \cdot \partial_{z}Z^\alpha \N - \partial_{z} Z^\alpha \partial_{t} \eta },
\end{align}
where the last term requires furthermore analysis. Due to  \eqref{eqeta}, it holds that
$$ |Z_{3} \hat \eta| \lesssim  |\tilde \chi(|\xi| z ) \hat{h}|$$
where $\tilde \chi$ has a slightly bigger support than $\chi$. Hence $Z_{3}$ acts as a zero order operator:
\beq\label{Z3mieux}
\norm{\nabla Z_{3} \eta } \ls \abs{h}_{1\over 2}.
\eeq
This yields that if $\alpha_{3} \neq 0$, one gains at least one derivative, and thus,
\beq\label{tqqq1} \norm{ v \cdot \partial_{z}Z^\alpha \N - \partial_{z} Z^\alpha \partial_{t} \eta }  \le\linf\( \abs{h}_{\X^{m-2,\thal}} +
\abs{\partial_{t}
 h}_{\X^{m-2,\hal}} \) \le  \linf \abs{h}_{\X^{m-1,\hal}}.
 \eeq
Hence it suffices to estimate this term for the case with $| \alpha |=m-1$ and $\alpha_{3}= 0$. Note that \eqref{eqeta} implies that
\beq\label{etaconv}
\eta=\(1+\frac{z}{b}\){1 \over z^2} \psi\( {\cdot \over z}\)\star_{y} h:= \(1+\frac{z}{b}\)\psi_{z} \star_{y} h,
\eeq
where $\star_{y}$ stands for the convolution in the $y$ variable and $\psi$ is  in  $L^1(\mathbb{R}^2)$. Then
\begin{align}
 v \cdot \partial_{z}Z^\alpha \N - \partial_{z} Z^\alpha \partial_{t} \eta=&-
 \frac{1}{b}\( v_{y}\cdot \psi_z\star_{y} \nabla_{y} Z^{\alpha} h
+\psi_{z} \star_{y} \partial_{t} Z^{\alpha} h\)
\\& -\(1+\frac{z}{b}\)\pa_z\( v_{y}\cdot \psi_z\star_{y} \nabla_{y} Z^{\alpha} h
+\psi_{z} \star_{y} \partial_{t} Z^{\alpha} h\) := \mathcal{T}_{\alpha}=\mathcal{T}_{\alpha1}+\mathcal{T}_{\alpha2}.\nonumber
\end{align}
It is clear that
\beq \label{Talpha0}
 \norm{ \mathcal{T}_{\alpha1} }  \le\linf\( \abs{h}_{\X^{m-1,\hal}} + \abs{\partial_{t} h}_{\X^{m-1,-\hal}} \) \le  \linf \abs{h}_{\X^{m,-\hal}}.
 \eeq
For the second term, one can separate it into two cases. For $-b< z\le -\frac{b}{2}$, then
$$
\mathcal{T}_{\alpha2}=- \frac{1}{bz} Z_3\( v_{y}\cdot \psi_z\star_{y} \nabla_{y} Z^{\alpha} h
+\psi_{z} \star_{y} \partial_{t} Z^{\alpha} h\).
$$
It then follows from \eqref{Z3mieux} that
\beq\label{tqqq} \norm{ \mathcal{T}_{\alpha2} }_{L^2\( \{-b\le z\le -\frac{b}{2}\}\)}   \le\linf\( \abs{h}_{\X^{m-1,\hal}} + \abs{\partial_{t}
 h}_{\X^{m-1,-\hal}} \) \le  \linf \abs{h}_{\X^{m,-\hal}}.
 \eeq
For $-\frac{b}{2}\le z<0$, $\mathcal{T}_{\alpha2}$ can be written as
\beq\label{Talphadecfin}
\mathcal{T}_{\alpha2}= -\(1+\frac{z}{b}\)\(v_y(t,y,0)\cdot\partial_{z} \(  \psi_{z} \star_{y} \nabla_{y} Z^{\alpha} h \)
+\partial_{z}\(  \psi_{z} \star_{y} \partial_{t} Z^{\alpha} h \)\) +  \mathcal{R},
\eeq
 where
 $$ \abs{\mathcal{R}} \leq \linf|z(z+b)| \abs{ \partial_{z} \( \psi_{z} \star  Z^\alpha \nabla h\)}
 \leq \linf\abs{ Z_{3} \( \psi_{z} \star  Z^\alpha \nabla h \)}.$$
By using \eqref{Z3mieux} again, one can get that
\beq\label{Talphadecfin1}
\norm{\mathcal{R}}_{L^2\({\Omega}  \cap \{  -\frac{b}{2}\le z<0\}\)}
\leq \linf \abs{h}_{\X^{m-1,\hal} }.\eeq
To estimate the first term in \eqref{Talphadecfin}, one notes that
\begin{align*}  &  v_y(t,y,0)\cdot \partial_{z} \(  \psi_{z} \star_{y} \nabla_{y} Z^{\alpha} h \)
+\partial_{z}\(  \psi_{z} \star_{y} \partial_{t} Z^{\alpha} h \)  = \partial_{z} \( \psi_{z} \star_{y}\( v_y(t,y,0)\cdot \nabla_{y} Z^\alpha h
+ \partial_{t} Z^\alpha h \) \)\\
& \qquad\qquad\qquad\quad +\partial_{z} \(\int_{\mathbb{R}^2}\(  v_{y}(t,y,0) - v_{y}(t,y',0) \)\cdot \psi_{z}(y-y') \nabla_{y} Z^\alpha h(t,y')
\).
\end{align*}
For the second term in the right hand side of the above, one can employ the Taylor formula for $v_{y}$ to get  that
$$ \abs{ (v_{y}(t,y,0) - v_{y}(t,y', 0))\partial_{z} \psi_{z}(y-y') } \leq \linf  {1 \over z^2} \tilde{\psi}\({y-y' \over z}\), $$
where $\tilde{\psi} $ is still an $L^1$ function. This yields that
\begin{align}\label{Talphadecfin2}
&\sup_{z \in (- \frac{b}{2}, 0)}\abs{ \partial_{z} \(\int_{\mathbb{R}^2}\(  v_{y}(t,y,0) - v_{y}(t,y',0) \)\cdot \psi_{z}(y-y') \nabla_{y}
Z^\alpha h(t,y') \)}_{0}\leq \linf \abs{h}_{\X^{m-1,\hal}}.
\end{align}
For the first term, one shall use \eqref{bordh} to get
$$v_y\cdot\nabla_{y} Z^\alpha h  + \partial_{t} Z^\alpha h  = -\partial_z^\varphi v\cdot\N Z^\alpha h+ \mathcal{C}^\alpha(h)+  V^\alpha \cdot \N,
$$
which implies,  since $|\alpha|= m-1$,
$$\abs{ v_y\cdot\nabla_{y} Z^\alpha h + \partial_{t} Z^\alpha h  }_{ \hal }
       \leq \linf\( \abs{v}_{\X^{m-1,\hal}}+ \abs{h}_{\X^{m-1,\hal}}\).$$
It thus holds that
\begin{align}\label{Talphadecfin3}
\norm{  \partial_{z} \( \psi_{z} \star_{y}\( v_y(t,y,0)\cdot \nabla_{y} Z^\alpha h  - \partial_{t} Z^\alpha h \) \)}   & \ls  \abs{
v_y\cdot\nabla_{y} Z^\alpha h  - \partial_{t} Z^\alpha h  }_{ \hal } \\
&\leq  \linf \(\abs{v}_{\X^{m-1,\hal}}+ \abs{h}_{\X^{m-1,\hal}}\),\nonumber
\end{align}
Collecting all the estimates \eqref{Talphadecfin1}--\eqref{Talphadecfin3}, one deduces from \eqref{Talphadecfin} that
\beq\label{tqqq2} \norm{ \mathcal{T}_{\alpha}}_{L^2\({\Omega}  \cap \{z\ge -\frac{b}{2}\}\)} \leq \linf \(\abs{v}_{\X^{m-1,\hal}}+
\abs{h}_{\X^{m-1,\hal}}\).
\eeq
In view of the estimates \eqref{tqqq0}, \eqref{tqqq00}, \eqref{tqqq1}, \eqref{Talpha0}, \eqref{tqqq} and \eqref{tqqq2}, by the trace estimates,
one finally gets
\beq\label{CSzomega}
\norm{\mathcal{C}_{S3}^1} \le \linf \(\norm{\pa_z v}_{\X^{m-1}} +\norm{  v}_{\X^{m-1,1}} + \abs{h}_{\X^{m-1,\hal}}\).
\eeq

As a consequence of \eqref{cssss}--\eqref{CSyomega}, \eqref{CSzomega} and
Lemma \ref{mingrad}, we deduce from \eqref{omeganh1} that
\begin{align}
&\norm{\omega^\alpha_{h}(t)}^2 + \eps \int_{0}^t \norm{\nabla \omega^\alpha_{h} }^2  \le \Lambda_0\norm{ \omega(0)}_{\X^{m-1}}^2+\int_0^t \linf
\(\eps\(\norm{\nabla v}_{\X^{m-2,2}}+\abs{ h}_{\X^{m-2,\fhal}}\) \right.
\\&\nonumber\qquad \qquad  +\abs{h}_{\X^{m-1,\hal}}+\norm{\pa_zv}_{\X^{m-1}} +\norm{v}_{ {\X^{m-1,1}}} \Big)\norm{ \omega^\alpha_{nh}} +\linf
\eps^\hal\abs{h}_{\X^{m-1,\hal}} \norm{\nabla \omega^\alpha_{h} }  .
\end{align}
We then conlude the proposition by using Cauchy's inequality.
\end{proof}

It remains to estimate $\omega^\alpha_{h}$ of \eqref{eqomegah}. Note that Lemma \ref{lembord} and the trace estimate imply that
\begin{align}
\label{omegab2}
\sqrt{\eps}  \int_{0}^t\abs{Z^\alpha \omega}_0^2 &\leq \sqrt{\eps} \int_{0}^t\linf\(\abs{v}_{\X^{m-1,1}}^2 + \abs{h}_{\X^{m-1,1}}^2  \)
 \\\nonumber& \leq   \sqrt{\eps}\int_{0}^t \linf  \(\norm{\nabla v}_{\X^{m-1,1}}\norm{  v}_{\X^{m-1,1}}  + \norm{
 v}_{\X^{m-1,1}}^2+\abs{h}_{\X^{m-1,1}}^2\)  .
\end{align}
Thus the main difficulty will be to handle the non-homogeneous boundary value problem, \eqref{eqomegah}, whose boundary value is at a low level
of regularity (it is $L^2_{t,y}$ and no more) due to \eqref{omegab2}. This creates two
difficulties: the first is that one cannot lift the boundary condition easily and  perform a standard energy estimate; the second one  is
that due to the lower regularity of the boundary estimate, one cannot expect an estimate of  $\norm{\omega^\alpha_{h}}$ in $L^\infty(0,T)$ independent of
$\eps$, as was well explained in Section 10.2 of \cite{MasRou}: in using the model of the heat equation, one may expect a
control in $H^\frac{1}{4}(0,T)\subset L^4(0,T)$.

\begin{prop}\label{nor2}
Assume that $\linf(t)+\int_0^t\norm{\pa_{zz}v}_{1,\infty}\le M$ for $M>0$. Then for $|\alpha|\le m-1$, the solution $\omega_{h}^\alpha$ of
\eqref{eqomegah} satisfies the estimate
\beq\label{nor2estimate}
\norm{ \omega^\alpha_{h} }_{L^4_T L^2}^2 \leq  \Lambda(M) \int_0^T \( \norm{  v}_{\X^{m-1,1}}^2+\abs{h}_{\X^{m-1,1}}^2\)+\eps\int_0^T\norm{\nabla
v}_{\X^{m-1,1}}^2.
\eeq
\end{prop}
\begin{proof}
{Owing to the cut-off function $\chi$, the situation here of \eqref{eqomegah}, which is defined in the half space $\{z<0\}$, is exactly same as} Section 10.2 in \cite{MasRou}, thus this proposition is a restatement of Proposition 10.4 in
\cite{MasRou}. For completeness, we will sketch the main idea and some steps of the proof and state our estimates with a slight
modification.

First, it is convenient to eliminate the convection term in \eqref{eqomegah} by considering Lagrangian coordinates. Define a
parametrization of {$\tilde\Omega(t):=\{x_3<h(t,x_1,x_2)\},$ $X(t,\cdot):\{z<0\}\mapsto \tilde\Omega(t)$}, by
\beq\label{lagrange1}
\begin{cases}
\partial_{t} X(t,x)= u(t, X(t,x))= v(t,\Phi(t, \cdot)^{-1}\circ X)
\\ X(0, x)= \Phi(0,x),
\end{cases}
\eeq
where the map $\Phi(t, \cdot)$ is defined by \eqref{diff}. {Note that here $v$ is a smooth extension onto $\{z<0\}$ and hence $u$ is the corresponding smooth extension onto $\tilde\Omega(t)$.} Denote $J(t,x)= | \mbox{det }\nabla X(t,x)|$ for the Jacobian of the change of
coordinates. Then $J(t,x)= J_{0}(x)$ by the divergence free condition. Define
\beq\label{Omegaalphadef}
\Omega^\alpha=e^{-\gamma t} \omega_{h}^\alpha (t, \Phi^{-1}\circ X),
\eeq
where $\gamma >0$ is a large parameter to be chosen. Then  $\Omega^\alpha$ solves
\beq\label{omegalagrange}
a_{0}\big( \partial_{t}\Omega^\alpha+ \gamma \Omega^\alpha\big) - \eps  \partial_{i}\( a_{ij} \partial_{j}\Omega^\alpha \)= 0\quad\text{in
}{\{z<0\}}.
\eeq
Here $ a_{0}= |J_{0}|^{1\over 2}$ and the matrix $ (a_{ij})$ is defined by
$$ ( a_{ij} ) =  |J_{0}|^{1 \over 2} P^{-1}\text{ with } P_{ij}=  \partial_{i}X \cdot \partial_{j} X.$$

Due to the assumption $\linf(t)+\int_0^t\norm{\pa_{zz}v}_{1,\infty}\le M$, Lemma 10.5 in \cite{MasRou} holds. Then the
following estimates hold:
\begin{eqnarray}
\label{J0}   & &\norm{J}_{W^{1, \infty}}+\norm{J^{-1}}_{W^{1, \infty}} \leq  \Lambda_{0}, \\
\label{nablaX} & & \norm{ \nabla X}_{L^\infty}  + \norm{ \partial_{t} \nabla X }_{L^\infty} \leq \Lambda(M) e^{ t\Lambda(M)},  \\
\label{dtnablaX} & & \norm{ \nabla X}_{1, \infty} +  \norm{ \partial_{t} \nabla X}_{1, \infty}  \leq \Lambda(M)e^{ t\Lambda(M)}, \\
\label{ZnablaX}  & &\sqrt{\eps}\norm{ \nabla^2 X }_{1, \infty}+\sqrt{\eps}\norm{ \partial_{t} \nabla^2 X}_{L^\infty} \leq   \Lambda(M)(1+t) e^{t
\Lambda(M)}.
\end{eqnarray}
Indeed, the estimate \eqref{J0} follows directly by $J(t,x)= J_{0}(x)$. Next, \eqref{lagrange1} implies that
\beq\label{eqDX}
\partial_{t} \nabla X= \nabla v \nabla \Phi^{-1} \nabla X
\eeq
and hence
$$\norm{\nabla X}_{L^\infty} \leq  \norm{\nabla X(0)}_{L^\infty}+ \Lambda_{0} \int_{0}^t \norm{\nabla v}_{L^\infty} \norm{ \nabla X}_{L^\infty}
\leq \Lambda_{0} + \Lambda(M) \int_{0}^t \norm{ \nabla X}_{L^\infty}.$$
This yields the first part of  \eqref{nablaX} by the Gronwall inequality.
Next, applying one spatial conormal derivative to \eqref{eqDX}, one can get
$$\norm{\nabla X}_{1, \infty} \leq  \norm{\nabla X(0)}_{1, \infty}+ \Lambda_{0} \int_{0}^t \norm{\nabla v}_{1, \infty} \norm{ \nabla X}_{1,
\infty}
\leq \Lambda_{0} + \Lambda(M) \int_{0}^t \norm{ \nabla X}_{1, \infty}.$$
and hence the first part of \eqref{dtnablaX} follows from the Gronwall inequality.
The estimates hold also for the time derivative in \eqref{nablaX} and \eqref{dtnablaX} by using again the equation \eqref{eqDX}. To prove
\eqref{ZnablaX}, one applies $\sqrt{\eps}\pa_{x_i}$ to \eqref{eqDX} to find that
\beq\label{eqD2X}
\partial_{t} \sqrt{\eps}\pa_{x_i} \nabla X= \sqrt{\eps}  \(\pa_{x_i}\nabla v \nabla \Phi^{-1}\nabla X+\nabla v \pa_{x_i}\nabla \Phi^{-1}\nabla
X+\nabla v \nabla \Phi^{-1} \pa_{x_i}\nabla X\).
\eeq
It follows from \eqref{dtnablaX} that
$$ \sqrt{\eps }\norm{ \nabla^2 X}_{1, \infty}  \leq \Lambda(M)(1+t) + \Lambda(M) \int_{0}^t \sqrt{\eps} \norm{ \nabla^2 X}_{1, \infty}
+ \Lambda(M) e^{\Lambda(M) t} \int_{0}^t   \sqrt{\eps} \norm{ \nabla^2 v}_{1, \infty}$$
and hence, by using the assumption $\int_0^t\norm{\pa_{zz}v}_{1,\infty}\le M$,
$$  \sqrt{\eps }\norm{\nabla^2 X}_{1, \infty}  \leq \Lambda(M) ( 1+t) e^{\Lambda(M) t} + \Lambda(M) \int_{0}^t \sqrt{\eps} \| \nabla^2 X\|_{1,
\infty}$$
and so the  first part of \eqref{ZnablaX} follows from the Gronwall inequality. For the second part of \eqref{ZnablaX},
it follows by using again \eqref{eqD2X} and the fact that
$\Lambda_\infty$ involves the control of $ \sqrt{\eps}\norm{ \nabla ^2 v}_{L^\infty}$.

As consequences of the estimates \eqref{J0}--\eqref{ZnablaX}, one gets
\beq\label{compact}
a_{0}\geq m ,\ a_{3,3} \geq m,  \     (a_{ij}) \geq c_{0} \mbox{Id}
\eeq
by a suitable choice of $m$ and $c_0$ depending on $M$ for $t\in [0,1]$ and that
\begin{align}\label{compact2}
&\norm{( a_{0}, a_{ij}) }_{L^\infty} + \sqrt{\eps}\norm{\partial_{z} ( a_{0}, a_{ij})  }_{L^\infty}  \leq \Lambda(M),
 \\&\norm{\partial_{t,y}( a_{0}, a_{ij}) }_{L^\infty} + \sqrt{\eps}  \norm{\partial_{t,y}\nabla a_{ij}}_{L^\infty} \leq \Lambda(M).
\end{align}
Note that these coefficients are lack of uniform regularity with respect to normal variables.
To get the estimates for solutions of \eqref{omegalagrange}, the authors in  \cite{MasRou} use the paradifferential calculous to prove Theorem
10.6 in \cite{MasRou}, which yields that there exists $\gamma_{0}$ depending only on $M$ such that for $\gamma \geq \gamma_{0}$, the solution of
\eqref{omegalagrange} satisfies the estimate
\beq
\norm{\Omega^{\alpha}}_{H_T^{1 \over 4} L^2 }^2 \leq
\sqrt{\eps} \Lambda(M)\int_{0}^T  \abs{\Omega^{\alpha}}_0^2.
\eeq
By the Minkowski inequality and the one-dimensional Sobolev embedding $H^{1 \over 4} \subset L^4$,
$$\norm{\Omega^\alpha}_{L^4_T L^2 }^2
      \ls \norm{\Omega^\alpha}_{ L^2 L^4_T}^2
       \ls\norm{\Omega^{\alpha}}_{L^2 H_T^{1 \over 4} }^2=\norm{\Omega^{\alpha}}_{H_T^{1 \over 4} L^2 }^2,$$
one then has
\beq
\norm{\Omega^\alpha}_{L^4_T L^2 }^2
      \le
\sqrt{\eps} \Lambda(M)\int_{0}^T  \abs{\Omega^{\alpha}}_0^2 .
\eeq
This and \eqref{Omegaalphadef} show that
$$ \norm{\omega^\alpha_{h}}_{L^4_T L^2 }^2   \leq  \Lambda(M) \sqrt{\eps} \int_{0}^T\abs{Z^\alpha \omega}_0^2.$$
Thus \eqref{nor2estimate} follows from \eqref{omegab2} and Cauchy's inequality.
\end{proof}

\section{$L^\infty$ estimates}\label{secinfty}

In order to close the a priori estimates, we shall now estimate the $L^\infty$ norms contained in $\linf$, $\norm{\pa_z v
}_{\Y^{{\[\!\frac{m}{2}\!\]}+2}}$ and $\sqrt{\eps} \norm{\partial_{zz}v}_{L^\infty}$, and $\int_0^t\sqrt{\eps}\norm{\partial_{zz}v}_{1,\infty}$ that was
used in Proposition \ref{nor2}. We will prove that they can be bounded in terms of the quantities in the left hand side of the estimates of
Propositions \ref{propdzvm-2} and \ref{conormv1}, which will then be shown to be in $L^\infty$ in time.

The key point is to use again the quantity $S_\n$, and one has the following:
\begin{lem}\label{propLinfty1}
For any $k\in \mathbb{N}$:
\beq\label{dzv1infty1}
\norm{\pa_z v }_{\Y^k}\le\Lambda\(\frac{1}{c_0},\abs{h}_{\Y^{k+1}}\) \( \norm{S_\n }_{\Y^{k}}+\norm{v }_{\Y^{k+1}}\).
\eeq
Also,
\beq\label{dzv1infty11}
 \sqrt{\eps}\norm{\partial_{zz}v}_{L^\infty}\leq  \Lambda_0\(\sqrt{\eps} \norm{\partial_{z}S_{\n}}_{L^\infty} + \norm{S_{\n}}_{1,
 \infty}  + \norm{v}_{2, \infty}\)
 \eeq
 and
 \beq\label{dzv1infty111}
 \sqrt{\eps}\norm{\partial_{zz}v}_{1,\infty}\leq  \Lambda_0\(\sqrt{\eps} \norm{\partial_{z}S_{\n}}_{1,\infty} + \norm{S_{\n}}_{2,
 \infty}  + \norm{v}_{3, \infty}\).
 \eeq
\end{lem}
\begin{proof}
This follows again from \eqref{eqdzvnk} and \eqref{eqdzvPin}.
\end{proof}

As a consequence of  Lemma \ref{propLinfty1}, it suffices to estimate $ S_\n $. {Similarly as in Section \ref{sectionnorm2}, one may need only to estimate $   S_\n  $ near $\{z=0\}$}. In this step, as in \cite{MasRou12,MasRou}, it is convenient to use a coordinate system where  the Laplacian has the simplest
expression. We shall thus use a normal geodesic coordinate system in the vicinity of $\partial\Omega$. Note that this coordinate
system has not been used before since it requires more regularity for the boundary: to get an $H^m$ (or $\mathcal{C}^m$) coordinate system, the
boundary needs to be $H^{m+1}$ (or $\mathcal{C}^{m+1}$). Nevertheless, at this step, this is not a problem since one needs only to estimate a low
number of derivatives
of the velocity, say ${\[\!\frac{m}{2}\!\]}+2$, while the boundary is $H^m$ and $m$ can be as large as needed. We shall choose the cut-off
function $\chi$ in order to get a well defined coordinate system in the vicinity of  {$\{z=0\}$}.

Define a different parametrization of the vicinity of $\{z=0\}$ by
\begin{align} \label{diffgeo}
\begin{split}
\Psi(t, \cdot) : \,   {\{z<0\}} & \rightarrow   \Omega(t) \\
   (y,z) &\mapsto    \left( \begin{array}{ll}  y \\ h(t,y)   \end{array} \right) + z \n^b(t,y),
\end{split}
\end{align}
where $\n^b= \N^b/ |\N^b|$ is the unit exterior normal with $\N^b=(-\partial_{1}h, - \partial_{2} h, 1)$. Note that $\Psi(t, \cdot) $ is a
diffeomorphism from $\mathbb{R}^2\times( - \delta, 0)$ to a vicinity of $\partial \Omega(t)$  for some  $\delta$ which depends only on $c_{0}$,
and  for every $t \in [0, T^\es]$ thanks to \eqref{apriori}. By this parametrization, the induced Riemannian metric has the block structure
\beq\label{blockg}
g(t,y,z)= \( \begin{array}{cc} \tilde{g}(t,y,z) & 0  \\ 0  & 1 \end{array} \).
\eeq
Hence, the Laplacian in this coordinate system reads:
\beq\label{laplacien}
\Delta_{g}f= \partial_{zz}f +\hal \partial_{z} \big( \ln |g| \big) \partial_{z} f  + \Delta_{\tilde g} f,
\eeq
where $|g|$ denotes the determinant of the matrix $g$ and $\Delta_{\tilde{g}}$ is defined by
\beq
\Delta_{\tilde{g}} f=  { 1 \over |\tilde{g}|^{1 \over 2} }  \sum_{1 \leq i, \, j \leq 2} \partial_{y^i}\big(\tilde{g}^{ij} |\tilde{g}|^{1 \over
2} \partial_{y^j} f \big).
\eeq

To use this coordinate system, one shall first localize the equation for $S^\varphi v$ in a vicinity of  $\{z=0\}$. Set
\beq\label{Schidef}
S^\chi= \chi(z) S^\varphi v,
\eeq
where $\chi(z) \in [0, 1]$ is smooth compactly supported near $\{z=0\}$ and takes the value $1$ in a vicinity of $\{z=0\}$. \eqref{eqS} yields
that
\beq\label{Schieq}
\partial_{t}^\varphi S^\chi + (v \cdot \nabla^\varphi)S^\chi - \eps \Delta^\varphi S^\chi=  F_{S^\chi}\quad \text{in } {\{z<0\}},
\eeq
where
\beq\label{Fchidef}
F_{S^\chi}=  F^\chi + F_{v}
\eeq
with
\begin{eqnarray}
& & F^\chi= \big(V_{z} \partial_{z} \chi\big) S^\varphi v  -  \eps \nabla^\varphi \chi \cdot \nabla^\varphi S^\varphi v - \eps \Delta^\varphi
\chi \, S^\varphi v ,\label{fcex} \\
& & F_{v} = - \chi \big( D^\varphi)^2 q -{\chi \over 2}\big( (\nabla ^\varphi v)^2 + ((\nabla^\varphi v)^t)^2\big).\label{fvex}
\end{eqnarray}
Next, define implicitly $\tilde{S}$ in ${\tilde\Omega(t)}$ by $\tilde S (t,  \Phi(t,y,z))= S^\chi(t,y,z)$. Then \eqref{Schieq} yields
\beq
\partial_{t}   \tilde S  + u \cdot \nabla \tilde S  - \eps \Delta  \tilde S = F_{S^\chi}(t, \Phi(t, \cdot)^{-1})\quad\text{in }{\tilde\Omega(t)}.
\eeq
Finally, define $S^\Psi$ in $  {\{z<0\}} $ by  $S^\Psi(t,y,z)= \tilde S (t,\Psi(t,y,z))\equiv  S^\chi(t, \Phi(t, \cdot)^{-1} \circ \Psi)$. It then
follows from (8.14) and \eqref{laplacien} that
\beq\label{eqSPsi}
\partial_{t}  S^\Psi + w \cdot \nabla  S^\Psi - \eps\big(  \partial_{zz}   S^\Psi  + {1 \over 2 } \partial_{z} \big( \ln |g| \big) \partial_{z}
S^\Psi  + \Delta_{\tilde g}   S^\Psi \big)= F_{S^\chi} (t, \Phi^{-1} \circ \Psi(t,\cdot)) \quad \text{in } {\{z<0\}},
\eeq
where the vector field $w$ is given by
\beq\label{defW}
w=\overline{\chi} (\nabla \Psi)^{-1}\big( u(t,  \Psi ) - \partial_{t} \Psi \big)\equiv\overline{\chi} (\nabla \Psi)^{-1}\big( v(t, \Phi^{-1}
\circ \Psi ) - \partial_{t} \Psi \big).
\eeq
Note that $S^\Psi$ is compactly supported in $z$ in a vicinity of $\{z=0\}$. The function $\overline{\chi}(z)$ is a function with a slightly
larger support such that $\overline \chi S^\Psi= S^\Psi$. The introduction of this function allows to have $w$ also supported in a vicinity of
$\{z=0\}$. Note that on $\{z=0\}$, $w_3=0$. Indeed, since
\beq\label{wbord1}
\nabla \Psi(t,y,0)= \left( \begin{array}{ccc}
        1  & 0 & \n_{1}^b \\ 0 & 1 & \n_{2}^b   \\ \partial_{1} h & \partial_{2} h  & \n_{3}^b \end{array} \right),
\eeq
thus for any $Y\in \mathbb{R}^3$,
\beq\label{cdv} ((\nabla\Psi(t,y,0))^{-1} Y)_{3}= Y \cdot \n^b.
\eeq
Hence, \eqref{defW} implies that on $\{z=0\}$,
\beq\label{w3bord}
w_{3}= v \cdot \n^b - \partial_{t} h \n_{3}^b= {1 \over |\N^b|} \big( v \cdot \N^b - \partial_{t}h)= 0
\eeq
thanks to the kinematic boundary condition.

Now, set
\beq\label{etadef}
S_{\n}^\Psi(t,y,z)=\Pi^b(t,y) S^\Psi(t,y,z) \n^b(t,y)
\eeq
with $\Pi^b= \mbox{Id}- \n^b \otimes \n^b$. Note that $\Pi^b$ and $\n^b$  are independent of $z$. Moreover, since the equation
\eqref{eqSPsi} is compactly supported in $z$ in a vicinity of $\{z=0\}$, this yields that $S_{\n}^\Psi$ solves
\beq\label{eqSnPsi}
\partial_{t} S_{\n}^\Psi + w \cdot \nabla  S_{\n}^\Psi - \eps\big(  \partial_{zz}  + \hal \partial_{z} \big( \ln |g| \big) \partial_{z} \big)
S_{\n}^\Psi  = F_{\n}^\Psi\quad \text{in }\{z<0\},
\eeq
where
\beq\label{Fetadef}
F_{\n}^\Psi=  \Pi^b F_{S^\chi} \n^b + F_{\n}^{\Psi, 1}  + F_{\n}^{\Psi, 2}
\eeq
with
\begin{eqnarray}
& &   \label{Feta1}
F_{\n}^{\Psi, 1}=  \big(( \partial_{t} + w_{y} \cdot \nabla_{y}) \Pi^b\big)S^\Psi \n^b + \Pi^b S^\Psi\big( \partial_{t}+ w_{y}\cdot \nabla_{y})
n^b, \\& & \label{Feta2} F_{\n}^{\Psi, 2}=  - \eps \Pi^b\big( \Delta_{\tilde{g}} S^\Psi\big) \n^b.
\end{eqnarray}
On $\{z=0\}$, $S_{\n}^\Psi = S_{\n}=0$. Furthermore, it is convenient to eliminate the term $\eps \partial_{z} \ln |g| \partial_{z}$ in
the equation \eqref{eqSnPsi}.
Set
\beq\label{etadefbis}
\rho = |g|^{1 \over 4} S_{\n}^\Psi,
\eeq
then
\beq\label{eqrhobis}
\partial_{t}  \rho + w \cdot \nabla   \rho - \eps  \partial_{zz}   \rho= \mathcal{H}\quad \text{in }\{z<0\},
\eeq
where
\beq\label{Fg}
\mathcal{H}= |g|^{ 1 \over 4 } \(F_{\n}^\Psi + F_{g}\)\text{
with }
F_{g}=  \rho|g|^{-\hal} \( \partial_{t} + w \cdot \nabla - \eps \partial_{zz} \) |g|^{1 \over 4}.
\eeq
Trivially, on $\{z=0\}$, $\rho= 0$.

\subsection{Estimate of $\norm{\pa_z v}_{\Y^{{\[\!\frac{m}{2}\!\]}+2}}$}\label{sec linf1}

We now establish the first $L^\infty$ estimate. Note that it is equivalent to estimate $S_{\n}$ or $S_{\n}^\Psi$ or  $\rho$.  Indeed,  by the
definition \eqref{etadef},
using the chain rule and the fact that $Z$ is tangent to $\{z=0\}$, one has
$$ \norm{S_{\n}^\Psi}_{\Y^{k}}  \leq \Lambda \(\frac{1}{c_0},\abs{h}_{\Y^{k+1}}\) \norm{ \Pi^b S^\varphi v \, \n^b}_{\Y^{k}}.$$
Since $|\Pi- \Pi^b| + | \n- \n^b|= \mathcal{O}(z)$ in the vicinity of $\{z=0\}$, thus
\beq\label{etaSn}
\norm{S_{\n}^\Psi}_{\Y^{k}}  \leq \Lambda \(\frac{1}{c_0},\abs{h}_{\Y^{k+1}}\) \(\norm{ S_\n}_{\Y^{k}}+\norm{ v}_{\Y^{k+1}}\).
\eeq
Similar arguments show that
\beq\label{Sneta}
\norm{S_{\n}}_{\Y^{k}}  \leq \Lambda \(\frac{1}{c_0},\abs{h}_{\Y^{k+1}}\) \(\norm{ S_\n^\Psi}_{\Y^{k}}+\norm{ v}_{\Y^{k+1}}\).
\eeq
On the other hand, it is easy to see that it is equivalent to estimate $\rho$ or $S_{\n}^\Psi$. By \eqref{dzv1infty1}, it thus suffices to
estimate $\rho$.

Set
\beq
\widetilde{\Lambda}_{\infty}(t)=\Lambda\( {1 \over c_{0}},   \abs{h}_{\X^{{\[\!\frac{m}{2}\!\]}+7}}+\norm{\nabla v}_{\Y^{{\[\!\frac{m}{2}\!\]}+2}} +
\norm{v}_{\X^{{\[\!\frac{m}{2}\!\]+}7}}  + \norm{\nabla v}_{\X^{{\[\!\frac{m}{2}\!\]+}6}}   \).
\eeq

\begin{prop}\label{propdzvinfty}
For $m\ge {13}$, it holds that
\beq
\norm{ \rho(t) }_{\Y^{{\[\!\frac{m}{2}\!\]}+2}} \ls\norm{ \rho(0)}_{\Y^{{\[\!\frac{m}{2}\!\]}+2}}
+ \int_{0}^t \widetilde{\Lambda}_{\infty}\(1+\eps\norm{\nabla S_\n }_{\X^{m-2}}\)  .
\eeq
\end{prop}
\begin{proof}
Apply $Z^\alpha$, for $\alpha\in \mathbb{N}^{1+3}$ with $|\alpha|=k\le {\[\!\frac{m}{2}\!\]}+2$, to \eqref{eqrhobis} to obtain
\beq\label{eqetaFP0}
\partial_{t}Z^\alpha{\rho}+w\cdot\nabla Z^\alpha{\rho}-\eps\partial_{zz}Z^\alpha{\rho}=Z^\alpha {\mathcal{H}}+\mathcal{C}_S\quad \text{in
}\{z<0\},
\eeq
where
\beq\label{Ht}
\mathcal{C}_S=\mathcal{C}_S^1+\mathcal{C}_S^2
\eeq
with
\beq
\mathcal{C}_S^1=\[Z^\alpha,w_y\]\cdot  \nabla_y \rho+\[Z^\alpha,w_3\]\cdot  \pa_z \rho:=\mathcal{C}_{Sy}+\mathcal{C}_{Sz}\text{ and
}\mathcal{C}_S^2=-\eps\[Z^\alpha,\partial_{zz}\] \rho.
\eeq
The maximum principle on \eqref{eqetaFP0} yields that, since $Z^\alpha{\rho}=0$ on $\{z=0\}$,
\begin{align}\label{linfty1}
\norm{Z^\alpha{\rho}}_{L^\infty}\le \norm{Z^\alpha{\rho}(0)}_{L^\infty}+\int_0^t\(\norm{Z^\alpha {\mathcal{H}}}_{L^\infty}+\norm{
\mathcal{C}_S}_{L^\infty}\).
\end{align}
The right hand side of \eqref{linfty1} can be estimated as follows. For the commutator $\mathcal{C}_S^1$, the direct estimates yield
\beq\label{eeqq10}
\norm{\mathcal{C}_{Sy}}_{L^\infty}\ls \norm{w_y}_{\Y^k}\norm{\rho}_{\Y^k}.
\eeq
To estimate $\mathcal{C}_{Sz}$, by expanding the commutator and using \eqref{idcom}, one needs to estimate terms of the form
$$\norm{ Z^\beta  w_{3}  \partial_{z} Z^\gamma  {\rho} }_{L^\infty}$$
with $\beta +  \gamma \le \alpha$ and $|\gamma|\le |\alpha|-1$. Since $w_3=0$ on $\{z=0\}$, so
\beq\label{eeqq}
\norm{ Z^\beta  w_{3}  \partial_{z} Z^\gamma  {\rho} }_{L^\infty}=\norm{\frac{ Z^\beta  w_{3}}{z(z+b)}  Z_3 Z^\gamma  {\rho}
}_{L^\infty}\ls\norm{\pa_{z}w_3}_{\Y^k} \norm{ \rho}_{\Y^k}.
\eeq
For the commutator $\mathcal{C}_S^2$, using \eqref{idcom} repeatedly leads to
\begin{align}\label{idcom10}
- \[Z^\alpha,\partial_{zz}\] \rho&=\pa_z\(\[Z^\alpha, \partial_{z}\]{\rho}\) +\[Z^\alpha,\pa_z\] \partial_{z}{\rho}
\\\nonumber& =\sum_{| \beta | \leq k-1} \pa_z\( c_{\beta } \partial_{z} ( Z^\beta  \rho)\) +\sum_{| \beta | \leq k-1} c_{\beta } \partial_{z} (
Z^\beta \pa_z\rho)
\\\nonumber& =\sum_{| \beta' | \leq k-1}  c_{\beta' } Z^{\beta'} \partial_{z} \rho   +\sum_{| \beta' | \leq k-1} \tilde c_{\beta'}  Z^{\beta'}
\pa_{zz}\rho.
\end{align}
Hence,
\beq\norm{ \mathcal{C}_S^2 }_{L^\infty}\ls \eps \norm{ \pa_z  {\rho} }_{\Y^{k-1}}+\eps \norm{ \pa_{zz}  {\rho} }_{\Y^{k-1}}.
\eeq
Note that using the equation \eqref{eqrhobis} implies
\beq
\eps \norm{ \pa_{zz}  {\rho} }_{\Y^{k-1}} \le \norm{\partial_{t}  \rho + w \cdot \nabla   \rho-\mathcal{H}}_{\Y^{k-1}}
\le \norm{ \rho}_{\Y^{k}} +\norm{  w \cdot \nabla   \rho}_{\Y^{k-1}}+\norm{ \mathcal{H}}_{\Y^{k-1}}.
\eeq
Recall from \eqref{dzzves11} that
\beq
\eps \norm{ \pa_z  {\rho} }_{\Y^{k-1}}\le \eps \Lambda\(\frac{1}{c_0},\abs{h}_{\Y^{k+1}}\)\norm{ \nabla^2  v }_{\Y^{k-1}}\le \linf.
\eeq
As \eqref{eeqq10} and \eqref{eeqq},
\beq\label{eeqq333}\norm{  w \cdot \nabla   \rho}_{\Y^{k-1}}\ls \(\norm{ w_y}_{\Y^k}+\norm{\pa_{z}w_3}_{\Y^k}\) \norm{ \rho}_{\Y^k}.
\eeq
Thus,
\beq\label{eeqq111}\norm{ \mathcal{C}_S^2 }_{L^\infty}\ls \(1+\norm{ w_y}_{\Y^k}+\norm{\pa_{z}w_3}_{\Y^k}\) \norm{ \rho}_{\Y^k}+\norm{
\mathcal{H}}_{\Y^{k-1}}+\linf.\eeq
Consequently, in light of the estimates \eqref{eeqq10}, \eqref{eeqq} and \eqref{eeqq111} into \eqref{linfty1}, one obtains
\beq\label{rhoes1}
\norm{  \rho }_{\Y^k} \ls\norm{  \rho(0) }_{\Y^k}
+ \int_{0}^t  \norm{ \mathcal{H}}_{\Y^{k}}+\(1+\norm{ w_y}_{\Y^k}+\norm{\pa_{z}w_3}_{\Y^k}\) \norm{ \rho}_{\Y^k}+\linf.
\eeq

Now we estimate $\mathcal{H}$. Note first that
$$
\norm{ |g|^{\frac{1}{4}}F_{g} }_{\Y^{k}}\le \Lambda\(\frac{1}{c_0},\norm{\rho}_{\Y^{k}}+\norm{w}_{\Y^{k}}+\abs{h}_{\Y^{k+3}}\).
$$
Next, it follows from \eqref{Feta1} and \eqref{Feta2} that
$$ \norm{ |g|^{\frac{1}{4}}F_\n^{\Psi, 1} }_{\Y^{k}} \le \Lambda\( {1 \over c_{0}}, \abs{h}_{\Y^{k+2}}
     + \norm{w }_{\Y^{k}}+ \norm{\nabla v }_{\Y^{k}}\)$$
     and
     $$  \norm{ |g|^{\frac{1}{4}}F_\n^{\Psi,2} }_{\Y^{k}} \le \eps \Lambda\( {1 \over c_{0}}, \abs{h}_{\Y^{k+3}}
    \)\norm{\nabla v }_{\Y^{k+2}}.$$
Using \eqref{fcex} and \eqref{fvex}, the fact that $F^\chi$  is supported away from
$\{z=0\}$ and Lemma \ref{iii}, one gets
$$  \norm{ |g|^{\frac{1}{4}}\Pi^b F ^\chi \n^b }_{\Y^{k}}\le \Lambda\( {1 \over c_{0}}, \abs{h}_{\Y^{k+2}}
+\norm{ v }_{\X^{k+4}}\)$$
and
$$  \norm{ |g|^{\frac{1}{4}}\Pi^b F_v \n^b }_{\Y^{k}}\le \Lambda\( {1 \over c_{0}}, \abs{h}_{\Y^{k+2}}+\norm{\nabla v }_{\Y^{k}}
     \)\(1+ \norm{\nabla^2 q }_{\Y^{k}}\).$$
Recalling (8.27), (8.22) and (8.11), and collecting these estimates, one arrives at
\begin{align}\label{hhes}
\norm{\mathcal{H}}_{\Y^{k}}\le \Lambda\( {1 \over c_{0}}, \abs{h}_{\Y^{k+3}}+\norm{ v }_{\X^{k+4}}+\norm{\nabla v }_{\Y^{k}}
     \)\(1+\eps\norm{\nabla v }_{\Y^{k+2}}+ \norm{\nabla^2 q }_{\Y^{k}}\).
\end{align}
Recall from \eqref{qinfty} that
\begin{align} \label{qinfty1}
\norm{\nabla^2 q }_{\Y^{{\[\!\frac{m}{2}\!\]}+2}}
 &\nonumber\leq \Lambda\( {1 \over c_{0}},   \abs{h}_{\Y^{{\[\!\frac{m+6}{4}\!\]+}4}}  +\abs{h}_{\X^{{\[\!\frac{m}{2}\!\]+7}}}+\norm{\nabla
 v}_{\Y^{{\[\!\frac{m+6}{4}\!\]+}3}} +  \norm{v}_{\X^{{\[\!\frac{m}{2}\!\]+7}}}  + \norm{\nabla v}_{\X^{{\[\!\frac{m}{2}\!\]+6}}}   \)
\\& \leq \Lambda\( {1 \over c_{0}},   \abs{h}_{\Y^{{\[\!\frac{m}{2}\!\]}+5}}  +\abs{h}_{\X^{{\[\!\frac{m}{2}\!\]}+7}}+\norm{\nabla v}_{\Y^{{\[\!\frac{m}{2}\!\]}+2}} +
\norm{v}_{\X^{{\[\!\frac{m}{2}\!\]+}7}}  + \norm{\nabla v}_{\X^{{\[\!\frac{m}{2}\!\]+}6}}   \)
\end{align}
if $m\ge 10$. Hence, \eqref{rhoes1} implies
\begin{align}
\norm{  \rho }_{\Y^{{\[\!\frac{m}{2}\!\]}+2}} \ls\norm{  \rho_0 }_{\Y^{{\[\!\frac{m}{2}\!\]}+2}}
+ \int_{0}^t  &  \Lambda\( {1 \over c_{0}},   \abs{h}_{\X^{{\[\!\frac{m}{2}\!\]}+7}}+\norm{\nabla v}_{\Y^{{\[\!\frac{m}{2}\!\]}+2}} +  \norm{v}_{\X^{{\[\!\frac{m}{2}\!\]+}7}}
+ \norm{\nabla v}_{\X^{{\[\!\frac{m}{2}\!\]+}6}}   \)\nonumber
\\&\quad \times\(1+\eps\norm{\nabla v }_{\Y^{{\[\!\frac{m}{2}\!\]}+4}}\)  .
\end{align}
Then the desired estimates (8.29) follows for $m\ge {13}$ so that ${\[\!\frac{m}{2}\!\]}+4+1\le m-2$.
\end{proof}

\subsection{Estimate of $\sqrt{\eps}\norm{\pa_{zz} v}_{L^\infty}$}\label{sec linf2}

The next $L^\infty$ estimate is the only place where one needs to use the compatibility condition {\eqref{comc}} on the initial data.  As in the previous
subsection, by  \eqref{dzv1infty11}, one can reduce
the problem to the estimate of $\sqrt{\eps} \norm{\partial_{z} \rho }_{L^\infty}.$
\begin{prop}\label{dzzvLinfty}
Assume that the initial data satisfies the compatibility condition $S_\n(0)=0$ on $\{z=0\}$. Then it holds that for $m \geq 6$,
\beq
\sqrt{\eps}\norm{\partial_{z} \rho (t)}_{L^\infty}\ls  \sqrt{\eps} \norm{\partial_{z} \rho_{0}}_{L^\infty} + \int_{0}^t  {\linf \over \sqrt{t-
\tau}}.
\eeq
\end{prop}
\begin{proof}
The proof follows the spirit of the proof of Proposition 9.8 in \cite{MasRou}. Recall that $ \rho $ satisfies \eqref{eqrhobis} in $\{z<0\}$ with
$\rho=0$ on $\{z=0\}$. Note that one can not apply $\sqrt{\eps}\partial_{z}$ to \eqref{eqrhobis}
and then use the maximum principle due to boundary condition. We shall use a precise description of the
solution of \eqref{eqrhobis}. Indeed, one can use the one-dimensional heat kernel of  $\{z<0\}$:
\beq
G(t,z,z')= {1 \over \sqrt{4\pi \eps t}}\( e^{- { (z-z')^2 \over 4\eps t }} -  e^{- { (z+z')^2 \over 4\eps t }} \)
\eeq
to write that
\begin{align}\label{duhamel}
\sqrt{\eps}\partial_{z} \rho(t,y,z)=&\int_{0}^{+ \infty} \sqrt{\eps}\partial_{z}G(t,z, z') \rho_{0}(y, z') dz'
\\&\nonumber+ \int_{0}^t \sqrt{\eps} \partial_{z} G(t-\tau,z,z')\( \mathcal{H}(\tau,y,z') - w \cdot \nabla \rho\) dz' d\tau.
\end{align}
Since $\rho_{0}=0$ on $\{z=0\}$, thanks to the compatibility condition, one can integrate by parts the first term to obtain
\beq\label{dzzeta1}
\sqrt{\eps}\norm{\partial_{z} \rho (t)}_{L^\infty}\ls  \sqrt{\eps} \norm{\partial_{z} \rho_{0}}_{L^\infty}  + \int_{0}^t  {1 \over \sqrt{t-
\tau}} \(\norm{\mathcal{H}}_{L^\infty}+ \norm{ w \cdot \nabla \rho }_{L^\infty}\).
\eeq
Next, it follows from \eqref{hhes} with $k=0$ that
\beq\label{eeqq444}
\norm{ \mathcal{H}}_{L^\infty} \le \linf\( 1+ \norm{\nabla^2 q}_{L^\infty}\)  \le \linf.
\eeq
On the other hand, as \eqref{eeqq333},
\beq\label{eeqq555}
\norm{w \cdot \nabla \rho }_{L^\infty}   \le \(\norm{w_y}_{L^\infty} +\norm{\pa_{z}w_3}_{L^\infty} \) \norm{ \rho}_{1,\infty}  \le \linf.
\eeq
Consequently, plugging \eqref{eeqq444} and \eqref{eeqq555} into \eqref{dzzeta1} yields that
\beq
\sqrt{\eps}\norm{\partial_{z} \rho (t)}_{L^\infty}\ls  \sqrt{\eps} \norm{\partial_{z} \rho_{0}}_{L^\infty} + \int_{0}^t  {\linf \over \sqrt{t-
\tau}},
\eeq
which completes the proof.
\end{proof}

\subsection{Estimate of $\sqrt{\eps}\norm{\pa_{zz} v}_{1,\infty}$}

The last $L^\infty$ estimate is the one that was used in Proposition \ref{nor2}. By \eqref{dzv1infty111}, one can again reduce the problem to the
estimate of $\sqrt{\eps
}\norm{ \partial_{z}  \rho}_{1,\infty}$.
\begin{prop}\label{propro1}
For $m \geq 6$, it holds that:
\beq\label{Linfty0t1}
\int_0^t\sqrt{\eps }\norm{ \partial_{z}  \rho}_{1,\infty} \ls   \sqrt{t}  \norm{ \rho _{0}}_{1, \infty}+ t \int_{0}^t {\linf\over \sqrt{t -
\tau}}.
\eeq
\end{prop}
\begin{proof}
We will use a different argument from the proof of Lemma 9.9 in \cite{MasRou}.  A direct use of  the Duhamel formula \eqref{duhamel} yields
\beq
\sqrt{\eps }\norm{ \partial_{z}  \rho}_{1,\infty} \ls {1 \over \sqrt{t}} \norm{ \rho _{0}}_{1, \infty}+ \int_{0}^t{1 \over \sqrt{t- \tau}}\(
\norm{ \mathcal{H}}_{1, \infty} + \norm{ w \cdot \nabla \rho }_{1, \infty} \).
\eeq
Next, it follows as the previous arguments that
\beq
\norm{ \mathcal{H}}_{1,\infty} \le \linf\( 1+ \norm{\nabla^2 q}_{1,\infty}\)  \le \linf
\eeq
and
\beq
\norm{w \cdot \nabla \rho }_{1,\infty}   \le \linf \(\norm{w_y}_{2,\infty}+ \norm{\pa_{z}w_3}_{2,\infty} \)\norm{ \rho }_{2, \infty} \le \linf.
\eeq
Consequently,
\beq
\sqrt{\eps }\norm{ \partial_{z}  \rho}_{1,\infty} \ls {1 \over \sqrt{t}} \norm{ \rho _{0}}_{1, \infty}+ \int_{0}^t {\linf\over \sqrt{t - \tau}}.
\eeq
Integration in time yields
\beq
\int_0^t\sqrt{\eps }\norm{ \partial_{z}  \rho}_{1,\infty} \ls   \sqrt{t}  \norm{ \rho _{0}}_{1, \infty}+ t \int_{0}^t {\linf\over \sqrt{t -
\tau}}.
\eeq
which completes the proof of the proposition.
\end{proof}

\section{Proof of Theorem \ref{main}}\label{finalsec}

In this section, collecting the estimates  obtained in Sections \ref{secconormal}--\ref{secinfty}, we can prove Theorem \ref{main} in the similar way as that for
Theorem
1.1 of \cite{MasRou} with slight modifications.

Recall $\mathcal{N}(T)$ and $\mathcal{Q}(T)$. For two parameters $R$ and $c_{0}$ to be chosen $1/c_{0}\ll R$, define
\begin{align}\label{TTTTTT}
        T^{\es}_{*}= \sup_{T>0}\left\{T\in [0, 1] \mid  \mathcal{N}(t) \leq R, \ \abs{h(t)}_{2, \infty} \leq {1 \over c_{0}},\
        \partial_{z} \varphi(t) \geq \frac{{c_{0}}}{2}\right.
        \\ \left.\qquad\qquad\qquad\qquad\qquad\qquad\qquad\text{ and }  g  -   \partial_{z}^\varphi q (t)  \geq  {c_{0}\over
2}\text{ on } \{z=0\}, \ \forall t \in [0, T].
        \right\}.\nonumber
\end{align}
Proposition \ref{propro1} yields
\beq \int_{0}^T\sqrt{\eps }\norm{\pa_{zz}v}_{1,\infty} \leq \Lambda(R).
\eeq
This allows one to use Proposition \ref{nor2}, which together with Propositions \ref{conormv1}, \ref{conormvm}, \ref{propdzvm-2},
\ref{propdzvinfty} and \ref{dzzvLinfty} implies that, by a suitable linear combination,
\beq\label{Em1}
\mathcal{N}(t)
       \leq \Lambda\( {1 \over c_{0}}, {R_0}\) + \Lambda(R)\( T^{1 \over 2} +\mathcal{W}(T)+\mathcal{W}(T)^\hal\),
\eeq
   where
$$\mathcal{W}(T):=\int_0^T \( \abs{\dt^mh}_{0}^2 + \sigma     \abs{\dt^mh}_{1}^2
+ \norm{\dt^m v}_0 ^2 + \norm{\pa_z v}_{\X^{m-1}}^2\). $$
It follows from the Cauchy-Schwarz inequality that
\beq
\mathcal{W}(T)\le \mathcal{N}(T)T^\hal \le R T^\hal. \eeq
Hence, one deduces from \eqref{Em1} that
       \beq
       \label{Em2}
        \mathcal{N}(t)
       \leq \Lambda\( {1 \over c_{0}}, {R_0}\) + \Lambda(R) T^{1 \over 4}.
 \eeq
On the other hand, since $\mathcal{N}(T)$ involves time derivatives, one gets easily that
\beq\label{choixc01}
\abs{h(t)}_{2, \infty} \leq \abs{h(0)}_{2, \infty}+ \Lambda(R) T,\eeq
\beq\label{choixc02}
\partial_{z} \varphi(t) \geq  \partial_{z} \varphi(0)- \Lambda(R) T
\eeq
and
\beq\label{taylorcon}
g  -   \partial_{z}^\varphi q (t) \geq  g  -   \partial_{z}^\varphi q (0) - \Lambda(R) T.
\eeq
Consequently, one can choose $c_0$ so that $\abs{h(0)}_{2, \infty}\le \frac{1}{2c_0}$ and then $R= 2 \Lambda\( {1 \over c_{0}},
{R_0}\)$, then there exists $T_{*}$ which depends only on
$ R$ so that for $T \leq \min(T_{*}, T^{\es}_{*})$,
$$  \mathcal{N}(t) \leq \frac{3R}{4}, \quad \partial_{z} \varphi(t) \geq \frac{{3c_{0}}}{4}, \ \abs{h(t)}_{2, \infty} \leq {3 \over 4c_{0}}\text{
and }  g  -   \partial_{z}^\varphi q (t)  \geq  {3c_{0}\over
4}\text{ on } \{z=0\}, \ \forall t \in [0, T].$$
This yields  $T^{\es}_{*} \geq T_{*}$ by the definition \eqref{TTTTTT} and also the estimate \eqref{mainborne1}. The proof of Theorem
\ref{main} is thus completed. \hfill$\Box$

{
\section{Proof of Theorem \ref{main2}}\label{finalsec2}

In this section, we will prove Theorem \ref{main2} by first proving the uniform in $\varepsilon$ and $\sigma$ local well-posedness of \eqref{NSv} and then showing the inviscid limit. Consider the initial data $v ^\es_0\in H^{2m}(\Omega)$ and $h^\es_0\in H^{2m+1}(\Sigma)$ satisfying the assumptions in Theorem \ref{main2}.  Then for fixed $\varepsilon>0$ and $\sigma>0$, according to the existence result of \cite{TW14}, by the $m$-th compatibility conditions \eqref{compp} one can get a positive time $T^\es$ for which a unique solution $(v^\es,h^\es)$ of \eqref{NSv} achieving this initial data exists on $[0, T^\es]$.

As in \cite{MasRou}, an important remark is that if $\mathcal{N}(T_1)<+\infty$, then the solution above can be continued on $[0,T_2]$, $T_2 >T_1$ with  $\mathcal{N}(T_2)<+\infty$. Indeed, if $\mathcal{N}(T_1)<+\infty$, one can use the parabolic regularity for \eqref{NSv} on $[T_1/2, T_1]$ as \cite{beale_2} to get that the solution actually is smooth on $[T_1/2, T_1]$ and in particular, one finds that $v^\es(T_1)\in H^{2m}(\Omega)$ and $h^\es(T_1)\in H^{2m+1}(\Sigma)$ and that the $m$-th compatibility conditions \eqref{compp} hold at the time $T_1$. These allow one to use again the existence result of \cite{TW14} to continue the solution. Consequently, by this remark, from Theorem \ref{main}  one has the uniform estimate $\mathcal{N}(T)\le C$ and hence that  the
  solution $(v^\es,h^\es)$ actually exists on $[0,T]$.

The uniform estimate $\mathcal{N}(T)\le C$  allows one to deduce that as $\varepsilon\rightarrow 0$, up to extraction of a subsequence,   $(v^\es,h^\es)$ converges to a limit $(v^\sigma,h^\sigma)$ in the norms of any spaces which contain the set of functions obeying \eqref{mainborne12} as a compact subset (recalling that we have the time derivatives estimates in $\mathcal{N}(T)$). These convergences are more than sufficient for one to
pass to the limit in \eqref{NSv} for each $t\in
[0,T]$. Then one finds that the limit $(v^\sigma,h^\sigma)$ is a
strong solution of the free-surface  Euler equations \eqref{Eulerv} on $[0,T]$ that takes the initial data $(v_0^\sigma,h_0^\sigma)$ and satisfies the estimate \eqref{mainborne12}. Note that, one can prove, as in \cite{MasRou,EL14}, the uniqueness of solutions to \eqref{Eulerv} satisfying \eqref{mainborne12}. This implies in turn that the whole family $(v^\es,h^\es)$ converges to $(v^\sigma,h^\sigma)$. The proof of Theorem
\ref{main2} is thus completed. \hfill$\Box$}

  \appendix

 \section{Sobolev conormal spaces}\label{aa1}

 We recall the Sobolev conormal spaces $\X^m$ and $\Y^m$ from \eqref{xym}.
  \begin{lem}
  \label{sob}
The following product and commutator estimates hold.

$(i)$
       For $ |\alpha|+|\beta|=k \geq 0$:
            \beq
     \label{gues}
     \norm{ Z^{\alpha} fZ^{\beta} g }\lesssim \norm{f}_{\X^k}\norm{g}_{\Y^{{\[\!\frac{k}{2}\!\]}}}+ \norm{f}_{\Y^{{\[\!\frac{k}{2}\!\]}}}   \norm{g}_{\X^k}
     .
      \eeq

$(ii)$ For  $ | \alpha | = k \ge 1$:
  \beq
  \label{com}
 \norm{\[Z^\alpha, f\]g } \lesssim \norm{Zf}_{\X^{k-1}} \norm{g}_{\Y^{{\[\!\frac{k-1}{2}\!\]}}} + \norm{Zf}_{\Y^{{\[\!\frac{k-1}{2}\!\]}}} \norm{g}_{\X^{k-1}}
  \eeq

$(iii)$ For $| \alpha |=k \geq 2$, define the symmetric commutator
\beq\label{symcom}
\[Z^\alpha, f, g\]= Z^\alpha (f g) - Z^\alpha f \, g - f  Z^\alpha \,g.
\eeq
   Then
   \beq
   \label{comsym}\norm{\[Z^\alpha, f, g\] } \lesssim \norm{Zf}_{\X^{k-2}} \norm{Zg}_{\Y^{{\[\!\frac{k}{2}\!\]-1}}} + \norm{Zf}_{\Y^{{\[\!\frac{k}{2}\!\]-1}}}
   \norm{Zg}_{\X^{k-2}} .
   \eeq
   \end{lem}
     \begin{proof}
  The product estimate \eqref{gues} follows by controlling the product with the lower order derivative term in $L^\infty$ and the higher order
  derivative term in $L^2$. To prove the commutator estimate \eqref{com}, one uses the Leibnitz formula to expand
   $$\[Z^\alpha, f\]g =  \sum_{  \beta + \gamma = \alpha\atop
    \beta \neq 0   } C_{\beta, \gamma}   Z^\beta f  Z^\gamma g.$$
 Since $\beta \neq 0$, one can write $Z^\beta =  Z^{\beta- \beta'} Z^{\beta'}$ with $ |\beta' |=1$. Then \eqref{gues} yields
    \begin{align*}
   \norm{ Z^{\beta- \beta'} Z^{\beta'}f   Z^\gamma g } \lesssim \norm{Z^{\beta'}f}_{\X^{k-1}}\norm{g}_{\Y^{{\[\!\frac{k-1}{2}\!\]}}} +
   \norm{Z^{\beta'}f}_{\Y^{{\[\!\frac{k-1}{2}\!\]}}} \norm{g}_{\X^{k-1}}.
     \end{align*}
This proves \eqref{com}. The commutator estimate \eqref{comsym} can be proved in the same way.
          \end{proof}

    We shall also use the Sobolev tangential spaces defined by
  $$ H^s_{tan}(\Omega)= \left\{ f \in L^2(\Omega),\ \norm{f}_{H^{s}_{tan}}= \norm{ \Lambda^s f }_{L^2}<\infty\right\},\quad s\in \mathbb{R},  $$
  where $\Lambda^s$ is the tangential Fourier multiplier by $\big(1+ |\xi|^2 \big)^{s \over 2},\ \xi\in \mathbb{R}^2$.
   Note that
   $$ \norm{f}_{H^s_{tan}} \ls  \norm{ f}_{k}  \text{ for }s\le k,\ k \in \mathbb{N}.$$
\begin{lem}
\label{proptrace}
The following anisotropic Sobolev embedding and trace estimates hold.

$(i)$ For $s_{1}+ s_{2}>2,\ s_3+s_4 >2$:
  \beq
  \label{emb}
  \norm{f }_{L^\infty}  \ls   \norm{\partial_{z}f}_{H^{s_{1}}_{tan}}^\hal \norm{f}_{H^{s_{2}}_{tan}}^\hal+\norm{f}_{H^{s_{3}}_{tan}}^\hal
  \norm{f}_{H^{s_{4}}_{tan}}^\hal.
  \eeq

$(ii)$ For $s_{1}+ s_{2}=s_3+s_4= 2 s$:
  \beq
  \label{trace}
\abs{f }_{s} \ls \norm{ \partial_{z} f }^\hal_{H^{s_{1}}_{tan}}\norm{f}^\hal_{H^{s_{2}}_{tan}}+\norm{f}_{H^{s_{3}}_{tan}}^\hal
\norm{f}_{H^{s_{4}}_{tan}}^\hal.
 \eeq
  \end{lem}
\begin{proof}
We need to modify the proof of Proposition 2.2 in \cite{MasRou} since our domain here is of finite depth. To get the anisotropic Sobolev
embedding estimate \eqref{emb}, one first notes that
 $$ \abs{\hat f (\xi, z)} \leq \abs{\hat f (\xi, z')}+ \( \int_{-z'}^z 2\abs{\partial_{z} \hat f(\xi, x_3 ) }\abs{ \hat f (\xi, x_3)} \, dx_3
 \)^{1 \over 2 }.$$
Integrating the inequality above with respect to $z'\in(-b,0)$ yields
 $$ \abs{\hat f (\xi, z)} \ls \int_{-b}^0 \abs{\hat f (\xi, z')}dz'+ \( \int_{-b}^0  \abs{\partial_{z} \hat f(\xi, x_3 ) }\abs{ \hat f (\xi,
 x_3)} \, dx_3  \)^{1 \over 2 }.$$
Hence, it follows from the Cauchy-Schwarz inequality and the fact that $s_{1}+ s_{2}>2,\ s_3+s_4 >2$  that
\begin{align}
 \nonumber&\norm{f}_{L^\infty} \leq \sup_{z\in (-b,0)}\int_{\mathbb{R}^2_\xi}\abs{ \hat f(\xi, z)} \, d\xi
 \\\nonumber&\quad\ls \int_{\mathbb{R}^2_\xi} \int_{-b}^0\abs{\hat f (\xi, z')}\, d\xi dz' + b \int_{\mathbb{R}^2_\xi}\( \int_{-b}^0
 \abs{\partial_{z} \hat f(\xi, x_3 ) }\abs{ \hat f (\xi, x_3)} \, dx_3  \)^\hal  d\xi
 \\\nonumber&\quad\ls  \(\int_{-b}^0\int_{\mathbb{R}^2_\xi} \( 1 + |\xi |\)^{s_{3}+ s_{4}}\abs{\hat f (\xi, z )}^2  d\xi dz\)^\hal + \(
 \int_{\mathbb{R}_\xi^2}\( 1 + |\xi |\)^{s_{1}+ s_{2}}
   \int_{-b}^0 \abs{\partial_{z} \hat f(\xi, z )}\abs{\hat f (\xi, z)} dz  d \xi \)^\hal
    \\\nonumber&\quad\ls \norm{  \Lambda^{s_{3}} f }^\hal \norm{\Lambda^{s_{4}} f}^\hal+\norm{ \partial_{z} \Lambda^{s_{1}} f }^\hal
    \norm{\Lambda^{s_{2}} f}^\hal.
\end{align}

To prove the trace estimate \eqref{trace}, since $s_1+s_2=2s$, one may write
\begin{align}
\abs{f(\cdot, 0)}_{H^s}^2 &=\abs{f(\cdot, z')}_{H^s}^2+  \int_{\mathbb{R}_y^2} \int_{z'}^0   2 \partial_{z}\Lambda^s f(z,y)\, \Lambda^sf(z,y) dz
dy
\\\nonumber&=\abs{f(\cdot, z')}_{H^s}^2+  \int_{\mathbb{R}_y^2} \int_{z'}^0   2 \partial_{z}\Lambda^{s_1} f(z,y)\, \Lambda^{s_2}f(z,y) dz dy
\end{align}
Integrating the equality above with respect to $z'\in(-b,0)$ and using the Cauchy-Schwarz inequality give the desired estimate.
\end{proof}

Following similar arguments, we also have the following Poincar\'e inequality.
\begin{lem}
\label{proppoin}
It holds that
 \beq
  \label{poin}
\norm{f}\ls \abs{f }_{0} +\norm{\pa_z f}.
 \eeq
  \end{lem}
\begin{proof}
The proof of the estimate \eqref{trace} with $s_1=s_2=s=0$ also leads to
$$
\norm{f}\ls \abs{f }_{0} +\norm{\pa_z f}^\hal\norm{f}^\hal.
$$
Then the Poincar\'e inequality \eqref{poin} follows by Cauchy's inequality.
\end{proof}

We also recall the classical product and commutator estimates in $\mathbb{R}^2$:
\begin{lem}\label{sobbord}
The followings hold.
\begin{eqnarray}\label{sobr2}
& &\abs{\Lambda^s(fg)}_{L^2}\ls\abs{f}_{L^\infty } \abs{\Lambda^s g}_{L^2 }+\abs{g}_{L^\infty } \abs{\Lambda^s f}_{L^2 } \text{ for }s\ge 0, \\
 & &\label{comr2}\abs{ \[\Lambda^s, f \] \nabla_y  g}_{L^2 } \ls  \abs{ \nabla_y f}_{L^\infty } \abs{\Lambda^{s}g}_{L^2 } +  \abs{ \nabla_y
 g}_{L^\infty }\abs{\Lambda^s f}_{L^2 } \text{ for }s\ge 0, \\
& &\label{cont2D}\abs{fg}_{1 \over 2} \ls \abs{ f}_{1, \infty} \abs{g}_{1\over 2}\text{ and }\abs{fg}_{-\hal} \ls \abs{ f}_{1, \infty}
\abs{g}_{-\hal}.
\end{eqnarray}
\end{lem}
\begin{proof}
These estimates \eqref{sobr2} and \eqref{comr2} are classical, see \cite{KP88} for example. Note that
$$\abs{fg}_{0} \ls \abs{ f}_{L^\infty} \abs{g}_{0}\text{ and }\abs{fg}_{1} \ls \abs{ f}_{1, \infty} \abs{g}_{1},$$
the estimate \eqref{cont2D} follows by the theory of interpolation and duality.
\end{proof}

Note that Lemma \ref{sob} also holds on $\mathbb{R}^2$, while we also need the following for half regularities.
\begin{lem}\label{sob2}
For $ |\alpha|+|\beta|=k \geq 0$:
\beq\label{gues2}
\abs{ Z^{\alpha} fZ^{\beta} g }_s\lesssim \abs{f}_{\X^{k,s}}\abs{g}_{\Y^{{\[\!\frac{k}{2}\!\]}+1}}+ \abs{f}_{\Y^{{\[\!\frac{k}{2}\!\]}+1}}
\abs{g}_{\X^{k,s}},\ s=-\hal,\hal  .      \eeq
\end{lem}
\begin{proof}
The estimate \eqref{gues2} follows by using \eqref{cont2D} to control the product with the higher order derivative term in $H^{s}$ and the lower
order derivative term in $W^{1,\infty}$.
\end{proof}

\section{Poisson extension}\label{aa2}

We recall the extension $\eta$ of $h$ onto $\{z\le 0\}$ defined by \eqref{eqeta} with parameter $A>0$ in the following form
\beq \label{zeta}
\eta(y,z)= \( 1+\frac{z}{b}\)\zeta(y,z)\ \text{ with }\ \hat{\zeta}(\xi, z)=  \exp{(A|\xi|z)} \hat{h}(\xi).
\eeq

We  first verify that $\varphi$ defined by \eqref{eqphi} is a diffeomorphism.
\begin{prop} \label{diffeoprop}
Assume that $h_0\in H^{s}(\{z=0\}), s> 5/2$ and $h_0>-b$. Then there exists sufficiently small $A>0$ such that
\beq \label{diffeo}
\partial_{z} \varphi_0\ge \hal\( 1+\frac{1}{b}h_0\)>0\text{ in }\Omega.
\eeq
 \end{prop}
 \begin{proof}
 Note that
 \begin{align}
\abs{\pa_z\hat \zeta (\xi,z)}\le \abs{\pa_z\hat\zeta(\xi,z')}+\int_{z'}^z \abs{\pa_z^2\hat\zeta(\xi,x_3)}dx_3.
\end{align}
Integrating the inequality above with respect to $z'\in(-b,0)$, one can deduce
\begin{align}
 \nonumber b \norm{\pa_z\zeta}_{L^\infty} \leq b \sup_{z\in [-b,0]}\int_{\mathbb{R}^2_\xi}\abs{ \pa_z\hat \zeta   (\xi, z)} \, d\xi\le
 \int_{-b}^0
 \int_{\mathbb{R}^2_\xi} \abs{\pa_z\hat\zeta(\xi,z )}dz+\int_{-b}^0 \int_{\mathbb{R}^2_\xi}\abs{\pa_z^2\hat\zeta(\xi,z)}dz.
\end{align}
For $s>1$, it then follows from the Cauchy-Schwarz inequality and the definition \eqref{zeta} that
\begin{align}
 \nonumber
 \int_{-b}^0 \int_{\mathbb{R}^2_\xi} \abs{\pa_z\hat\zeta(\xi,z )}dz&\ls \(\int_{-b}^0 \int_{\mathbb{R}^2_\xi} (1+|\xi|)^{2 s
 }\abs{\pa_z\hat\zeta(\xi,z )}^2dz\)^\hal
 \\\nonumber& =A \(\int_{\mathbb{R}^2_\xi} (1+|\xi|)^{2s}|\xi|^2 \abs{\hat{h}(\xi)}^2\int_{-b}^0\exp{(2A|\xi|z)}dz\)^\hal
 \\\nonumber& =A \(\int_{\mathbb{R}^2_\xi} (1+|\xi|)^{2s}|\xi|^2 \abs{\hat{h}(\xi)}^2\frac{1-\exp{(2A|\xi|z)}}{2A|\xi|}dz\)^\hal
 \ls  A^\hal  \abs{h}_{s+\hal}
\end{align}
and
\begin{align}
 \nonumber
 \int_{-b}^0 \int_{\mathbb{R}^2_\xi} \abs{\pa_z\hat\zeta(\xi,z )}dz&\ls \(\int_{-b}^0 \int_{\mathbb{R}^2_\xi} (1+|\xi|)^{2 s
 }\abs{\pa_z\hat\zeta(\xi,z )}^2dz\)^\hal
 \\\nonumber& =A^2 \(\int_{\mathbb{R}^2_\xi} (1+|\xi|)^{2s}|\xi|^4 \abs{\hat{h}(\xi)}^2\int_{-b}^0\exp{(2A|\xi|z)}dz\)^\hal
 \ls  A^\frac{3}{2}  \abs{h}_{s+\frac{3}{2}} .
\end{align}
We thus deduce
\beq\label{positivity}
 \norm{\pa_z\zeta}_{L^\infty}\ls A^\frac{1}{2}  \abs{h}_{s+\frac{3}{2}}.
\eeq

Now we prove \eqref{diffeo}. It follows from the definitions \eqref{eqphi} and \eqref{zeta} that
\beq
 \pa_z\varphi_0 = 1+\frac{1}{b}\zeta_0+ \(1+\frac{1}{b}\)\pa_z\zeta_0.
\eeq
By \eqref{positivity}, this yields that for $s>\frac{5}{2}$,
\begin{align}
\nonumber \pa_z\varphi_0 &\ge 1+\frac{1}{b}h_0+\frac{1}{b}(\zeta_0-h_0)+ \(1+\frac{1}{b}\)\pa_z\zeta_0
 \\\nonumber&\ge 1+\frac{1}{b}h_0- \(2+\frac{1}{b}\)\norm{\pa_z\zeta_0}_{L^\infty}\ge 1+\frac{1}{b}h_0- A^\frac{1}{2}\(2+\frac{1}{b}\)
 \abs{h_0}_{ \frac{5}{2}}\ge \hal\( 1+\frac{1}{b}h_0\)
\end{align}
if $A$ has been chosen sufficiently small.
 \end{proof}

 We also have the following well-known estimates for  $\eta$.
\begin{lem}\label{propeta}
For $s\in\mathbb{R}$:
\beq\label{etaharm}
\norm{\eta }_{H^s } \ls \abs{h }_{s-\hal}.
\eeq
For $k \in \mathbb{N}$:
\beq\label{etainfty}
\norm{  \eta }_{W^{k, \infty} } \ls \abs{h }_{k, \infty } .
\eeq
\end{lem}
\begin{proof}
One deduces in the same way as Proposition 3.1 in \cite{MasRou} that
$$
  \norm{ \nabla \zeta }_{H^s } \ls \abs{h }_{s+{1 \over 2} },\ s\in\mathbb{R}\text{ and }\norm{  \eta }_{W^{k, \infty} } \ls \abs{h }_{k, \infty
  },\ k \in \mathbb{N}.
$$
Then the estimates \eqref{etaharm}--\eqref{etainfty} follow by noting that $\eta=(1+\frac{z}{b})\zeta$ for $z\in[-b,0]$.
\end{proof}

\section{Some geometric estimates}\label{aa3}

We recall that the control of quantities like $\int_{\Omega} | \nabla^\varphi  f|^2 d \V$ yields a control of the standard $H^1$ norm of  $f$.
\begin{lem}\label{mingrad}
Assume that  $\partial_{z} \varphi \geq c_{0} $ and   $\norm{ \nabla \varphi}_{L^\infty} \leq {1 \over c_{0}}$ for some $c_{0}>0$, then
\beq\norm{  \nabla f } ^2 \leq  \Lambda_0\int_{\Omega} | \nabla^\varphi f |^2\, d \V.
\eeq
\end{lem}
\begin{proof}
We refer to Lemma 2.8 in \cite{MasRou}.
 \end{proof}

  We also need the Korn type inequality to control the energy dissipation term.
 \begin{lem}\label{Korn}
Assume that $\partial_{z} \varphi \geq  c_{0} $ and $\norm{ \nabla \varphi }_{L^\infty} + \norm{ \nabla^2 \varphi}_{L^\infty} \leq {1 \over
c_{0}}$ for
some $c_{0}>0,$ then
\beq\label{estKorn}
\norm{ \nabla v}^2  \leq \Lambda_0 \( \int_{ \Omega } |S^\varphi v |^2 \, d\V + \norm{ v}^2 \).
\eeq
\end{lem}
\begin{proof}
We refer to Proposition 2.9 in \cite{MasRou}.
 \end{proof}

Finally, we will also need the following $H^{-1/2}$ boundary estimates for functions satisfying $v\in L^2$ and $\nabla^\varphi\cdot v \in L^2$.

\begin{lem}\label{-1/2b}
If $\norm{ \nabla \varphi }_{L^\infty}   \leq {1 \over c_{0}}$ for some $c_{0}>0,$ then
\begin{equation}\label{-1/2es}
 \abs{v \cdot \N}_{-\hal}\le \Lambda_0 \(\norm{v}+\norm{\nabla^\varphi\cdot v}  \).
\end{equation}
\end{lem}
\begin{proof}
We adapt the proof of Lemma 3.3 in \cite{GT13}.
We will only prove the result on $\{z=0\}$. Let $\psi \in H^{1/2} $, and let $\tilde{\psi} \in H^1(\Omega)$ be a bounded extension. Then
\begin{align*}
 \int_{z=0} \psi v \cdot \N &= \int_\Omega \nabla^\varphi\cdot (\tilde{\psi}v) d \V = \int_\Omega \(\nabla^\varphi \tilde{\psi}\cdot
 v+\tilde{\psi} \nabla^\varphi\cdot v   \) d \V \\
 &\le \Lambda_0 \(\norm{\tilde{\psi}}\norm{\nabla^\varphi\cdot v} +\norm{\nabla \tilde{\psi}}\norm{v} \)
 \le \Lambda_0 \abs{ \psi }_{\hal}\(\norm{v}+\norm{\nabla^\varphi\cdot v}  \).
\end{align*}
Then the estimate \eqref{-1/2es} follows from this inequality above by taking the supremum over all $\psi$ so that $\abs{ \psi }_{\hal}\le1$.
\end{proof}



\end{document}